\documentclass[10pt,a4paper]{amsart}
\usepackage[toc,page]{appendix}
\usepackage{a4}
\usepackage[active]{srcltx}
\usepackage{amsmath,bbm,esint}
\usepackage{amssymb}
\usepackage{amsfonts}
\usepackage{mathrsfs} 
\usepackage{graphicx}
\usepackage{tikz}
\usepackage{color}
\usepackage[colorlinks=true,urlcolor=blue,
citecolor=red,linkcolor=blue,linktocpage,pdfpagelabels,bookmarksnumbered,bookmarksopen]{hyperref}

\usepackage{mathtools}
\mathtoolsset{showonlyrefs}

\newcounter{hyp}
\def\a{\alpha}
\def\b{\beta}
\def\d{\delta}
\def\g{\gamma}

\def\vphi{\varphi}

\def\t{\tau}
\def\e{\varepsilon}

\usepackage[normalem]{ulem}

\newcommand{\cA}{{\mathcal A}}
\newcommand{\cB}{{\mathcal B}}

\newcommand{\cE}{{\mathcal E}}

\newcommand{\cF}{{\mathcal F}}
\newcommand{\cG}{{\mathcal G}}

\newcommand{\cK}{{\mathcal K}}

\newcommand{\cT}{{\mathcal T}}
\newcommand{\cP}{{\mathcal P}}

\newcommand{\cU}{{\mathcal U}}

\newcommand{\wt}{{\widetilde w}}

\newcommand{\at}{{\widetilde \a}}

\newcommand{\mt}{{\widetilde m}}

\newcommand{\R}{{\mathbb R}}

\newcommand{\T}{\mathbb{T}}
\newcommand{\N}{\mathbb{N}}
\newcommand{\M}{{\mathbb M}}

\newcommand{\ov}{\overline}

\newcommand{\weakly}{\ensuremath{\rightharpoonup}}
\newcommand{\weaklys}{\stackrel{\star}{\rightharpoonup}}

\newcommand{\sD}{\mathscr{D}}

\newcommand{\sM}{{\mathscr M}}
\newcommand{\sP}{{\mathscr P}}

\newcommand{\one}{\mathbbm{1}}

\newcommand{\ds}{\displaystyle}

\newcommand{\spt}{{\rm{spt}}}

\newcommand{\diver}{{\rm{div}}}

\newcommand{\loc}{{\rm{loc}}}

\newcommand{\dd}{\hspace{0.7pt}{\rm d}}

\DeclareMathOperator*{\esssup}{ess\,sup}

\usepackage{color}

\newcommand{\Blue}{\color{black}}

\newcommand{\be}{\begin{equation}}
\newcommand{\ee}{\end{equation}}

\newcommand{\ba}{\begin{array}}
\newcommand{\ea}{\end{array}}

\newtheorem{theorem}{Theorem}[section]
\newtheorem{definition}[theorem]{Definition}
\newtheorem{lemma}[theorem]{Lemma}
\newtheorem{corollary}[theorem]{Corollary}
\newtheorem{proposition}[theorem]{Proposition}

\newtheorem{problem}[theorem]{Problem}

\newtheorem{example}[theorem]{Example}
\newtheorem{remark}[theorem]{Remark}

\newtheorem{innercustomgeneric}{\customgenericname}
\providecommand{\customgenericname}{}
\newcommand{\newcustomtheorem}[2]{%
  \newenvironment{#1}[1]
  {%
   \renewcommand\customgenericname{#2}%
   \renewcommand\theinnercustomgeneric{##1}%
   \innercustomgeneric
  }
  {\endinnercustomgeneric}
}

\newcustomtheorem{customthm}{Framework}

\newtheorem{assumption}{\textbf{Assumption}}
\numberwithin{equation}{section}

\title[]{A variational approach to first order kinetic Mean Field Games with local couplings}  

\author[M. Griffin-Pickering]{Megan Griffin-Pickering} 
\address{Department of Mathematical Sciences, Durham University, Durham DH1 3LE, United Kingdom}
\email{megan.k.griffin-pickering@durham.ac.uk}

\author[A.R. M\'esz\'aros]{Alp\'ar R. M\'esz\'aros}  
\date{\today}
\address{Department of Mathematical Sciences, Durham University, Durham DH1 3LE, United Kingdom}
\email{alpar.r.meszaros@durham.ac.uk} 

\makeatletter
\@namedef{subjclassname@2020}{%
  \textup{2020} Mathematics Subject Classification}
\makeatother

\begin{document}

\begin{abstract}
First order kinetic mean field games formally describe the Nash equilibria of deterministic differential games where agents control their acceleration, asymptotically in the limit as the number of agents tends to infinity. 
The known results for the well-posedness theory of mean field games with control on the acceleration assume either that the running and final costs are regularizing functionals of the density variable, or the presence of noise, i.e. a second-order system.
In this article we construct global in time weak solutions to a first order mean field games system involving kinetic transport operators, where the costs are local (hence non-regularizing) functions of the density variable with polynomial growth.
We show the uniqueness of these solutions on the support of the agent density.
This is achieved by characterizing solutions through two convex optimisation problems in duality. 
As part of our approach, we develop tools for the analysis of mean field games on a non-compact domain by variational methods. 
We introduce a notion of `reachable set', built from the initial measure, that allows us to work with initial measures with or without compact support.
In this way we are able to obtain crucial estimates on minimizing sequences for merely bounded and continuous initial measures. These are then carefully combined with $L^1$-type averaging lemmas from kinetic theory to obtain pre-compactness for the minimizing sequence. 
Finally,  under stronger convexity and monotonicity assumptions on the data, we prove higher order Sobolev estimates of the solutions.
\end{abstract}

\maketitle

\tableofcontents

\section{Introduction}

The aim of the theory of mean field games (MFG for short) is to characterize limits of Nash equilibria of stochastic or deterministic differential games when the number of agents tends to infinity. Such models were first proposed about 15 years ago, simultaneously by Lasry-Lions (\cite{LasLio06i, LasLio06ii, LasLio07}) and Huang-Malham\'e-Caines (\cite{HuaMalCai}).

This theory turned out to be extremely rich in applications and it provided excellent mathematical questions. Its literature has witnessed a huge increase in the last decade. From the theoretical viewpoint, there are two main approaches to the study of MFG. One is based on analytical and PDE techniques, while the other is a probabilistic approach. The first approach goes back to the original works of Lasry-Lions and has been extended in a great variety of directions in the subsequent years by many authors. If a non-degenerate idiosyncratic noise is present in the models, this typically yields a parabolic structure for the corresponding PDEs and one can expect (strong) classical solutions or a suitable regularity for weak solutions to the corresponding PDE systems, even when the corresponding Lagrangians are local functions of the density variable. For a non-exhaustive list of works in this direction we refer the reader to \cite{Amb:18, Amb:21, CirGof:20, CirGof:21, GomPimSan16, GomPimVos, GomPimSan15, Por15}. The probabilistic approach proved to be equally successful for problems involving Lagrangians that are nonlocal functions of the measure variable. This approach seems to be very powerful for handling different kinds of noises in combination with the non-degenerate idiosyncratic one, such as the common noise. For a non-exhaustive collection of works in this direction we refer to \cite{CarDel:13, CarDelLac, CarDel:vol1, CarDel:vol2}. 

When the model lacks a non-degenerate idiosyncratic noise, this clearly poses technical difficulties in the analysis. Typically, it means that additional structural assumptions need to be imposed on the data to be able to hope for (weak) solutions. Such conditions are, for instance, suitable notions of convexity/monotonicity (cf. \cite{AchCarDelPorSan,MesMou}), or the presence of a suitable variational structure, as in the case of potential games (\cite{Cardaliaguet2013, Cardaliaguet-Graber, CarGraPorTon, GraMes, GraMesSilTon, Gra:14}). In the case of local couplings, it was pointed out by Lions in \cite{Lions-course} that the MFG system (including the planning problem) can be transformed into a degenerate elliptic system in space-time with oblique boundary conditions. Relying on this idea, in a quite general setting, under suitable assumptions on the data (such as strict monotonicity and strong convexity of the Hamiltonians in the measure and momentum variables, respectively; regularity and positivity conditions on the initial data), it has been proven recently in \cite{Mun:1, Mun:2} that the corresponding first order MFG systems have smooth classical solutions.

For an excellent, relatively complete account on the subject and a summary of results to date we refer the reader to the collection \cite{AchCarDelPorSan}.

\medskip

In this manuscript we study a class of first order {\it kinetic} MFG systems, involving Lagrangians that are local functions of the density variable and that possess a variational structure, in the sense of \cite{Cardaliaguet2013, Cardaliaguet-Graber, CarGraPorTon}. 

In our setting, the MFG system can be formally written as 
\begin{equation}\label{eq:main}
\left\{
\ba{ll}
-\partial_t u(t,x,v) - v\cdot D_x u(t,x,v)+H(x,v,D_v u(t,x,v))=f(x,v,m), & {\rm{in}}\ (0,T)\times\M\times\R^{d},\\
\partial_t m(t,x,v) + v\cdot D_x m(t,x,v) - \diver_v(mD_{p}H(x,v,D_v u(t,x,v))) = 0, & {\rm{in}}\ (0,T)\times\M\times\R^{d},\\
m(0,x,v)=m_0(x,v),\ \ u(T,x,v)=g(x,v,m_T), & {\rm{in}}\ \M \times \R^{d}.
\ea
\right.
\end{equation}
Here $\M$ denotes either that $d$-dimensional flat torus $\T^d$ or the whole $d$-dimensional Euclidean space $\R^d$ and is the physical space for the position $x$ of the agents, while the velocity vector $v$ of the agents lies in $\R^d$. $T>0$ is an arbitrary time horizon, $H:\M\times\R^d\times\R^d\to\R$ is the Hamiltonian function, while $f,g:\M\times\R^d\to\R$ stand for the running and final costs of the agents, respectively.

Under suitable assumptions on the data, we obtain the global in time existence, uniqueness and Sobolev regularity of weak solutions to \eqref{eq:main}, relying on two convex optimisation problems in duality.  One of these problems can be seen as an optimal control problem for the Hamilton-Jacobi equation, while its dual is an optimal control problem for the continuity equation (cf. \cite{Cardaliaguet2013, Cardaliaguet-Graber, CarGraPorTon}).

\medskip

\noindent {\bf Review of the literature in connection to our work.}

\medskip

MFG systems of type \eqref{eq:main} have been introduced in the context of models when agents control their acceleration. It seems that such a model can be traced back to the work \cite{NouCaiMal} (in the engineering community), where the authors proposed a MFG model where agents control their acceleration. In the mathematical community, the first works in this framework seem to be the ones \cite{AchManMarTch:20, BarCar, CanMen}. These works consider Hamiltonians (with our notation $H-f$) and final cost functions that are nonlocal regularizing functions in the measure variable. Moreover, the Hamiltonians need to be either purely quadratic or have quadratic growth in the momentum variable. In addition, in \cite{AchManMarTch:20, BarCar} further conditions on the initial measure $m_0$ are also imposed. In \cite{AchManMarTch:20} $m_0$ is taken to be compactly supported and H\"older continuous, while in \cite{BarCar} $m_0$ is taken to be compactly supported. These two works construct weak solutions to the corresponding MFG system in the sense that the Hamilton-Jacobi equation has to be understood in the viscosity sense, while the continuity equation is understood in the sense of distributions. In \cite{CanMen} the initial measure $m_0$ can be quite general and the corresponding Hamiltonian does not need to have the so-called `separable structure' which was assumed in \cite{AchManMarTch:20, BarCar} and is also assumed in this manuscript. These more general hypotheses come at the price of obtaining a weaker notion of solution to the MFG system: the so-called mild solutions. However, the authors show that, under the additional separability assumption on the Hamiltonian, mild solutions become more standard weak solutions in the sense described above.

\smallskip

Several interesting new works are built on the models introduced in \cite{AchManMarTch:20, BarCar, CanMen}. In \cite{CarMen} the authors study the ergodic behaviour of MFG systems, for the case of Hamiltonians that are purely quadratic in the momentum variable and nonlocal regularizing coupling functions $f,g$, with additional growth assumption on $f$ in the $v$ variable. In \cite{AchManMarTch:21} the authors obtain mild solutions to MFG under acceleration control and state constraints, under assumptions similar to the ones in \cite{AchManMarTch:20} on the Hamiltonians, with the possibility to consider Hamiltonians that are power-like functions in the momentum variable. Lastly, in \cite{Men} the author studies a perturbation problem associated to MFG under acceleration control, where the (Lagrangian) cost associated to the acceleration vanishes. 

\smallskip

MFG models with \emph{degenerate} diffusion share some common features with kinetic type problems. In this context we can mention several works. In \cite{DraFel} and \cite{FelGomTad} the authors study time independent MFG systems with purely quadratic Hamiltonians and nonlocal regularizing coupling functions, where the diffusion operator is hypoelliptic or satisfies a suitable H\"ormander condition. It is also worth mentioning that our system \eqref{eq:main} shares some similarities with MFG models where agents interact also through their velocities. In this direction we refer to the works \cite{AchKob:21, GomVos, GraMulPfe, Kob, SanShi}.

\smallskip

Finally, a second order MFG system of type \eqref{eq:main} has been recently studied in \cite{Mimikos-Stamatopoulos}. In this work the author obtains weak and renormalized solutions (in the spirit of \cite{Por15}) to a MFG system that involves a non-degenerate diffusion in the $v$ direction. This seems to be the only work in the context of kinetic type MFG models where the coupling functions $f$ and $g$ are taken to be local functions of the density variable $m$. Here the Hamiltonian $H$ is assumed to depend only on the momentum variable and either to be globally Lipschitz continuous or to have quadratic/sub-quadratic growth. There are several summability properties and moment bounds imposed on the initial density $m_0$. In the case of Lipschitz continuous Hamiltonians, the coupling functions $f,g$ are supposed to fulfill several further assumptions: a strong uniform increasing property in the $m$ variable and their derivatives in the $(x,v)$ variable must have a linear growth condition in the $m$ variable.

In \cite{Mimikos-Stamatopoulos} the presence of the diffusion in the $v$ direction allows the author to use suitable De Giorgi type arguments to show that the solution to the Fokker-Planck equation is bounded and has fractional Sobolev regularity. These estimates seem to be instrumental to set up a fixed point scheme and to show that the MFG system has a weak solution. Furthermore, the presence of this diffusion allows to obtain second order Sobolev estimates for the MFG system.

\medskip

\noindent {\bf Description of our results.}

\medskip

As highlighted above, in this work we are inspired by \cite{Cardaliaguet2013, Cardaliaguet-Graber, CarGraPorTon} and we obtain existence and uniqueness of weak solutions to \eqref{eq:main} (in the sense of Definition \ref{def:notion_solution}) via two convex optimisation problems in duality (Problem \ref{prob:value} and Problem \ref{prob:density}). Compared to these works, several major differences arise which require new ideas. A first obvious difference is that in our setting (in contrast to the compact setting of the flat torus which is considered in the mentioned references) the velocity variable $v$ lives in the non-compact space $\R^d$. This clearly introduces technical issues in the analysis.

To prove our main results, the general outline of our programme is the same as the one of \cite{Cardaliaguet2013, Cardaliaguet-Graber, CarGraPorTon}: prove the duality for Problem \ref{prob:value} and Problem \ref{prob:density}; suitably relax Problem \ref{prob:value} (this will be Problem \ref{prob:relaxed}) and show that the value of this is the same as the original one; show existence of optimizers for the relaxed problem and apply the duality result again to obtain existence of solutions in a suitable weak sense. In this manuscript $H$ is supposed to have a superlinear growth in the momentum variable, and $f$ and $g$ are supposed to have polynomial growth in their last variables. The growth of $f,g$ may be taken independently of the growth of the Hamiltonian (we refer to the next section for the precise assumptions). 

To show that the value of the relaxed problem is the same as the original one, a standard approach used in \cite{Cardaliaguet2013, Cardaliaguet-Graber, CarGraPorTon} is to test the Hamilton-Jacobi inequality of any competitor by competitors of the dual problem (i.e. solutions to the continuity equation). To justify this computation a mollification argument was applied for solutions to the continuity equation. In our case, this mollification alone is not enough because of the non-compact setting. Therefore a delicate cut-off argument has to be also implemented. 

\medskip

The most delicate part, however, is to obtain existence of optimizers to the relaxed problem and in particular to obtain proper compactness results for the minimizing sequences. First, in our case the time trace of the solutions to the Hamilton-Jacobi inequality constraint in Problem \ref{prob:relaxed} is quite weak: $u(t,\cdot)$ has to be understood as a locally finite signed Radon measure. Since in this work $m_0$ may have {\it non-compact support}, it takes additional effort to give a meaning to $\int_{\M\times\R^d}m_0 u_0(\dd x\dd v)$ (a term that appears in the objective functional present in Problem \ref{prob:relaxed}). Our construction, although completely different, has some similarities in spirit with the one in \cite{OrrPorSav}, to define similar time boundary traces.

In order to obtain suitable estimates for the minimizing sequence of the relaxed problem, in \cite{Cardaliaguet2013, Cardaliaguet-Graber, CarGraPorTon} a typical trick was to test the Hamilton-Jacobi inequality constraint by the initial measure $m_0$. For this reason, it was necessary to impose enough regularity, and more importantly a \emph{uniform} positive lower bound of this density everywhere. Because of this, estimates on the quantity $\int_{\T^d}m_0 u_0\dd x$, would readily yield summability estimates on $u_0$ solely. We emphasize that in this manuscript we assume that $m_0$ is merely a bounded and continuous probability density and so we take a completely different route when obtaining such estimates. We introduce the {\it reachable set} $\cU_{m_0}$, a set of points in time, space and velocity that can be reached from $\spt(m_0)$ with arbitrary smooth admissible controls (cf. Definition \ref{def:reachable}). In fact, by the controllability of the underlying ODE system, which satisfies the Kalman rank condition, we have $\cU_{m_0}=\left(\{0\}\times\spt(m_0)\right)\cup\left((0,T)\times\M\times\R^d\right)$. In order to obtain our crucial estimates on the corresponding minimizing sequence we use well chosen test functions that are supported in $\cU_{m_0}$. This construction seems to be new in the literature on variational MFG and we believe that it could be instrumental also in other settings, to possibly relax regularity, positivity or compact support assumptions on $m_0$.

As there is no Hopf-Lax type representation formula available for solutions to our Hamilton-Jacobi equations (which was the case in \cite{Cardaliaguet2013, Cardaliaguet-Graber}), first, we obtain estimates on truncations of the solutions. These are similar in flavour to the corresponding estimates in \cite{CarGraPorTon}, and such ideas date back to \cite{Sta}. As our terminal data typically have merely local summability, this will be the source of additional technical issues (in contrast to \cite{CarGraPorTon}, where the terminal data was taken to be regular enough). 

Let us underline that the ideas and constructions that we have described so far allow us to obtain summability estimates on $u$ and $D_v u$, using the structure of the problem. 
This is not sufficient to yield weak precompactness for minimizing sequences due to the lack of regularity estimates in $x$. To recover the necessary compactness we make use of
 {\it averaging lemmas} available in kinetic theory. 
Averaging lemmas go back to the works \cite{GLPS, GPS} and provide improved regularity and compactness properties for velocity averages of solutions of kinetic transport equations (see Subsection~\ref{sec:compactness} for the precise definitions). For more details and a survey of results we refer the reader to the review \cite{Jabin} and the references cited therein. When regularity with respect to $v$ is additionally available, similar properties can be deduced for the full density function: we refer for instance to \cite{Bouchut2002} for regularity results in the $L^p$ case for $1 < p < +\infty$.
We carefully tailor this approach to our setting, combining our estimates on $D_v u$ with $L^1$ averaging lemmas \cite{GSR, GSR_NS, Han-Kwan} to deduce precompactness for minimizing sequences.
In this way we prove Theorem \ref{thm:existence} on the existence of a minimizer of Problem \ref{prob:relaxed}. This in turn implies Theorem \ref{thm:existence_MFG}, that system \eqref{eq:main} has a (unique) weak solution. As was similarly obtained in \cite{Cardaliaguet2013, Cardaliaguet-Graber, CarGraPorTon}, we show the uniqueness of $m$ and the uniqueness of $u$ on $\{m>0\}$.

A natural question that arises in the context of variational MFG is whether the variational structure and further strong monotonicity and convexity assumptions on the data would yield higher order Sobolev estimates on weak solutions. Such estimates were recently obtained in more classical frameworks in \cite{CarMesSan,GraMes, GraMesSilTon, GraMulPfe, ProSan, San}. In this manuscript we pursue similar Sobolev estimates, implied by taking stronger assumptions on the data. In comparison with the works \cite{GraMes, GraMesSilTon}, in our setting we need to work with a considerably weaker notion of time trace of $u$, which is not stable under perturbations of the initial measure $m_0$. Therefore, our Sobolev estimates remain local in time on $(0,T]$. Another delicate difference is due to the presence of the kinetic transport term. Because of this, a careful choice of  perturbations need to be used, which take into account the kinetic nature of the problem. As a result of this, interestingly, first we obtain estimates on differential operators of the form $(tD_x+D_v)$ applied to $m$ and $D_vu$. For the precise results in this direction we refer to Theorem \ref{thm:reg_v}, Corollary \ref{cor:cor_reg_1} and Corollary \ref{cor:cor_reg_2}.

\smallskip

The structure of the manuscript is as follows.  In Section \ref{sec:2} we state our standing assumptions and main results. In Section~\ref{sec:problems} we present the two variational problems in duality along with the relaxed problem of the primal problem. In Section \ref{sec:3} we have collected some preliminary estimates on weak solutions of the Hamilton-Jacobi inequality obtained on the reachable set $\cU_{m_0}$. In Section \ref{sec:4} we show that the relaxed problem has the same value as the primal problem and hence the duality result holds. Section \ref{sec:5} contains the existence result of a solution to the relaxed problem. Here we rely on the combination of the estimates derived in the previous sections and suitably tailored averaging lemmas from kinetic theory, applied in our context for distributional subsolutions to kinetic Hamilton-Jacobi equations. In Section \ref{sec:6} we show that optimizers of the variational problems in duality provide weak solutions to the MFG system and, conversely, weak solutions are also optimizers of the variational problems. Furthermore, strong convexity yields (partial) uniqueness of these solutions. Section \ref{sec:regularity} is devoted to the derivation of higher order Sobolev estimates for the weak solutions. These require further assumptions on the data. 

We end the paper with two appendix sections. In Appendix \ref{app:time_regularity} we discuss the time regularity of distributional subsolutions to kinetic Hamilton-Jacobi equations which allow us to construct suitable notions of time traces. Finally, in Appendix \ref{app:2} we show that truncations and maxima of distributional subsolutions to kinetic Hamilton-Jacobi equations remain distributional subsolutions to suitably modified equations. 

\section{Standing Assumptions and Main Results }\label{sec:2}

In this section we state our main results on the existence, uniqueness and Sobolev regularity of solutions to the MFG system.

We define $\cF$ and $\cG$ to be the anti-derivatives of the coupling functions $f$ and $g$ with respect to $m$:
\be
\cF(x,v,m) = \int_0^m f(x,v, m' ) \dd m ' \qquad \cG(x,v,m) = \int_0^m g(x,v, m' ) \dd m ' .
\ee
Throughout, we make the following assumptions on the Hamiltonian and coupling functions.

\begin{assumption} \label{hyp:Hamiltonian}
\begin{enumerate}
\item[(H\arabic{hyp})] The Hamiltonian $H(x,v, p)$ is continuous in all variables, and convex and differentiable with respect to $p$. Furthermore, for some $r > 1$, $H$ satisfies bounds of the form
\be \label{hyp:H}
\frac{1}{cr} |p|^r - C_H \leq H(x,v, p) \leq \frac{c}{r} |p|^r + C_H,
\ee
for all $(x,v,p) \in\M \times\R^d \times \R^d$ and some constants $c > 0$ and $C_H \geq 0$. 
Finally, the function $H_0(x,v) : = H(x,v,0)$ has positive part $(H_0)_+ \in C_0(\M\times\R^d)$, where $C_0(\M\times\R^d)$ denotes the closure of the space $C_c(\M\times\R^d)$ with respect to the uniform norm.
\stepcounter{hyp}
\item[(H\arabic{hyp})] $\cF(x,v,m)$ is continuous in all variables and strictly convex and differentiable with respect to $m$ for $m > 0$. Moreover, it satisfies the growth condition
\be \label{hyp:F-growth}
\frac{1}{cq} m^{q} - C_F(x,v) \leq \cF(x,v, m) \leq \frac{c}{q} m^{q} + C_F(x,v), \quad m \geq 0
\ee
where $q > 1$ and the function $C_F \in L^1(\M \times \R^d)$. 
For $m < 0$, we set $\cF(x,v,m) = + \infty$.
\stepcounter{hyp}
\item[(H\arabic{hyp})] $\cG(m)$ is continuous and strictly convex. Moreover, it satisfies the growth condition
\be \label{hyp:G-growth}
\frac{1}{c} m^{s} - C_G(x,v) \leq \cG(x,v, m) \leq {c} m^{s} + C_G(x,v), \quad m \geq 0 ,
\ee
for some $C_G \in L^1(\M \times \R^d)$ and $1 < s \leq q$. 
 For $m < 0$, we set $\cG(x,v,m) = + \infty$.
\stepcounter{hyp}

\item[(H\arabic{hyp})] The initial datum $m_0 \in C_b(\M \times \R^d)$ is a probability density.

\end{enumerate}

\end{assumption}
\stepcounter{hyp}

We note that since $m_0$ is imposed to be a bounded probability density, by interpolation, it is uniformly bounded in $L^\alpha(\M\times\R^d)$, for any $\alpha\in[1,+\infty].$ We emphasize that here we impose growth conditions on $\cF, \cG$ rather than on $f,g$.

\begin{example}
For any $q > 1$ and continuous bounded function $c$ {\color{black}such that $c \geq c_0$ with $c_0 > 0$ a strictly positive constant}, the function
\be
\cF(x,v,m) = \begin{cases} c(x,v) m^q & m \geq 0 \\
+ \infty & m < 0 ,
\end{cases}
\ee
satisfies the given assumptions.
\end{example}

\begin{definition}[Reachable set]\label{def:reachable}
It will be useful to define the set $\cU_{m_0}\subseteq[0,T]\times\M\times\R^d$ to be the set of points potentially reachable by a collection of agents initially distributed according to $m_0$ and evolving according to the control system
\be
\dot x = v, \quad \dot v = a,
\ee
for some control $a \in C([0,T] ; \R^d)$. Observe that the previous control system satisfies the classical Kalman rank condition, and so we have
\be
\cU_{m_0} = \{0\} \times \{ m_0  > 0 \} \cup (0, T] \times \M \times \R^d .
\ee
\end{definition}

Under these standing assumptions, we define the following notion of weak solution to the MFG system.

\begin{definition}\label{def:notion_solution}
We say that $(u,m)$ is a weak solution to \eqref{eq:main}, if the following are fulfilled:

(i) $u\in L^1_{\rm{loc}}(\cU_{m_0})$, $D_v u\in L^r_{\rm{loc}}(\cU_{m_0})$ and $m|D_v u|^r\in L^1((0,T)\times\M\times\R^d)$; 

(ii) $m\in L^q((0,T)\times\M\times\R^d)$ and $m_T\in L^s(\M\times\R^d)$;

(iii) $(u_0)_+\in (L^\infty + L^{q'})(\M \times \R^d)$ and $(u_0)_-$ is a locally finite Radon measure supported in $\{m_0>0\}$.

(iv) $$
\left\{
\begin{array}{l}
-\partial_t u - v\cdot D_x u +H(x,v,D_v u ) \leq f(x,v,m), \, {\rm{in}} \; \sD'((0,T) \times \M \times \R^d) \\  [3pt]
 \ u_T \leq g(\cdot,\cdot,m_T), \ \ {\rm{in\ }}\sD'(\M\times\R^d).
\end{array}
\right.
$$

(v) The continuity equation from \eqref{eq:main} holds in $\sD'((0,T) \times \M \times \R^d)$.

(vi) $\ds\int_{\M\times\R^d}m_0 u_0(\dd x\dd v)$ is finite.

(vii) The following energy equality holds:
\begin{align}\label{eq:energy_equality}
&\quad \int_{\M\times\R^d}m_0u_{0}(\dd x\dd v)-\int_{\M\times\R^d} g(x,v,m_T)m_T\dd x\dd v\\
& =\int_0^T\int_{\M\times\R^d}f(x,v,m) m\dd x\dd v\dd t+ \int_0^T\int_{\M\times\R^d}\left[D_{p_v}H(x,v,D_v u)\cdot  D_v u-H(x,v,D_v u)\right]m\dd x\dd v\dd t .
\end{align}
\end{definition}

\subsection{Existence and Uniqueness}

The first of our main results is the existence and uniqueness of these weak solutions.

\begin{theorem}\label{thm:existence_MFG}
Let Assumption~\ref{hyp:Hamiltonian} hold.
Then there exists a weak solution $(u,m)$ of the mean field game system \eqref{eq:main} in the sense of Definition~\ref{def:notion_solution}. This solution is unique, in the sense that if $(u_1, m_1)$ and $(u_2, m_2)$ are both weak solutions in the sense of Definition~\ref{def:notion_solution}, then $m_1 = m_2$ almost everywhere and $u_1=u_2$ almost everywhere on the set $\{m_1 > 0\}$.
\end{theorem}

\subsection{Regularity}

Our second main result is Sobolev regularity for weak solutions of the mean field games system \eqref{eq:main}.
For this result we assume quadratic growth of the Hamiltonian ($r=2$) and stronger convexity and regularity hypotheses on the data, as follows.

\begin{assumption} \label{hyp:stronger}
\begin{equation}\label{hyp:coupling_strong}
	\tag{H\arabic{hyp}} \text{(Conditions on the coupling functions)
	There exists $C>0$ such that the functions $f,g$ satisfy}
\end{equation}
	\begin{equation}\label{hyp:f_Lipschitz_x}
	|f(x_1,v_1,m) - f(x_2,v_2,m)| \leq C(m^{q-1}+1)(|x_1-x_2|+|v_1-v_2|)\  \ \forall (x_1,v_1),(x_2,v_2) \in \M\times\R^d, \ m \geq 0.
	\end{equation}
	and
	\begin{equation}\label{hyp:g_Lipschitz_x}
	|g(x_1,v_1,m) - g(x_2,v_2,m)| \leq C(m^{s-1}+1)(|x_1-x_2|+|v_1-v_2|)\  \ \forall (x_1,v_1),(x_2,v_2) \in \M\times\R^d, \ m \geq 0.
	\end{equation}
Moreover, there exists $c_f,c_g > 0$ such that
	\begin{equation} \label{hyp:f_strongly_monotone}
	\left( f(x,v,\tilde m) - f(x,v,m)\right)(\tilde m - m) \geq c_f\min\{\tilde m^{q-2},m^{q-2}\}|\tilde m-m|^2 \ \forall \tilde m, m \geq 0, \ \tilde m \neq m.
	\end{equation}
	\begin{equation} \label{hyp:g_strongly_monotone}
	\left( g(x,v,\tilde m) - g(x,v,m)\right)(\tilde m - m) \geq c_g\min\{\tilde m^{s-2},m^{s-2}\}|\tilde m-m|^2 \ \forall \tilde m, m \geq 0, \ \tilde m \neq m.
	\end{equation}
 In the above assumptions, if $q < 2$ or $s<2$ one should interpret $0^{q-2}$ and $0^{s-2}$ as $+\infty$.
	In this way, when $\tilde m = 0$, for instance, \eqref{hyp:f_strongly_monotone} reduces to $f(x,v,m)m \geq c_f m^q$, as in the more regular case $q \geq 2$. Similar comments can be made for \eqref{hyp:g_strongly_monotone}. 
	
	\stepcounter{hyp}
\begin{multline}\label{hyp:H_strong_coercive}
	\tag{H\arabic{hyp}} \text{(Quadratic growth and strong coercivity assumption on }H\text{) Suppose that}\ r=2\ \text{and there exist }\\
	j_1,j_2:\R^d \to \R^d\text{ and }c_H > 0\text{ such that}
\end{multline}
	\begin{equation} \label{eq:Hcoercivity}
	H(x,v,P) + L(x,v,W) - P \cdot W \geq c_H|j_1(P) - j_2(W)|^2.
	\end{equation}
	In particular, and in light of our restriction \eqref{hyp:H}, we assume that $j_1$ and $j_2$ have linear growth.
	\stepcounter{hyp}

\begin{equation}\label{hyp:D_xv^2L}
L(\cdot,\cdot,W)\in C^2(\M\times\R^d)\ {\rm{and}}\ 
	|D_{xx}^2 L(x,v,W)|,|D_{xv}^2 L(x,v,W)|,|D_{vv}^2 L(x,v,W)| \leq C_0|W|^{2} + C_0, \tag{H\arabic{hyp}}
\end{equation}
$\forall (x,v,W)\in \M\times\R^d\times\R^d$.
\stepcounter{hyp}
\end{assumption}

Under these additional assumptions, we prove the following result. The proof is carried out in Section~\ref{sec:regularity}.

\begin{theorem}\label{thm:reg-main}
Suppose that $(u,m)$ is a weak solution to \eqref{eq:main} in the sense of Definition \ref{def:notion_solution} and that \eqref{hyp:coupling_strong}, \eqref{hyp:H_strong_coercive}, \eqref{hyp:D_xv^2L} hold.

Then, there exists $\ov C>0$ such that
$$\|m^{\frac{q}{2} - 1} D_{x,v}m\|_{L^2_\loc((0,T]\times\M\times\R^d)} \leq \ov C,\ \ \ \  \|m^{1/2} D_{x,v} D_v u\|_{L^2_\loc((0,T]\times\M\times\R^d)} \leq \ov C$$
and 
$$
\|m_T^{\frac{s}{2} - 1} D_{x,v} m_T\|_{L^2(\M\times\R^d)} \leq \ov C.
$$
\end{theorem}

\begin{remark}
The estimates appearing in this statement are informal; we in fact obtain uniform $L^2$-type summability of differential quotients (see estimate \eqref{ineq:fundamental} below). The corresponding Sobolev estimates, however, are more delicate to obtain, because these would need to be understood in the sense of weighted Sobolev spaces or more generally in the sense of Sobolev spaces with respect to measures. 
Their precise versions would need to involve tangent spaces with respect to the measure $m$, but these are beyond the scope of the current manuscript. We refer to \cite{BouButSep} on this topic. 
\end{remark}

\section{Variational problems in duality} \label{sec:problems}

We will prove existence of a solution to the MFG system \eqref{eq:main} through a variational characterisation. In this section we set up the variational problems used to obtain solutions.
We recall that here and throughout the rest of the manuscript, we will work under Assumption~\ref{hyp:Hamiltonian}.

\subsection{Optimal control of the Hamilton-Jacobi equation: smooth setting}

We define the Fenchel conjugates of $\cF$ and $\cG$ respectively by
\be
\cF^\ast(x,v, \beta) : = \sup_{m \geq 0} \left\{ \beta m - \cF(x,v, m) \right\} \qquad \cG^\ast(x,v,u) : = \sup_{m \geq 0} \left\{ u m - \cG(x,v,m) \right\} .
\ee

Under our assumptions on $\cF$, we have the bounds
\be \label{eq:Fconj-lb}
\begin{cases}
c |\beta|^{q'} - C_F(x,v)  \leq \cF^\ast(x,v, \beta) \leq c^{-1} |\beta|^{q'} + C_F(x,v) & \beta > 0, \\
- \cF(x,v, 0) = 0 \leq \cF^\ast(x,v, \beta) \leq - \inf_{m \geq 0 } \cF(x,v,m) \leq C_F(x,v) & \beta \leq 0 ,
\end{cases}
\ee
where $q' = q/(q-1)$ denotes the H\"{o}lder conjugate exponent of $q$. Note also that $\cF^\ast$ is non-decreasing. Similar observations hold for $\cG^\ast$.

Using this, we define the following functional: for $u \in C_b^1([0,T] \times \M \times \R^d)$, let
\begin{align}\label{functional:variational_value}
\cA( u) : = \int_0^T\int_{\M\times\R^d}\cF^*(x,v, - \partial_t u &- v \cdot D_x u + H(x,v, D_v u))\dd x\dd v\dd t \\
&- \int_{\M\times\R^d}u(0,x,v)m_0(x,v)\dd x\dd v + \int_{\M\times\R^d}\cG^*(x,v,u(T,x,v)) \dd x \dd v,
\end{align}
whenever the integrals are meaningful, and set $\cA(u)=+\infty$ otherwise.
We define a first variational problem associated to this problem.

{\Blue
\begin{problem} \label{prob:value}
Minimize $\cA(u)$ over $u \in E_0$, where $E_0$ denotes the space
\be \label{def:E0}
E_0 : = \{ u \in C^1_b([0,T] \times \M \times \R^d) : |v| |D_x u| \in L^\infty([0,T] \times \M \times \R^d) \} .
\ee
\end{problem}

\begin{remark}
$E_0$ is a Banach space when equipped with the norm
\be
\| u \|_{E_0} : = \| u \|_{L^\infty} + \| D_v u \|_{L^\infty} + \| (1 + |v|) D_x u \|_{L^\infty}
\ee
\end{remark}

}

\subsection{Optimal control of the continuity equation}

To state the dual problem we define the the Lagrangian $L:\M\times\R^{2d}\to\R$, which is the Fenchel conjugate of the Hamiltonian $H$ in the last variable. In other words, for any $(x,v,\a)\in\M\times\R^{2d}$, we define
\be \label{def:Lagrangian}
L(x,v,\a):=\sup_{p\in\R^d}\left\{ \a\cdot p - H(x,v,p)\right\}.
\ee
Note that $L$ then satisfies upper and lower bounds of the form
\be
\frac{1}{c_L} |\a|^{r'} - C_L \leq L(x,v,\a) \leq c_L |\a|^{r'} + C_L ,
\ee
{\Blue where $r ' = r/(r-1)$ denotes the H\"older conjugate exponent of $r$. }

For pairs $(m,w) \in L^1([0,T]\times\M\times\R^d) \times L^1([0,T]\times\M\times\R^d)$, we define the functional
\begin{align*}
\cB(m, w) : = \int_0^T\int_{\M\times\R^d}\cF(x,v,  m)\dd x\dd v\dd t & + \int_0^T\int_{\M\times\R^d}L \left (x,v,- \frac{w}{m} \right)m\dd x\dd v\dd t \\ 
&+\int_{\M\times\R^d}\cG(x,v,m_T(x,v))\dd x\dd v ,
\end{align*}
with the convention that
\be
L \left (x,v,- \frac{w}{m} \right)m = \begin{cases}
0 & m = w = 0,\\
+ \infty & m=0, w \neq 0 .
\end{cases}
\ee
We then define a second variational problem, (formally) dual to the first.

\begin{problem} \label{prob:density}
{\Blue
Minimize $\cB(m,w)$ over the set $\cK_\cB$ of pairs $(m,w) \in L^1([0,T]\times\M\times\R^d) \times (L^1([0,T]\times\M\times\R^d))^d$ with $m \geq 0$, subject to $(m,w) $ satisfying the following continuity equation:
\be\label{eq:continuity_main}
\partial_t m + v \cdot D_x m + \diver_v w = 0,\ \  {\rm{in}}\ \   \sD'((0,T)\times\M\times\R^d) 
\ee
and $m \vert_{t=0} = m_0$ in the sense of a weak trace. 
}

{\Blue
\begin{remark} \label{rmk:m_trace}
Let us comment on the weak trace of $m$ with respect to the time variable. Since we are interested in competitors $(m,w)$ for which  $\cB(m,w)$ is finite, there must exist a vector-valued measurable function $V \in L^{r'}(m \dd x \dd v \dd t)$, i.e. for which
\be
\int_0^T \int_{\M \times \R^d} |V|^{r'} m \dd x \dd v \dd t < + \infty ,
\ee
such that $w=V \, m$ (i.e. $V$ is the density of $w$ with respect to $m$). So, we notice that the previous equation can be written as 
$$
\partial_t m + \diver_{x}(vm)+ \diver_v (Vm) = 0.
$$
Since $V m = w \in L^1([0,T]\times\M\times\R^d)$, we have $|V| m \in L^1([0,T]\times\M\times\R^d)$.
We are then able to prove that $m$ has a narrowly continuous representative $[0,T]\ni t\mapsto m_t\in\sP(\M\times\R^d)$, so that in particular $m|_{t=0}$ and $m|_{t=T}$ are meaningful.
This is essentially a consequence of \cite[Lemma 8.1.2]{AmbGigSav}, with minor modifications to account for the fact that $v m$ is only \emph{locally integrable}; we sketch this in the appendix in Lemma~\ref{app:lem:m_trace}.
\end{remark}
}

\end{problem}

\subsection{Duality}

\begin{lemma} \label{lem:duality1}
We have the following duality:
\be \label{eq:duality}
\inf_{u \in C^1_b} \cA(u) = - \min_{(m,w) \in \cK_B} \cB(m,w) .
\ee
\end{lemma}
\begin{proof}
This is an application of the classical Fenchel-Rockafeller duality theorem.
{\Blue
Recall that we defined the Banach space $E_0$ above in \eqref{def:E0}.
Then let $E_1$ be defined by
\be 
E_1 : = C^0_b([0,T] \times \M \times \R^d ; \R) \times C^0_b([0,T] \times \M \times \R^d ; \R^d) ;
\ee
we will express elements of $E_1$ as pairs $(\phi, \psi)$ of continuous bounded functions, where $\phi$ is real-valued and $\psi$ is vector-valued. $E_1$ is a Banach space with respect to the uniform norm.
}
On these spaces we define the respective functionals
\be
\cA_0( u) : = - \int_{\M\times\R^d}u(0,x,v)m_0(x,v)\dd x\dd v + \int_{\M\times\R^d}\cG^*(x,v,u(T,x,v)) \dd x \dd v 
\ee
and
\be
\cA_1( \phi, \psi) : = \int_0^T\int_{\M\times\R^d}\cF^*(x,v, -\phi + H(x,v, \psi))\dd x\dd v\dd t .
\ee
{\Blue Note that these functionals are convex.}
We also define the bounded linear map $\Lambda: E_0 \to E_1$ by
\be
\Lambda u : = (\partial_t u + v \cdot D_x u, D_v u) .
\ee
Then
\be
\cA(u) = \cA_0(u) + \cA_1(\Lambda u) .
\ee

We wish to apply Fenchel-Rockafeller duality. In order to do this we must verify the existence of $u \in E_0$ such that $\cA_0(u), \cA_1(\Lambda u) < + \infty $ and $\cA_1$ is continuous at $\Lambda u$.
For example, we may take $u$ to be of the form
{\Blue
\be
u(t,x,v) := \zeta(x-vt, v) + 2 C_H (t-T),
\ee}
where $C_H$ denotes the constant from the bounds on the Hamiltonian \eqref{hyp:H}. We then take $\zeta \in C^1_b(\M \times \R^d)$ non-negative to have sufficiently strong decay at infinity so that
\be \label{eq:zeta-decay}
\zeta \in L^{s'}(\M \times \R^d), \quad D_{x,v} \zeta \in L^{r q'}(\M \times \R^d).
\ee
Explicitly, for the case $\M = \T^d$ we may take for example
\be
\zeta(x,v) = \zeta(v) = (1 + |v|^2)^{-k/2} ; \qquad k > \max \left\{ \frac{d}{s '}, \frac{d}{r q'} \right\} ,
\ee
{\Blue
in which case $|v| |D_x u| = 0$ and therefore $u \in E_0$. \\
}
For the case $\M = \R^d$ we may take
\be
\zeta(x,v) = \zeta(v) = (1 + |x|^2 + |v|^2)^{-k/2} ; \qquad k > \max \left\{ \frac{2d}{s '}, \frac{2d}{r q'} .\right\} .
\ee
{\Blue
In this case,
\be
D_x u = D_x \zeta(x-vt, v) = \frac{- k (x - vt) }{(1 + |x- vt|^2 +  |v|^2)^{1 + k/2}},
\ee
and so
\be
|v||D_x u| = \frac{ k |v| |x - vt| }{(1 + |x- vt|^2 +  |v|^2)^{1 + k/2}} \leq \frac{k}{2} (1 + |x-vt|^2 + |v|^2)^{-k/2} \leq \frac{k}{2},
\ee
which implies that $u \in {\color{black}E_0}$.
}

{\Blue
Then, in either case,
\be
- \partial_t u - v \cdot D_x u + H(x,v, D_v u) = {\Blue - 2 C_H} + H(x,v, D_v u) \leq c |D_v u |^r - C_H \ee
It follows that the positive part satisfies
\be
[- \partial_t u - v \cdot D_x u + H(x,v, D_v u) ]_+ \leq c |D_v u |^r \one_{\{c |D_v u |^r > C_H\}} \in L^{q'} ([0,T] \times \M \times \R^d).
\ee
and thus by the bounds on $\cF^\ast$ \eqref{eq:Fconj-lb} we obtain
\be
\int_0^T\int_{\M\times\R^d}\cF^*(x,v, - \partial_t u - v \cdot D_x u + H(x,v, D_v u))\dd x\dd v\dd t < + \infty .
\ee 
That is, $\cA_1(\Lambda u)$ is finite.}

Moreover, $u_T \in L^{s'}(\M \times \R^d)$ and thus 
\be
\int_{\M\times\R^d}\cG^*(x,v,u(T,x,v)) \dd x \dd v  < + \infty .
\ee
Finally, since $u_0 = \zeta {\Blue - 2 C_H T}$ and $m_0$ is a probability density,
\be
- \int_{\M\times\R^d}u(0,x,v)m_0(x,v)\dd x\dd v < + \infty .
\ee
{\Blue Thus $\cA_0(u)$ is finite.}

{\Blue Now we verify that $\cA_1$ is continuous at $\Lambda u$ with respect to convergence in $E_1$. Consider the sequence of pairs $(\phi_n, \psi_n) \in E_1$, $n\in\N$, such that
\be
\phi_n = \partial_t u + v \cdot D_x u + \delta_n, \quad \psi_n = D_v u + \e_n,
\ee
where $(\delta_n, \e_n) \in E_1$ satisfy $\| (\delta_n, \e_n)\|_{L^\infty} \leq 2^{-n}$. Then
\be
- \phi_n + H(x,v, \psi_n) = - 2 C_H - \delta_n + H(x,v, D_v u + \e_n) .
\ee
Using the bounds \eqref{hyp:H} on the Hamiltonian, we obtain
\begin{align}
- \phi_n + H(x,v, \psi_n) &\leq - 2 C_H - \delta_n + C_H + \frac{c}{r}|D_v u + \e_n|^r \\
& \leq - C_H - \delta_n + 2^{r-1} \frac{c}{r} |D_v u|^r + 2^{r-1} \frac{c}{r} |\e_n|^r \\
& \leq - C_H + C (2^{-n} + 2^{-r n}) + C |D_v u|^r,
\end{align}
{\color{black}for some constant $C>0$.} Therefore, for all $n$ large enough that $C (2^{-n} + 2^{-r n}) < C_H$, for the positive part we have
\be
[- \phi_n + H(x,v, \psi_n)]_+ \leq C |D_v u|^r .
\ee
Then the bounds \eqref{eq:Fconj-lb} on $\cF^\ast$ imply that
\be
\cF^\ast(- \phi_n + H(x,v, \psi_n)) \leq 2 C_F + C |D_v u|^{r q'} \in L^1([0,T] \times \M \times \R^d),
\ee
(where the constant $C>0$ has changed line to line). The right hand side is in $L^1$ because we constructed $\zeta$ to satisfy \eqref{eq:zeta-decay}. We may therefore use it as a dominating function: since $(\delta_n, \e_n)$ certainly converges to zero pointwise (in fact in uniform norm), and $\cF^\ast$ is continuous with respect to the variable $\beta$, by dominated convergence we may conclude that
\begin{align}
\cA_1[(\phi_n, \psi_n)] & = 
\lim_{n\to+\infty} \int_0^T\int_{\M\times\R^d}\cF^*(x,v, -\phi_n + H(x,v, \psi_n))\dd x\dd v\dd t \\
& = \int_0^T\int_{\M\times\R^d}\cF^*(x,v, - \partial_t u - v \cdot D_x u + H(x,v, D_v u))\dd x\dd v\dd t \\
& = \cA_1(\Lambda u) .
\end{align}
Thus $\cA_1$ is indeed continuous at $\Lambda u$.

}

It remains to check that $\cA$ is bounded below on $E_0$. Let $u\in E_0$ and set $\beta:=- \partial_t u - v \cdot D_x u + H(x,v, D_v u)$. Then, using the growth assumptions on $\cF^*$ and $\cG^*$, 
similarly to the inequality \eqref{estim:beta} below, we have
\begin{align}
\nonumber\cA(u)&\ge \|\beta_+\|_{L^{q'}}^{q'}+ \|u(T,\cdot,\cdot)_+\|_{L^{s'}}^{s'}-\int_{\M\times\R^d} [TC_F(x,v)+C_G(x,v)]\dd x\dd v \\
\nonumber&-\int_{\M\times\R^d} u(0,x,v)_+m_0(x,v)\dd x\dd v\\
&\ge \|\beta_+\|_{L^{q'}}^{q'}+ \|u(T,\cdot,\cdot)_+\|_{L^{s'}}^{s'} -C
-\left ( \| (u_T )_+ \|_{L^{s'}} + C_H T + T^{1/q} \| \beta_+ \|_{L^{q'}} \right ) {\Blue ( \| m_0 \|_{L^1} + \| m_0 \|_{ L^q} )} \\
&\ge \inf_{a,b\ge 0}\left\{a^{q'}-c_0 a + b^{s'}-c_0 b-c_0\right\}>-\infty,
\end{align}
where $c_0$ was set to be a large positive constant depending only on $m_0,T, C_H, C_F, C_G$.

Therefore, we are in position to apply the Fenchel-Rockafeller duality theorem (cf. \cite[Chapter 3, Theorem 4.1]{Ekeland-Temam}), to conclude
\be
\inf_{u \in E_0} \cA(u) = \max_{(m,w) \in E_1'} \{ - \cA^\ast_0(\Lambda^\ast(m,w)) - \cA^\ast_1(-(m,w)) \}.
\ee
{\Blue
Here $E_1'$ denotes the dual space of $E_1$.
By \cite[IV.6]{Dunford-Schwartz} the dual space of $C^0_b$ may be identified with the space of bounded, regular, finitely additive set functions. Thus $E_1 '$ is the space of pairs $(m,w)$, where $m$ is a real-valued regular finitely additive set function, and $w$ is a $\R^d$-valued regular finitely additive set function.
}

It remains to identify
\be
 \max_{(m,w) \in E_1'} \left\{ - \cA^\ast_0(\Lambda^\ast(m,w)) - \cA^\ast_1(-(m,w)) \right\} .
 \ee
In what follows, we are going to show that the above maximisation problem actually admits solutions in a better space than $E_1'$. So, we have 
\be
 \max_{(m,w) \in E_1'} \left\{ - \cA^\ast_0(\Lambda^\ast(m,w)) - \cA^\ast_1(-(m,w)) \right\} =  \max_{(m,w) \in \tilde E_1'} \left\{ - \cA^\ast_0(\Lambda^\ast(m,w)) - \cA^\ast_1(-(m,w)) \right\},
 \ee
where the set $\tilde E_1'$ stands for pairs $(m,w)$ such that {\Blue $m$ is a finite Radon measure on $[0,T] \times \M \times \R^d$}
and $w$ is a finite vector-valued Radon measure on $[0,T] \times \M \times \R^d$ taking values in $\R^d$. The proof of this is postponed to Lemma \ref{lem:max_equal} below.

Then, by arguing as in \cite[Section 3.3]{Cardaliaguet2013}, we may identify that 
\be
 \max_{(m,w) \in E_1'} \{ - \cA^\ast_0(\Lambda^\ast(m,w)) - \cA^\ast_1(-(m,w)) \} =  \max_{(m,w) \in E_1'} - \cB(m,w)
 \ee
where the maximum is taken over $(m,w) \in E_1'$ such that $(m,w)\in L^1([0,T]\times\M\times\R^d)\times L^1([0,T]\times\M\times\R^d;\R^d)$ and $m \geq 0$ almost everywhere, such that
\be
\partial_t m + v \cdot D_x m + \diver_v w = 0 \quad \text{in }  \sD'((0,T)\times\M\times\R^d) , \qquad m \vert_{t=0} = m_0 .
\ee
Thus
\be 
\inf_{u \in {\color{black}E_0}} \cA(u) = - \min_{(m,w) \in \cK_B} \cB(m,w) .
\ee
\end{proof}

\begin{lemma}\label{lem:max_equal}
Using the notations and assumptions from Lemma \ref{lem:duality1}, we have 
\be
 \max_{(m,w) \in E_1'} \left\{ - \cA^\ast_0(\Lambda^\ast(m,w)) - \cA^\ast_1(-(m,w)) \right\} =  \max_{(m,w) \in \tilde E_1'} \left\{ - \cA^\ast_0(\Lambda^\ast(m,w)) - \cA^\ast_1(-(m,w)) \right\},
 \ee
\end{lemma}

\begin{proof}
Observe that any pair $(m,w) \in E_1'$ induces functionals on $C^0_c$ and $(C^0_c)^d$. Therefore, there exist a signed Radon measure $\tilde m$ with finite total variation and a finite vector-valued measure $\tilde w$ which coincide with, respectively, $m$ and $w$ on (the closure {\Blue with respect to the uniform norm} of) $C^0_c$ and $(C^0_c)^d$.
Then
\be
 \cA^\ast_1(-(m,w)) = \sup_{(\phi, \psi) \in E_1} \left\{ \langle - m, \phi \rangle + \langle -w, \psi \rangle - \int_0^t \int_{\M \times \R^d} \cF^\ast (x,v, -\phi + H(x,v, \psi)) \dd x \dd v \right\}
\ee

By considering functions of the form $\phi = l \chi + H_0$, for $H_0(x,v) : = H(x,v,0)$ (note that our assumptions on $H$ imply in particular that $H_0\in C_b$) and any non-negative $\chi \in C^0_b$ and $l > 0$, and $\psi = 0$, we find that $\cA^\ast_1(-(m,w)) = + \infty$ unless $m$ is a positive functional. Indeed, note that
\be
\int_0^t \int_{\M \times \R^d} \cF^\ast (x,v, -l \chi ) \dd x \dd v \leq \int_0^t \int_{\M \times \R^d} \sup_{\beta \leq 0} \cF^\ast (x,v, \beta) \dd x \dd v < + \infty ,
\ee
and $\sup_{l>0} \langle - m, l \chi \rangle = + \infty$ if $\langle m , \chi \rangle < 0$.

Next, by taking the supremum over the smaller set $(\phi, \psi) \in C^0_c \times (C^0_c)^d$ we have
\begin{align}
 \cA^\ast_1(-(m,w)) & \geq \sup_{(\phi, \psi) \in C^0_c \times (C^0_c)^d} \{ \langle - m, \phi + H_0 \rangle + \langle -w, \psi \rangle - \int_0^t \int_{\M \times \R^d} \cF^\ast (x,v, -\phi - H_0 + H(x,v, \psi)) \dd x \dd v \dd t \} \\
& =  \sup_{(\phi, \psi) \in C^0_c \times (C^0_c)^d} \{ \langle - \tilde m, \phi + H_0 \rangle + \langle - \tilde w, \psi \rangle - \int_0^t \int_{\M \times \R^d} \cF^\ast (x,v, -\phi - H_0 + H(x,v, \psi)) \dd x \dd v \dd t \} \\
& \qquad \qquad - \langle m - \tilde m, H_0 \rangle.
\end{align}
Let us underline that the assumption on $H_0$ plays a crucial role, otherwise the integral of $\cF^\ast$ might not be finite for compactly supported test functions.

Since $H$ is convex, for any $\chi_R \in C^0_b$ such that $0 \leq \chi_R \leq 1$,
\be
H(x,v, \chi_r \psi) \leq \chi_R H(x,v, \psi) + (1- \chi_R) H_0(x,v).
\ee
Thus
\be
- \phi \chi_R - H_0 + H(x,v, \chi_R \psi) \leq \chi_R \left (- \phi - H_0 + H(x,v,  \psi) \right ),
\ee
and in particular we can compare the positive parts:
\be
(- \phi \chi_R - H_0 + H(x,v, \chi_R \psi))_+ \leq (\chi_R \left (- \phi - H_0 + H(x,v,  \psi) \right ) )_+.
\ee
Since $\cF^\ast$ is non-decreasing,
\be
\cF^\ast \left (x,v, - \phi \chi_R - H_0 + H(x,v, \chi_R \psi) \right ) \leq \sup_{\beta < 0} \cF^\ast(x,v, \beta) + \cF^\ast \left (x,v, - \phi - H_0 + H(x,v, \psi) \right )  \in L^1 .
\ee

Hence, for all $\phi \in C^0_b$, $\psi \in (C^0_b)^d$ such that
\be
 \int_0^T \int_{\M \times \R^d} \cF^\ast (x,v, -\phi - H_0 + H(x,v, \psi)) \dd x \dd v \dd t < + \infty,
\ee
by dominated convergence we have
\begin{align*}
 \int_0^t \int_{\M \times \R^d} \cF^\ast (x,v, -\phi - H_0 &+ H(x,v, \psi)) \dd x \dd v \dd t\\ 
 &= \lim_{R \to +\infty}  \int_0^t \int_{\M \times \R^d} \cF^\ast (x,v, -\phi_R - H_0 + H(x,v, \psi_R)) \dd x \dd v \dd t ,
 \end{align*}
 where $\phi_R = \phi \chi_R$, $\psi_R = \psi \chi_R$ for some continuous $0 \leq \chi_R \leq 1$ converging pointwise to the constant function $1$ as $R$ tends to positive infinity.
We conclude that, for any $\tilde m$, $\tilde w$ (respectively signed, vector-valued) Radon measures with finite total variation,
\begin{multline}
\sup_{(\phi, \psi) \in C^0_c \times (C^0_c)^d} \{ \langle - \tilde m, \phi + H_0 \rangle + \langle - \tilde w, \psi \rangle - \int_0^t \int_{\M \times \R^d} \cF^\ast (x,v, -\phi - H_0 + H(x,v, \psi)) \dd x \dd v \dd t \} \\
= \sup_{(\phi, \psi) \in C^0_b \times (C^0_b)^d} \{ \langle - \tilde m, \phi \rangle + \langle - \tilde w, \psi \rangle - \int_0^t \int_{\M \times \R^d} \cF^\ast (x,v, -\phi + H(x,v, \psi)) \dd x \dd v \dd t \},
\end{multline}
where we have used that $H_0$ is also a $C^0_b$ function in order to relabel $\phi$. We have thus proved that
\be
 \cA^\ast_1(-(m,w)) \geq  \cA^\ast_1(-(\tilde m, \tilde w)) - \langle m - \tilde m , H_0 \rangle.
\ee

Next, note that if $m \in (C^0_b)'$ is a positive functional with Radon measure part $\tilde m$, then $m - \tilde m$ is also a positive functional: given $0 \leq \phi \in C^0_b$, let $0 \leq \chi_R \leq 1$ be a sequence of continuous functions, non-decreasing with $R$ and converging pointwise to the constant function 1 as $R$ tends to positive infinity. Then, since $0 \leq \phi \chi_R \leq \phi$, by dominated convergence and the positivity of $m$,
\be
\langle \tilde m , \phi \rangle = \lim_{R \to 0} \langle \tilde m , \phi \chi_R \rangle = \lim_{R \to 0} \langle m , \phi \chi_R \rangle \leq \langle m, \phi \rangle.
\ee
Since $(H_0)_+ \in C_0$, $\langle m - \tilde m , (H_0)_+ \rangle = 0 $ and thus $\langle m - \tilde m , H_0 \rangle \leq 0 $ for all $m$ such that $\cA^\ast_1(-(m,w,))$ is finite. Then
\be
- \cA^\ast_1(-(m,w)) \leq - \cA^\ast_1(-(\tilde m, \tilde w)) .
\ee

We now consider $\cA^\ast_0$. We assume from now on that $m \in (C^0_b)'$ is a \emph{positive} functional, since we only wish to consider $(m,w)$ for which $\cA^\ast_1(-(m,w)) < + \infty$.
\begin{align}
\cA^\ast_0(\Lambda^\ast(m,w)) & = \sup_{u \in {\Blue E_0}} \left \{ \langle \Lambda^\ast(m,w), u \rangle - \int_{\M\times\R^d}u(0,x,v)m_0(x,v)\dd x\dd v - \int_{\M\times\R^d}\cG^*(x,v,u(T,x,v)) \dd x \dd v  \right \} \\
& = \sup_{u \in {\Blue E_0}} \left \{ \langle (m,w), \Lambda u \rangle - \int_{\M\times\R^d}u(0,x,v)m_0(x,v)\dd x\dd v - \int_{\M\times\R^d}\cG^*(x,v,u(T,x,v)) \dd x \dd v  \right \} .
\end{align}
Then, taking supremum over the smaller set $u \in C^1_c$, we have
\be
\cA^\ast_0(\Lambda^\ast(m,w)) \geq \sup_{u \in C^1_c} \left \{ \langle (m,w), \Lambda u \rangle - \int_{\M\times\R^d}u(0,x,v)m_0(x,v)\dd x\dd v - \int_{\M\times\R^d}\cG^*(x,v,u(T,x,v)) \dd x \dd v  \right \} .
\ee
If $u \in C^1_c$, then $\Lambda u \in C^0_c$. Thus
\be
\cA^\ast_0(\Lambda^\ast(m,w)) \geq \sup_{u \in C^1_c} \left \{ \langle (\tilde m, \tilde w), \Lambda u \rangle - \int_{\M\times\R^d}u(0,x,v)m_0(x,v)\dd x\dd v - \int_{\M\times\R^d}\cG^*(x,v,u(T,x,v)) \dd x \dd v  \right \} .
\ee
We show that the right hand side is in fact equal to $\cA^\ast_0(\Lambda^\ast(\tilde m , \tilde w))$: given $u \in {\Blue E_0}$, let $\chi_R \in C^1_c(\M \times \R^d)$ be a sequence of cut-off functions such that $0 \leq \chi_R \leq 1$. {\Blue We construct $\chi_R$ such that their support is contained in $\overline{B}_{2R}(0) \subset \M \times \R^d$, the closed ball of radius $2R$, $\chi_R = 1$ on $\overline{B}_{R}(0)$, the closed ball of radius $R$, and $\| \nabla \chi_R \| \leq \frac{C}{R}$ for some constant $C > 0$ independent of $R$. Thus note in particular that $\chi_R \to 1$ and $\nabla \chi_R \to 0$ pointwise as $R \to + \infty$. Let $u_R : = u \chi_R$.} Then
\be
|\Lambda u_R | = |\partial_t u_R + v \cdot D_x u_R| \leq C \| u \|_{E_0}, \quad |u_R(0,x,v)| \leq |u(0,x,v)|, \quad [u_R(T,x,v)]_+ \leq [u(T,x,v)]_+.
\ee
Since $m\in (C^0_b)'$, we have $ \langle \tilde m , C \| u \|_{E_0} \rangle =  C \| u \|_{E_0} \langle \tilde m , 1 \rangle < + \infty$. Moreover $u(0, \cdot)$ is bounded and therefore integrable with respect to $m_0$. Finally, note that 
\be
\cG^*(x,v,u_R(T,x,v)) \leq \cG^*(x,v,u(T,x,v)) + \sup_{\beta_T < 0 } \cG^\ast(x,v, \beta_T).
\ee
Hence, if
\be
 \int_{\M\times\R^d}\cG^*(x,v,u(T,x,v)) \dd x \dd v < + \infty,
\ee
we may apply the dominated convergence theorem to find that
\begin{multline}
\langle (\tilde m, \tilde w), \Lambda u \rangle - \int_{\M\times\R^d}u(0,x,v)m_0(x,v)\dd x\dd v - \int_{\M\times\R^d}\cG^*(x,v,u(T,x,v)) \dd x \dd v \\
= \lim_{R \to +\infty} \langle (\tilde m, \tilde w), \Lambda u_R \rangle - \int_{\M\times\R^d}u_R(0,x,v)m_0(x,v)\dd x\dd v - \int_{\M\times\R^d}\cG^*(x,v,u_R(T,x,v)) \dd x \dd v .
\end{multline}
This completes the proof that the suprema over {\Blue $E_0$} and $C^1_c$ are equal for the Radon measure parts.
We conclude that
\be
- \cA^\ast_0(\Lambda^\ast(m,w)) \leq - \cA^\ast_0(\Lambda^\ast(\tilde m , \tilde w)) .
\ee

Now observe that, since the set $\tilde E_1'$ is contained in $E_1'$ (it is precisely the set of Radon measure parts of elements of $E_1'$),
\begin{align}
 \max_{(m,w) \in E_1'} \{ - \cA^\ast_0(\Lambda^\ast(m,w)) - \cA^\ast_1(-(m,w)) \} & \leq  \sup_{(m,w) \in E_1'} \{ - \cA^\ast_0(\Lambda^\ast(\tilde m, \tilde w)) - \cA^\ast_1(-(\tilde m, \tilde w)) \} \\
 & \leq  \sup_{(\tilde m, \tilde w) \in \tilde E_1 '} \{ - \cA^\ast_0(\Lambda^\ast(\tilde m, \tilde w)) - \cA^\ast_1(-(\tilde m, \tilde w)) \} \\
 & \leq  \sup_{(m,w) \in E_1'} \{ - \cA^\ast_0(\Lambda^\ast(m,w)) - \cA^\ast_1(-(m,w)) \} .
 \end{align}
 
 All of the above inequalities are therefore equalities. Moreover, since 
 \be
  - \cA^\ast_0(\Lambda^\ast(m,w)) - \cA^\ast_1(-(m,w)) \leq  - \cA^\ast_0(\Lambda^\ast(\tilde m, \tilde w)) - \cA^\ast_1(-(\tilde m,\tilde w)),
 \ee
 if $(m,w)$ attains the supremum then the same is true of the Radon measure part $(\tilde m, \tilde w)$. Thus, without loss of generality, the optimizer is given by some $(m,w) \in \tilde E_1'$, i.e. {\Blue a finite measure and a finite $\R^d$-valued measure.}

\end{proof}

\begin{remark}\label{rem:min_B}
Let us notice that the the minimizer of $\cB(m,w)$ is unique (by the convexity of $\cF,\cG$ and $L$ in their last variables). Moreover, the growth conditions on $\cF,\cG$ and $L$ imply that $m\in L^q((0,T)\times\M\times\R^d)$, $m_T\in L^s(\M\times\R^d)$ and $\frac{|w|^{r'}}{m^{r'-1}}\in L^1((0,T)\times\M\times\R^d)$. Furthermore, by H\"older's inequality, $w\in L^p((0,T)\times\M\times\R^d))$, with $p:=\frac{r'q}{r'+q-1}.$ These arguments are similar to the ones in \cite[Theorem 2.1]{Cardaliaguet-Graber} and \cite[Lemma 2]{Cardaliaguet2013}.
Furthermore, the equation satisfied by $m$ conserves mass, so that $m \in L^\infty_tL^1_{x,v,}$, and in fact $\| m_t\|_{L^1_{x,v}} = \|m_0\|_{L^1_{x,v}}$ for all $t \in [0,T]$.
\end{remark}

\subsection{The relaxed problem}

The third problem we define is a relaxation of Problem~\ref{prob:value}. Consider the functional
\begin{align}\label{functional:variational_relaxed}
\widetilde{\cA}(u, \beta,\beta_T) : = \int_0^T\int_{\M\times\R^d}\cF^*(x,v, \beta)\dd x\dd v\dd t &- \int_{\M\times\R^d}m_0(x,v)u_0(\dd x\dd v)  \\ 
&+ \int_{\M\times\R^d}\cG^*(\beta_T) \dd x \dd v .
\end{align}

\begin{problem}\label{prob:relaxed}
Minimize $\widetilde{\cA}(u,\beta,\beta_T)$
over the set $\cK_{\cA}$ of triples $(u, \beta,\beta_T) \in L^1_{\rm{loc}}(\cU_{m_0}) \times L^1_{\rm{loc}}(\cU_{m_0})\times L^1(\M\times\R^d)$ satisfying
\begin{itemize}
\item The positive part of $u$ satisfies $u_+ \in L^1_\loc([0,T] \times \M \times \R^d)$;
\item The positive part of $\beta$ satisfies $\beta_+ \in L^{q'}([0,T]\times \M \times \R^d)$.
\item The positive part of $\beta_T$ satisfies $(\beta_T)_+\in L^{s'}(\M\times\R^d)$; 
\item $D_v u \in L^r_{\loc}(\cU_{m_0})$,
\end{itemize}
and subject to \eqref{eq:weak_HJ}, understood in the sense of Definition \ref{def:solution}.
\end{problem}

\begin{definition}\label{def:solution}
We say that a triple $(u,\b,\b_T)$ that belongs to the spaces from Problem \ref{prob:relaxed} is a weak distributional solution to
\begin{equation}\label{eq:weak_HJ}
\left\{
\begin{array}{rcll}
-\partial_t u - v\cdot D_x u+H(x,v,D_v u) &\leq& \beta, & {\rm{in}}\ (0,T)\times\M\times\R^d,\\
u_T&\le& \beta_T, & {\rm{in\ }} \M\times\R^d,
\end{array}
\right.
\end{equation}
if 
\begin{equation}\label{eq:weak_HJ_test}
\int_0^T\int_{\M\times\R^d}u[\partial_t\phi+\diver(v\phi)] + \phi H(x,v,D_vu)\dd x\dd v\dd t\le \int_0^T\int_{\M\times\R^d}\beta\phi\dd x\dd v\dd t+\int_{\M\times\R^d}\beta_T\phi_T \dd x\dd v,
\end{equation}
for any $\phi\in C^1_c((0,T]\times\M\times\R^d)$ nonnegative. 
\end{definition}

\begin{remark}
(i) Let us emphasize that the weak form \eqref{eq:weak_HJ_test} encodes both inequalities from \eqref{eq:weak_HJ}, as we show this in Lemma \ref{lem:app_boundary_condition}.

(ii) $u_0$ is similarly understood as a certain notion of a trace at $t=0$ in a weak sense. In particular, the term
\be
- \int_{\M\times\R^d} m_0(x,v) u_0(\dd x\dd v) = -\langle u_0,m_0\rangle,
\ee
which appears in the definition of $\widetilde\cA$ is to be understood as in Definition~\ref{def:u0m0}. Moreover, we underline that this quantity is set to be $+\infty$, if there exist $\phi\in C^1_c(\{m_0>0\})$ nonnegative such that $\langle u_0,\phi\rangle=-\infty$.
\end{remark}

\section{The Hamilton-Jacobi Equation}\label{sec:3}

In this section, we analyse the equation \eqref{eq:weak_HJ}.
We take the assumptions appropriate to the minimisation problem we will consider. Therefore, we suppose throughout that $(u,\beta,\beta_T)\in \cK_A$ is such that $\widetilde \cA(u, \beta,\beta_T) < + \infty$. 
From the finiteness of the energy we deduce in particular that
\be
- \int_{\M \times \R^d} u_0 m_0 < + \infty .
\ee

\subsection{Upper bounds}

We prove upper bounds on $u$. First, we observe that for any constant $l \in \R$ the function $(u-l)_+: = \max\{ u- l, 0\}$ satisfies (see Lemma~\ref{lem:truncation})
\be\label{eq:HJ-co}
\left\{
\begin{array}{rcll}
- (\partial_t + v \cdot D_x ) (u - l)_+ + H(x,v, D_v u)  \one_{\{u > l\}} &\leq& \beta \one_{\{u > l\}}, & \text{ in } \sD'((0,T)\times\M\times\R^d),\\
(u_T-l)_+ &\le& (\beta_T-l)_+, & \text{ in } \sD'(\M\times\R^d).
\end{array}
\right.
\ee

We use the notation $L^\infty + L^{q'}$ to denote the set of functions
\be
\{ h = h_1 + h_2 : h_1 \in L^\infty, h_2 \in L^{q'}\},
\ee
which becomes a Banach space when equipped with the norm
\be
\| h \|_{L^\infty + L^{q'}} : = \inf \{ \| h_1 \|_{L^\infty} + \| h_2 \|_{L^{q'}} : h = h_1 + h_2\} .
\ee

We also use the notation $L^1 \cap L^q$ to denote the intersection of $L^1$ and $L^q$ made into a Banach space under the norm
\be
\| h \|_{L^1 \cap L^q} : = \max \{ \| h \|_{L^1},  \| h \|_{L^q}\} .
\ee
Note that the dual space is given by $(L^1 \cap L^q)^\ast = L^\infty + L^{q'}$.

\begin{lemma} \label{lem:value-ub}
Let $l\in\R$ be given and let $(u,\beta,\beta_T)\in\cK_A$
satisfy \eqref{eq:HJ-co}.

(i) Then $(u - l)_+\in L^\infty_t(L^\infty + L^{q'})_{x,v}$, with the a priori estimate
\begin{align}\label{ineq:beta_a_priori}
\| (u - l)_+ \|_{L^\infty_t(L^\infty + L^{q'})_{x,v} }& \leq \| (\beta_T - l)_+ \|_{(L^\infty + L^{q'})(\M \times \R^d) } + C_H T + T^{1/q} \| \beta_+ \|_{L^{q'}}\\
&\le \| (\beta_T - l)_+ \|_{L^{s'}(\M \times \R^d) } + C_H T + T^{1/q} \| \beta_+ \|_{L^{q'}}
\end{align}

(ii) Suppose in addition that $\widetilde\cA(u,\b,\b_T)<C_{\widetilde\cA}$. Then, there exists 
$$C=C\left(C_{\widetilde\cA},\|m_0\|_{L^1\cap L^q}, T, C_F, C_G\right)>0$$ 
such that
$$
\|\beta_+\|_{L^{q'}((0,T)\times\M\times\R^d)}+ \|(\b_T)_+\|_{L^{s'}(\M\times\R^d)}\le C.
$$
\end{lemma}
\begin{proof}
First, let us note that, since $(\beta_T)_+ \in L^{s'}(\M\times\R^d)$ and $s' \geq q'$ (by Assumption~\ref{hyp:Hamiltonian}), $(\beta_T)_+ \in (L^\infty + L^{q'})(\M \times \R^d)$ and thus also $(\beta_T - l)_+ \in (L^\infty + L^{q'})(\M \times \R^d)$.

(i) Let $t \in [0,T)$ be fixed. Let $0 \leq \psi \in C^\infty_c(\M \times \R^d) $
and consider
\be
\zeta(\tau, x,v) : = \psi(x+ (t - \tau)v, v).
\ee
Then $\zeta$ is smooth and compactly supported and satisfies
\be
\partial_\tau \zeta + v \cdot D_x \zeta = 0.
\ee
By using $\zeta$ as a test function for $(u - l)_+$ over $\tau \in [t,T]$, we obtain
\begin{align*}
\int_{\M \times \R^d} (u - l)_+(t,x,v) \psi(x,v) \dd x \dd v &\leq \int_{\M \times \R^d} (\b_T - l)_+ \zeta_T \dd x \dd v \\ 
&+ \int_t^T \int_{\M \times \R^d} \zeta \left [ \beta - H(x,v, D_v u)\right] \one_{\{ u > l \}} \dd x \dd v \dd \tau .
\end{align*}
Recall that when we write $(u - l)_+(t,\cdot, \cdot)$, we are always referring to the version of $u$ that is weakly right continuous with respect to time (cf. Appendix \ref{app:time_regularity}, Lemma \ref{lem:app_trace_1}).

Since $H \geq - C_H$,
we have
\be
\int_{\M \times \R^d}(u - l)_+(t,x,v) \psi(x,v) \dd x \dd v \leq \int_{\M \times \R^d} (\beta_T - l)_+ \zeta_T \dd x \dd v + \int_t^T \int_{\M \times \R^d} \zeta \left [\beta_+ + C_H \right] \dd x \dd v \dd \tau .
\ee
Then
\begin{align*}
\int_{\M \times \R^d}(u_t - l)_{+} \psi \dd x \dd v &\leq \| (\beta_T - l)_+ \|_{L^\infty + L^{q'}} \| \zeta_T \|_{L^1 \cap L^q} \\
&+ \| \zeta \|_{L^q([t,T] \times \M \times \R^d)} \| \beta_+ \|_{L^{q'}([t,T] \times \M \times \R^d)} + C_H \| \zeta \|_{L^1([t,T] \times \M \times \R^d)}.
\end{align*}
We compute
\begin{align}
\| \zeta \|_{L^1([t,T] \times \M \times \R^d)} &= \int_t^T \int_{\M \times \R^d} \psi(x+ (t - \tau)v, v) \dd x \dd v \dd \tau \\
& = \int_t^T \int_{\M \times \R^d} \psi(x, v) \dd x \dd v \dd \tau = (T - t) \| \psi\|_{L^1(\M \times \R^d)}
\end{align}
Similarly
\be
\| \zeta_T \|_{L^1(\M \times \R^d)} = \| \psi \|_{L^1(\M \times \R^d)}
\ee
and
\be
\| \zeta \|_{L^q([t,T] \times \M \times \R^d)} = (T - t)^{1/q}  \| \psi \|_{L^q(\M \times \R^d)} .
\ee
Thus
\be
\int_{\M \times \R^d}(u_t - l)_{+} \psi \dd x \dd v \leq \left ( \| (\beta_T - l)_+ \|_{L^\infty + L^{q'}} + C_H T + T^{1/q} \| \beta_+ \|_{L^{q'}} \right ) \| \psi \|_{L^1 \cap L^q}  .
\ee
This extends by density to all non-negative $\psi \in (L^1 \cap L^q)(\M\times\R^d)$, and general $\psi \in (L^1 \cap L^q)(\M\times\R^d)$ by non-negativity of $(u-l)_+$. We conclude by the fact that $ \| (\beta_T -l)_+ \|_{L^\infty + L^{q'}}\le \|(\beta_T-l)_+\|_{L^{s'}}$. The result follows.

\medskip

(ii) By the definition of $\tilde\cA$ and the assumptions on the data one has
\begin{align}\label{estim:beta}
\nonumber\widetilde\cA(u,\b,\b_T)&\ge \|\beta_+\|_{L^{q'}}^{q'}+ \|(\b_T)_+\|_{L^{s'}}^{s'}-\int_{\M\times\R^d} [TC_F(x,v)+C_G(x,v)]\dd x\dd v \\
\nonumber&-\int_{\M\times\R^d} u(0,x,v)_+m_0(x,v)\dd x\dd v\\
&\ge \|\beta_+\|_{L^{q'}}^{q'}+ \|(\b_T)_+\|_{L^{s'}}^{s'} -C
-\left ( \| (\b_T )_+ \|_{L^{s'}} + C_H T + T^{1/q} \| \beta_+ \|_{L^{q'}} \right ) \| m_0 \|_{L^1 \cap L^q},
\end{align}
where in the last inequality we used the estimate from (i). This further yields the claim in (ii).

\end{proof}

\begin{corollary}\label{rem:u_0}
Let $(u,\b,\b_T)$ be as in the statement of Lemma \ref{lem:value-ub} such that there exists $C_{\widetilde\cA}>0$ with $\widetilde\cA(u,\beta,\beta_T)<C_{\widetilde\cA}.$ Then, there exists $C=C\left(C_{\widetilde\cA},\|m_0\|_{L^1\cap L^q}, T, C_F, C_G\right)>0$ such that
the following hold.

(i) $\|(u_0)_+\|_{L^\infty+L^{q'}(\M\times\R^d)}\le C$;

\medskip

(ii) $\ds\int_{\M \times \R^d} m_0 (u_0)_-(\dd x \dd v) \le C.$
\end{corollary}

\begin{proof} We notice that (i) is a simple consequence of Lemma \ref{lem:value-ub}(i)-(ii), by setting $l=0$ and $t=0$ (in the sense of weak trace, given in Definition~\ref{def:trace}).

For (ii), we observe
\begin{align*}
\int_{\M \times \R^d} m_0 (u_0)_- (\dd x \dd v) &= - \int_{\M \times \R^d} m_0 u_0  (\dd x \dd v) + \int_{\M \times \R^d} (u_0)_+ m_0 \dd x \dd v\\
& \leq {\Blue \widetilde\cA(u,\beta,\beta_T) - \int_0^T\int_{\M\times\R^d}\cF^*(x,v, \beta)\dd x\dd v\dd t } \\
& {\Blue \qquad- \int_{\M\times\R^d}\cG^*(\beta_T) \dd x \dd v +\|(u_0)_+\|_{L^\infty+L^{q'}}\|m_0\|_{L^1\cap L^q} . }
\end{align*}
{ \Blue
By the bounds~\eqref{eq:Fconj-lb} on $\cF^\ast$ and the corresponding estimates for $\cG^\ast$,
\be
\sup_\beta \{ - \cF^\ast(x,v,\beta) \} \leq C_F(x,v), \quad \sup_{\beta_T} \{ - \cG^\ast(x,v,\beta_T) \} \leq C_G(x,v).
\ee
Hence, using the above bounds and (i), we obtain
\be
\int_{\M \times \R^d} m_0 (u_0)_- (\dd x \dd v) \leq  C_{\widetilde\cA}+\|C_F\|_{L^1} + \|C_G\|_{L^1} + C \|m_0\|_{L^1\cap L^q},
\ee
which completes the proof. }
\end{proof}

\subsection{Local $L^1$ bounds}

Next, we prove bounds on the negative parts of $u$ and $\beta$. We will obtain $L^1_\loc(\cU_{m_0})$ bounds, by use of a duality argument involving a certain class of test functions which satisfy the continuity equation associated to the control system.

\begin{lemma} \label{lem:duality-regular}
Let $a \in C^1_b([0,T] \times \M \times \R^d;\R^d)$  be a bounded control. Let $\phi_0 \in C^1_c(\M \times \R^d)$ satisfy $0 \leq \phi_0 \leq m_0$.
Let $\phi \in C^1_c(\cU_{m_0})$ be the solution of the continuity equation
\be
\partial_t \phi + v \cdot D_x \phi + \diver_v(a \phi) = 0 , \qquad \phi\vert_{t=0} = \phi_0 .
\ee
Then, for any $(u,\beta,\beta_T) \in \cK_A$ such that $\widetilde \cA(u, \beta,\beta_T) < + \infty$, the following hold:
\begin{itemize}
\item $u \in L^\infty_t L^1_{x,v} (\phi)$, that is,
\be
\esssup_{t \in [0,T]} \int_{\M \times \R^d} |u_t| \phi_t \dd x \dd v < + \infty .
\ee
\item The negative part of $\beta$ satisfies $\beta_- \in L^{1}_{t,x,v}(\phi)$.
\item The negative part of $\beta_T$ satisfies $(\beta_T)_-\in L^1_{x,v}(\phi_T)$.
\item The following estimate holds:
\begin{multline}
\| u \|_{L^\infty_t L^1_{x,v} (\phi)} + \| \beta_-\|_{L^{1}_{t,x,v}(\phi)} +\|(\b_T)_-\|_{L^1_{x,v}(\phi_T)}+ \| D_v u \|^r_{L^r_{t,x,v}(\phi)} \\ \leq C(a,\phi, m_0, H, T) \left ( 1 + \| \beta_+ \|_{L^{q'}} + \| (\beta_T)_+ \|_{L^\infty + L^{q'}} \right ) - \int_{\M \times \R^d} u_0 m_0 \dd x \dd v . 
\end{multline}
\end{itemize}
\end{lemma}
\begin{proof}
Note the following properties of $\phi$:
\begin{itemize}
\item $\phi$ is non-negative,
\item $\phi$ has compact support contained in $\cU_{m_0}$.
\item $\phi \in (L^1 \cap L^\infty)_{t,x,v}$.
\end{itemize}
In particular, since $\phi_t \in (L^1 \cap L^{q})(\M \times \R^d)$ for any $t \in [0,T]$, then
\be
\int_{\M \times \R^d} (u_t)_+ \phi_t \dd x \dd v \leq \| \phi_t \|_{(L^1 \cap L^{q})_{x,v}} \| u_+ \|_{L^\infty_t(L^\infty + L^{q'})_{x,v}} .
\ee
By Lemma~\ref{lem:value-ub},
\be \label{est:u-positive}
\| u_+ \|_{L^\infty_t(L^\infty + L^{q'})_{x,v}} \leq C(T,H) \left ( 1 + \| (\beta_T)_+ \|_{(L^\infty + L^{q'})(\M \times \R^d) } + \| \beta_+ \|_{L^{q'}_{t,x,v}} \right )
\ee
and thus for $t\in [0,T]$, we have
\be
\int_{\M \times \R^d} (u_t)_+ \phi_t \dd x \dd v \leq C(T,H, \phi) \left ( 1 + \| (\beta_T)_+ \|_{(L^\infty + L^{q'})(\M \times \R^d) } + \| \beta_+ \|_{L^{q'}_{t,x,v}} \right ),
\ee
where $(u_t)_+$ is understood in the sense of weak trace (cf. Lemma \ref{lem:app_trace_1}, Definition~\ref{def:trace}).

For the negative part we make use of the equation.
A density argument shows that $\phi$ is admissible as a test function in the weak form of the Hamilton-Jacobi inequality satisfied by $u$.
Thus for $0 \leq s < t \leq T$,
\begin{align*}
\int_{\M \times \R^d} u_s \phi_s \dd x \dd v - \int_{\M \times \R^d} u_t \phi_t \dd x \dd v  & + \int_s^t \int_{\M \times \R^d} u (\partial_t \phi + v \cdot D_x \phi ) \dd x \dd v \dd \tau \\
&+ \int_s^t \int_{\M \times \R^d} \phi H(x,v, D_v u)\dd x \dd v \dd \tau \leq \int_s^t \int_{\M \times \R^d} \beta \phi \dd x \dd v \dd \tau .
\end{align*}
We apply this in the case $s=0$, $t \in (0,T]$. Using the fact that $\phi$ satisfies the continuity equation in a pointwise sense,
\begin{align}\label{ineq:for_convergence}
\int_{\M \times \R^d} & u_0 \phi_0 \dd x \dd v - \int_{\M \times \R^d} u(t,x,v) \phi(t,x,v) \dd x \dd v \\ 
&- \int_0^t \int_{\M \times \R^d} u \diver_v(a \phi) \dd x \dd v \dd \tau + \int_0^t \int_{\M \times \R^d} \phi H(x,v, D_v u)\dd x \dd v \dd \tau \leq \int_0^t \int_{\M \times \R^d} \beta \phi \dd x \dd v \dd \tau .
\end{align}
Here, let us notice that we have used the existence of weak traces in the sense of Lemma \ref{lem:app_trace_1}. In particular the integral $\int_{\M \times \R^d} u_0 \phi_0 \dd x \dd v$ is meaningful and finite, since $\spt(\phi_0)\subseteq \spt(m_0)$ (Definition \ref{def:u0m0}).

Since $D_v u \in L^r_\loc(\cU_{m_0})$ and $a \phi \in C^1$ has compact support contained in $\cU_{m_0}$,
 we may integrate by parts to obtain
\be\label{eq:for_convergence}
- \int_0^t \int_{\M \times \R^d} u \, \diver_v(a \phi) \dd x \dd v \dd \tau = \int_0^t \int_{\M \times \R^d} a \cdot D_v u \, \phi \dd x \dd v \dd \tau .
\ee
Then estimate
\be
\left | \int_0^t \int_{\M \times \R^d} a \cdot D_v u \, \phi \dd x \dd v \dd \tau \right | \leq \| a \|_{L^{r'}(\phi)} \| D_v u \|_{L^r(\phi)} ,
\ee
{\Blue where, in order to lighten the notation, we have used the shorthand
\be
\| h\|_{L^p(\phi)} : = \left ( \int_0^T \int_{\M \times \R^d} |h|^p \phi \dd x \dd v \dd t \right )^{1/p}, \quad p \in [1, + \infty)
\ee
to denote $L^p$ norms with respect to the measure on $[0,T] \times \M \times \R^d$ with density $\phi$ with respect to Lebesgue measure.
}
Thus
\begin{multline}
 - \int_{\M \times \R^d} u_t \phi_t \dd x \dd v  + \int_0^t \int_{\M \times \R^d} \phi H(x,v, D_v u)\dd x \dd v \dd \tau \\ \leq \| a \|_{L^{r'}(\phi)} \| D_v u \|_{L^r(\phi)} - \int_{\M \times \R^d} u_0 \phi_0 \dd x \dd v + \int_0^t \int_{\M \times \R^d} \beta \phi \dd x \dd v \dd \tau .
\end{multline}
Using the lower bounds on the Hamiltonian $H$, rearranging terms and using Young's inequality for products (with a small parameter), we obtain
\begin{multline}
 - \int_{\M \times \R^d} u_t \phi_t \dd x \dd v  + \int_0^t \int_{\M \times \R^d} \beta_- \phi \dd x \dd v \dd \tau+ c \int_0^t \int_{\M \times \R^d} \phi |D_v u |^r \dd x \dd v \dd \tau \\ \leq C \| a \|_{L^{r'}(\phi)}^{r'} + C_H \int_0^t \int_{\M \times \R^d} \phi \dd x \dd v \dd \tau - \int_{\M \times \R^d} u_0 \phi_0 \dd x \dd v + \int_0^t \int_{\M \times \R^d} \beta_+ \phi \dd x \dd v \dd \tau .
\end{multline}
Then
\begin{align}\label{ineq:energy}
-  \int_{\M \times \R^d} &u_t \phi_t \dd x \dd v  + \int_0^t \int_{\M \times \R^d} \beta_- \phi \dd x \dd v \dd \tau+ c \int_0^t \int_{\M \times \R^d} \phi |D_v u |^r \dd x \dd v \dd \tau \\ 
 & \leq C(a,\phi) \left ( C_H + \| \beta_+ \|_{L^{q'}} \right ) - \int_{\M \times \R^d} u_0 m_0 \dd x \dd v  + \int_{\M \times \R^d} (u_0)_+ (m_0 - \phi_0) \dd x \dd v.
  \end{align}
Finally, since $0 \leq m_0 - \phi_0 \leq m_0 \in L^1 \cap L^q$, we use the $(L^\infty + L^{q'})_{x,v}$ bounds on the positive part $(u_0)_+$ (Equation~\eqref{est:u-positive}) to conclude that
\begin{multline}
 \int_{\M \times \R^d} |u_t| \phi_t \dd x \dd v  + \int_0^t \int_{\M \times \R^d} \beta_- \phi \dd x \dd v \dd \tau+ c \int_0^t \int_{\M \times \R^d} \phi |D_v u |^r \dd x \dd v \dd \tau \\ \leq C(a,\phi, m_0, H , T) \left ( 1 + \| \beta_+ \|_{L^{q'}} + \| (\beta_T)_+ \|_{L^\infty + L^{q'}} \right ) - \int_{\M \times \R^d} u_0 m_0 \dd x \dd v .  
 \end{multline}
Notice that by setting $t=T$, \eqref{ineq:energy} and the fact that $u_T\le \b_T$ (together with the bounds that we already have on $(\b_T)_+$) readily yield also that $(\b_T)_-\in L^1(\phi_T).$

This completes the proof.
\end{proof}

\begin{corollary}\label{cor:bounds_u}
Let $(u,\beta, \beta_T) \in \cK_A$ such that $\widetilde \cA(u, \beta,\beta_T) < + \infty$.
Then $u \in L^1_\loc(\cU_{m_0})$, $\beta_- \in L^1_\loc(\cU_{m_0})$, $(\b_T)_-\in L^1_{\rm{loc}}(\M\times\R^d)$ and $D_v u \in L^r_\loc(\cU_{m_0})$.
\end{corollary}
\begin{proof}
First, consider a compact set
\be
K \subset \bigcup_{a \in C^1_b, \, 0 \leq \phi_0 \in C^1_c(\{m_0 > 0 \})} \{ \phi > 0 \},
\ee
where $\phi \in C^1$ denotes the solution of the continuity equation
\be \label{eq:phi-cty}
\partial_t \phi + v \cdot D_x \phi + \diver_v (a \phi) = 0, \qquad \phi \vert_{t=0} = \phi_0 .
\ee
By compactness of $K$, there exist finitely many $\phi_i$, $i = 1, \ldots, k$ such that
\be
K \subset \bigcup_{i=1}^k \{ \phi_i > 0 \} .
\ee
The function $\max_i \phi_i$ is continuous and so
\be
0 < \delta_K : = \inf_{K} \max_i \phi_i .
\ee
Then
\be
\| u \|_{L^1(K)} + \| \beta_- \|_{L^1(K)} + \| D_v u \|_{L^r(K)} \leq \delta_K^{-1} \sum_{i=1}^k \| u \|_{L^1(\phi_i)} + \| \beta_- \|_{L^1(\phi_i)} + \| D_v u \|_{L^r(\phi_i)} .
\ee
By Lemma~\ref{lem:duality-regular}, this leads to the estimate
\be
\| u \|_{L^1(K)} + \| \beta_- \|_{L^1(K)} + \| D_v u \|_{L^r(K)} \leq C,
\ee
where $C=C \left (K, \widetilde \cA(u, \beta,\beta_T) \right )$. We now claim that
\be
\cU_{m_0}  \subset \bigcup_{a \in C^1_b, \, 0 \leq \phi_0 \in C^1_c(\{m_0 > 0 \})} \{ \phi > 0 \}.
\ee
This follows from the controllability of the ODE system
\be \label{eq:ODE}
\dot x = v, \quad \dot v = a
\ee
on $\M \times \R^d$.
That is, for any initial datum $(x_0, v_0) \in \M \times \R^d$ and target $(t_\ast ,x_\ast,v_\ast) \in (0, T] \times \M \times \R^d$, there exists a control function $a$ such that the solution $(x(\tau), v(\tau))$ of the ODE \eqref{eq:ODE} with $(x(0), v(0)) = (x_0, v_0)$ satisfies $(x(t_\ast), v(t_\ast)) = (x_\ast, v_\ast)$.

Next, note that (since $m_0$ is continuous) $\{ m_0 > 0 \}$ contains a closed ball $\bar{B}_r(x_0, v_0)$ for some point $(x_0, v_0)$ and some $r > 0$. Thus there exists $0 \leq \phi_0 \in C^1_c(\{ m_0 >  0\})$ such that $\phi_0 > 0$ on $\bar{B}_r(x_0, v_0)$. Consider the solution $\phi$ of \eqref{eq:phi-cty} for the control $a$ found above and with this choice of $\phi_0$. It follows that $(t_\ast, x_\ast, v_\ast) \in \{ \phi > 0 \}$.

Finally, we notice that by the structure of the set $\cU_{m_0}$, we have the bound $(\b_T)_- \in L^1_{\rm{loc}}(\M\times\R^d)$.

\end{proof}

\section{Duality for the Relaxed Problem}\label{sec:4}

\begin{theorem} \label{thm:duality}
Problems \ref{prob:density} and \ref{prob:relaxed} are in duality:
\be
\inf_{(u,\beta,\beta_T) \in \cK_\cA} \widetilde \cA(u,\beta,\beta_T) = - \min_{(m,w) \in \cK_\cB} \cB(m,w) .
\ee
\end{theorem}
\begin{proof}
For $u \in C^1_b([0,T] \times \M \times \R^d)$ such that $\cA(u)<+\infty$, the triple $(u, -\partial_t u - v \cdot D_x u + H(x,v, D_vu),u_T)$ lies in $\cK_{\cA}$. Thus
\be
\inf_{(u,\beta) \in \cK_\cA} \widetilde \cA(u,\beta,\beta_T) \leq \inf_{u \in C^1_b} \cA(u) .
\ee
By the duality result of Lemma~\ref{lem:duality1},
\be 
\inf_{(u,\beta,\beta_T) \in \cK_\cA} \widetilde \cA(u,\beta,\beta_T) \leq \inf_{u \in C^1_b} \cA(u) = - \min_{(m,w) \in \cK_{\cB}} \cB(m,w) .
\ee
It therefore remains only to prove the reverse inequality. This follows from Lemma~\ref{lem:inequality} below, which states that for all $(u, \beta,\beta_T) \in \cK_\cA$ and $(m,w) \in \cK_\cB$,
\be
\widetilde \cA(u,\beta,\beta_T) \geq - \cB(m,w) .
\ee
Taking the infimum over $(u, \beta,\beta_T) \in \cK_\cA$ and supremum over $(m,w) \in \cK_\cB$ gives
\be
\inf_{(u,\beta,\beta_T) \in \cK_\cA} \widetilde \cA(u,\beta,\beta_T) \geq - \min_{(m,w) \in \cK_\cB} \cB(m,w) 
\ee
as required.
\end{proof}

\begin{lemma} \label{lem:inequality}
Let $(u, \beta,\beta_T) \in \cK_\cA$ and $(m,w) \in \cK_\cB$ such that $\widetilde \cA(u,\beta,\beta_T) ,  \cB(m,w) < + \infty$.
Then
\be
\widetilde \cA(u,\beta,\beta_T) + \cB(m,w) \geq 0 .
\ee
\end{lemma}

In the proof of this lemma we require the following observation regarding the commutator between the operator $v \cdot D_x $ and the operator given by convolution with a fixed function.

\begin{lemma} \label{lem:commutator}
Let $\chi : \R^d \to [0, + \infty)$ be a function such that $(1 + |v|) \chi \in L^1_v(\R^d)$. Let $h \in L^p_v(\R^d)$ for $p \in [1, + \infty]$. Then
\be
v \chi \ast h - \chi \ast (vh) = (v \chi) \ast h .
\ee 
\end{lemma}
\begin{proof}
By direct computation, for all $v \in \R^d$,
\begin{align}
v \chi \ast h(v) - \chi \ast (vh) & = v \int_{\R^d} h(v - z) \chi(z) \dd z - \int_{\R^d} (v-z) h(v-z) \chi (z) \dd z \\
& = \int_{\R^d} z \chi (z) h(v-z) \dd z \\
& = (v \chi) \ast h .
\end{align}
\end{proof}

\begin{proof}[Proof of Lemma~\ref{lem:inequality}]
The overall idea of the proof is to use $m$ as a test function in the weak form of the inequality 
\be
-\partial_t u - v\cdot D_x u+H(x,v,D_v u) \leq \beta
\ee
and its terminal condition
$$
u_T\le \b_T.
$$
To make this valid, we must first introduce an approximation procedure.

First, we introduce a lower cut-off on $u$ and $\beta$. Let $l \leq 0$ and define $u_l : = \max \{ u, l \}$. Similarly, for $k \leq 0$, let $\beta_k : = \max \{ \beta, k \}$. Then by Lemma \ref{lem:truncation} we obtain
\be \label{eq:HJ-co2}
\left\{
\begin{array}{rcll}
-\partial_t u_l - v\cdot D_x u_l +H(x,v,D_v u) \one_{\{ u > l \}} &\leq& \beta_k \one_{\{ u > l \}}, & (0,T)\times\M\times\R^d\\[3pt]
(u_T)_l &\le& (\b_T)_l, & \M\times\R^d,
\end{array}
\right.
\ee
in the sense of distributions. By Lemma~\ref{lem:value-ub}, $u_l \in L^\infty_t (L^\infty + L^{q'})_{x,v}$. We emphasize that $k$ and $l$ are taken to be possibly independent at this point.

Next, we approximate $m$ by a function in $C^1_c$, which is then an admissible test function for the Hamilton-Jacobi equation \eqref{eq:HJ-co2}.
We regularize $m$ by convolution with a mollifier. For ease of presentation, it will be convenient to work with the time, space and velocity variables separately.
Fix $\chi \in C^\infty_c(\R^d)$ and define $\chi_\e$ for $\e > 0$ by
\be
\chi_\e \left ( v \right ) : = \e^{-d} \chi \left ( \frac{v}{\e} \right ) .
\ee
For the space variable, consider $\psi \in C^\infty_c(\R^{d})$ and for $\delta > 0$ let
\be
\psi_\delta(x) : = \delta^{-d} \psi \left ( \frac{x}{\delta} \right ) .
\ee
For the time variable, fix $\theta \in C^\infty_c(\R)$ 
and for $\eta > 0$ let
\be
\theta_\eta(t) : = \eta^{-1} \theta \left ( \frac{t}{\eta} \right ) .
\ee
We then define the full mollifier $\varphi$ by
\be
\varphi(t,x,v) = \theta_\eta(t) \psi_\delta(x) \chi_\e(v) .
\ee

Then define the smooth functions
\be
\mt : = \varphi \ast_{t,x,v} m \ \ {\rm{and}}\ \ \wt =\varphi \ast_{t,x,v} w.
\ee
Notice that for the convolution in time, $(m,w)$ needs to be extended. We choose the following extensions. We set $w(t,\cdot,\cdot)=0$ to for $t<0$ and $t>T$. Then, if $t<0$, we set $m(t,\cdot,\cdot)$ to be the solution to the problem
$$
\left\{
\begin{array}{ll}
\partial_t m+ v\cdot D_x m = 0, & {\rm{in}}\ (-\eta,0)\times\M\times\R^d,\\
m(0,\cdot,\cdot) = m_0, & {\rm{in}}\ \M\times\R^d.
\end{array}
\right.
$$
Similarly, for $t>T$ we set $m(t,\cdot,\cdot)$ to be the solution to 
$$
\left\{
\begin{array}{ll}
\partial_t m+ v\cdot D_x m = 0, & {\rm{in}}\ (T,T+\eta)\times\M\times\R^d,\\
m(T,\cdot,\cdot) = m_T, & {\rm{in}}\ \M\times\R^d,
\end{array}
\right.
$$
where $m_T$ is the trace of $m$ in time at $t=T$.

As the final step in the approximation, we localize $\mt$.
As localizers we consider smooth functions $\zeta_R \in C^\infty_c(\M \times \R^d)$ such that 
\be \label{hyp:zeta}
\zeta_R (x,v) = \begin{cases}
1 & |x| \leq R^2, |v| \leq R \\
0 & |x| > 2 R^2, |v| > 2R .
\end{cases} \qquad 
|D_x \zeta_R| \leq \frac{C}{R^2} , |D_v \zeta_R| \leq \frac{C}{R} \, .
\ee
We then define $\mt^{(R)} : = \zeta_R \mt$ and $\wt^{(R)} : = \zeta_R \wt$.

Then $\mt^{(R)}$ satisfies the equation
\be
\partial_t \mt^{(R)} + v \cdot D_x \mt^{(R)} + \diver_v \wt^{(R)} = \cE_{\eta, \delta, \e, R} ,
\ee
where the error term is given by
\be \label{def:error}
\cE_{\eta, \delta, \e, R} : =  \zeta_R \, \left [ v \cdot D_x, \chi_\e \ast_v \right ] \theta_\eta \ast_t \psi_\delta \ast_{x} m + (v \cdot D_x \zeta_R) \mt  + (D_v \zeta_R) \cdot \wt. 
\ee
Here, we use the standard commutator notation $[\Lambda_1,\Lambda_2]f:=\Lambda_1(\Lambda_2 f)-\Lambda_2(\Lambda_1 f)$, where $\Lambda_1,\Lambda_2$ are some  operators acting on the function $f$.

{\bf Convergence of the Error Term.}
We show that the error term $\cE_{\eta, \delta, \e, R}$ defined by \eqref{def:error} converges to zero in the space $L^1_t (L^1 + L^q)_{x,v}$, as $R \to + \infty$ and $\eta,\e,\delta \to 0$, under a certain relationship between these parameters.

For the first term, either for $p=1$ or $p=q$, using the explicit formula for the commutator we estimate
\begin{align}
\| \zeta_R \, \left [ v \cdot D_x, \chi_\e \ast_v \right ] \theta_\eta \ast_t \psi_\delta \ast_{x} m \|_{L^1_t L^p_{x,v}} & \leq \| \left [ v \cdot D_x, \chi_\e \ast_v \right ](\theta_\eta \psi_\delta) \ast_{t,x} m \|_{L^1_t L^p_{x,v}} \\
& \leq \| \left [ v, \chi_\e \ast_v \right ] \theta_\eta \ast_t (D_x \psi_\delta) \ast_{x} m \|_{L^1_t L^p_{x,v}} \\
& \leq \| (v \chi_\e) \ast_v \theta_\eta \ast_t (D_x \psi_\delta) \ast_{x} m \|_{L^1_t L^p_{x,v}} \\
& \leq \| v \chi_\e \|_{L^1} \| \theta_\eta \|_{L^1} \| D_x \psi_\delta \|_{L^1} \| m \|_{L^1_t L^p_{x,v}} \\
& \leq  \e \delta^{-1} \cdot C T^{1/p'} \| m \|_{ L^p_{t,x,v}},
\end{align}
where we have used Lemma \ref{lem:commutator} in the third inequality.
Thus by choosing $\e = \e(\delta)$ sufficiently small with respect to $\delta$, we may ensure that
 \be
 \lim_{\delta \to 0} \sup_{\eta, R} \| \zeta_R \, \left [ v \cdot D_x, \chi_{\e(\delta)} \ast_v \right ] \theta_\eta \ast_t \psi_\delta \ast_{x} m \|_{L^1_t L^p_{x,v}} = 0 .
 \ee

For the second term, observe that for all $R>0$, $|v \cdot D_x \zeta_R| \leq C R^{-1}$ and thus
\be
\| (v \cdot D_x \zeta_R) \mt \|_{L^1_t (L^1 \cap L^q)_{x,v}} \leq C R^{-1} \| \mt \|_{L^1_t(L^1 \cap L^q)_{x,v}} \leq C R^{-1} \| m \|_{L^1_t(L^1 \cap L^q)_{x,v}} \leq C_T R^{-1} \| m \|_{(L^1 \cap L^q)_{t,x,v}} .
 \ee
 It follows that
 \be
\lim_{R \to + \infty} \sup_{\eta, \delta, \e} \| (v \cdot D_x \zeta_R) \mt \|_{L^1_t (L^1 \cap L^q)_{x,v}} = 0.
 \ee

 For the third term, for either $p=1$ or $p=q$ we have
 \begin{align}
 \| (D_v \zeta_R)\cdot \wt \|_{L^1_t L^p_{x,v}} &\leq \frac{C}{R} \| \wt \|_{L^1_t L^p_{x,v}} \\
 & \leq \frac{C}{R}  \| \varphi \|_{L^1_t L^p_{x,v}} \| w \|_{L^1_{t,x,v}} \\
 & \leq \frac{C}{R} \delta^{-d/p'} \e^{-d/p'} \| w \|_{L^1_{t,x,v}}  .
 \end{align}
 Taking $\e = \e(\delta)$ as above, we can then ensure this term converges to zero by choosing $R = R(\delta)$ sufficiently large with respect to $\delta$ and $\e(\delta)$.
 Thus, for this choice of $\e(\delta), R(\delta)$, we have
\be
\lim_{\delta \to 0} \sup_\eta \| (D_v \zeta_R)\cdot \wt \|_{L^1_t L^p_{x,v}} = 0 .
\ee

Altogether, we have found that there exists a regime $R = R(\delta)$ and $\e = \e(\delta)$ such that
\be
\lim_{\delta \to 0} \sup_\eta \| \cE_{\eta, \delta, \e, R} \|_{L^1_t L^p_{x,v}} = 0 .
\ee

{\bf Testing the Equation.}
Using $\mt^{(R)}$ as a test function in the weak form of the equation for $u_l$, one obtains
\begin{align}
&\int_{\M \times \R^d} u_l(0,x,v) \mt^{(R)}(0,x,v) \dd x \dd v - \int_{\M \times \R^d} (\b_T)_l  \mt^{(R)}(T,x,v) \dd x \dd v \\ 
&+ \int_0^T \int_{\M \times \R^d} u_l (\partial_t \mt^{(R)} + v \cdot D_x \mt^{(R)} ) \dd x \dd v \dd t + \int_0^T \int_{\M \times \R^d} \mt^{(R)} H(x,v, D_v u)  \one_{\{u > l\}}  \dd x \dd v \dd t \\
&\leq \int_0^T \int_{\M \times \R^d} \beta_k \mt^{(R)}  \one_{\{u > l\}}  \dd x \dd v \dd t .
\end{align}
Using the equation satisfied by $\mt^{(R)}$, 
we have
\begin{align}\label{ineq:test_error}
&\int_{\M \times \R^d} u_l(0,x,v) \mt^{(R)}(0,x,v) \dd x \dd v - \int_{\M \times \R^d} (\b_T)_l \mt^{(R)}(T,x,v) \dd x \dd v \\ 
&- \int_0^T \int_{\M \times \R^d} u_l \diver_v \wt^{(R)} \dd x \dd v \dd t + \int_0^T \int_{\M \times \R^d} \mt^{(R)} H(x,v, D_v u)  \one_{\{u > l\}}  \dd x \dd v \dd t \\
&\leq \int_0^T \int_{\M \times \R^d} \beta_k \mt^{(R)}  \one_{\{u > l\}}  \dd x \dd v \dd t  - \int_0^T \int_{\M \times \R^d} u_l \, \cE_{\eta,\delta, \e, R} \dd x \dd v \dd t  .
\end{align}

Next, note that $D_v u \in L^r_\loc(\cU_{m_0})$. By the chain rule for Lipschitz functions composed with Sobolev-regular functions,
\be
D_v u_l = D_v u \one_{\{ u > l \}} .
\ee
Thus, using the definition of distributional derivative we may integrate by parts to obtain
\be
- \int_0^T \int_{\M \times \R^d} u_l \, \diver_v \wt^{(R)} \dd x \dd v \dd t = \int_0^T \int_{\M \times \R^d} \wt^{(R)} \cdot D_v u \one_{\{ u > l \}}  \dd x \dd v \dd t .
\ee

Since $\cB(m,w)$ is finite, $w$ is absolutely continuous with respect to $m$. It follows that there exists $\at\in L^{r'}([0,T]\times\M\times\R^d;\tilde m)$ such that
\be
\wt = \at \, \mt .
\ee
Thus
\begin{align}
 &\int_0^T \int_{\M \times \R^d} \left ( \wt^{(R)} \cdot D_v u + \mt^{(R)} H(x,v, D_v u) \right ) \one_{\{u > l\}}  \dd x \dd v \dd t \\
 & \qquad \qquad =  \int_0^T \int_{\M \times \R^d} \left ( \at \cdot D_v u +  H(x,v, D_v u) \right ) \mt^{(R)} \one_{\{u > l\}}  \dd x \dd v \dd t \\
 & \qquad \qquad \geq - \int_0^T \int_{\M \times \R^d} L(x,v, - \at) \mt^{(R)} \one_{\{u > l\}}  \dd x \dd v \dd t .
 \end{align}
 
 Substituting this, we obtain
 \begin{align}\label{ineq:energy-2}
&\int_{\M \times \R^d} u_l(0,x,v) \mt^{(R)}(0,x,v) \dd x \dd v - \int_{\M \times \R^d} (\b_T)_l \mt^{(R)}(T,x,v) \dd x \dd v \\ 
&- \int_0^T \int_{\M \times \R^d} L(x,v, - \at) \mt^{(R)} \one_{\{u > l\}}  \dd x \dd v \dd t \\
&\leq \int_0^T \int_{\M \times \R^d} \beta_k \mt^{(R)}  \one_{\{u > l\}}  \dd x \dd v \dd t  - \int_0^T \int_{\M \times \R^d} u_l \, \cE_{\eta, \delta, \e, R} \dd x \dd v \dd t  .
\end{align}

We have shown above that there exists a regime $R = R(\delta)$ and $\e = \e(\delta)$ such that the final term converges to zero uniformly in $\eta$ as $\delta$ tends to zero, since $u_l \in L^\infty_t (L^\infty + L^{q'})_{x,v}$. We now discuss the convergence of the other terms.

{\bf Boundary Terms.}
We consider the boundary terms at $t = 0,T$. Note that Lemma \ref{lem:value-ub} and Corollary \ref{rem:u_0} yield $u_l(0, \cdot) \in L^{q'} + L^\infty$ (since $(u_0)_+\in L^{\infty}+L^{q'}$ and $u_l(0,\cdot)$ is bounded below), while by Lemma \ref{lem:value-ub} and Corollary \ref{cor:bounds_u} we have $(\b_T)_l \in L^{s'} + L^\infty$

We first show that $\mt^{(R)}_0$ converges to $m_0$ in $L^1 \cap L^q$, and $\mt^{(R)}_T$ converges to $m_T$ in $L^1 \cap L^q$, in the limit as $\delta$ tends to zero for a certain regime $\eta = \eta(\delta)$ and $R = R(\delta), \e = \e(\delta)$ according to the regime already found above.

For $t = 0,T$, we write
\be \label{boundary-decomposition}
\mt^{(R)}_t - m_t = (\mt^{(R)}_t - \zeta_R \psi_\delta \ast_x \chi_\e \ast_v m_t) + \zeta_R (\psi_\delta \ast_x \chi_\e \ast_v m_t - m_t)  + m_t(\zeta_R - 1).
\ee

We first note that $m_0 \in (L^1 \cap L^q)_{x,v}$ by the assumption that it is a bounded probability density, while $m_T \in L^s_{x,v}$ since the energy $\cB(m,w)$ is finite, and $m_T \in L^1_{x,v}$ since the continuity equation conserves mass.

Then, since $|m_t(\zeta_R - 1)| \leq m_t$, if $m_t \in L^p$ (where $p\in\{1,q\}$ or $p\in\{1,s\}$) then by dominated convergence 
\be
\lim_{R \to + \infty} \| m_t(\zeta_R - 1)\|_{L^p_{x,v}} = 0 . 
\ee
Moreover, by continuity of translations in $L^p$,
\be
\lim_{\delta \to 0} \| \psi_\delta \ast_x \chi_{\e(\delta)} \ast_v m_t - m_t \|_{L^p_{x,v}} = 0 .
\ee
Since for all $R > 0$
\be
 \| \zeta_R (\psi_\delta \ast_x \chi_\e \ast_v m_t - m_t) \|_{L^p_{x,v}} \leq  \| \psi_\delta \ast_x \chi_\e \ast_v m_t - m_t \|_{L^p_{x,v}},
\ee
it follows that
\be
\lim_{\delta \to 0} \sup_R \| \zeta_R ( \psi_\delta \ast_x \chi_{\e(\delta)} \ast_v m_t - m_t ) \|_{L^p_{x,v}} = 0 .
\ee
Therefore, the latter two terms of \eqref{boundary-decomposition} converge to zero as $\delta$ tends to zero with $\e = \e(\delta)$ and $R = R(\delta)$ as already specified above, in $(L^1 \cap L^q)_{x,v}$ for $t = 0$ and in $(L^1 \cap L^s)_{x,v}$ for $t = T$.

It remains to estimate the difference $ \psi_\delta \ast_x \chi_\e \ast_v m_t - \tilde m_t$. For any function $f \in L^p_{x,v}$ ($p \in [1, +\infty]$ to be specified),
\be
\int_{\M \times \R^d} \left (  \psi_\delta \ast_x \chi_\e \ast_v m_t - \tilde m_t \right ) \zeta_R f \dd x \dd v = \int_{\M \times \R^d} \left ( m_t - (\theta_\eta \ast_t m)_t \right ) \psi_\delta \ast_x \chi_\e \ast_v (\zeta_R f) \dd x \dd v .
\ee
We use the notation $\tilde f : = \psi_\delta \ast_x \chi_\e \ast_v (\zeta_R f) $. Writing the time convolution explicitly, we obtain
\be
\int_{\M \times \R^d} \left ( m_t - (\theta_\eta \ast_t m)_t \right ) \tilde f \dd x \dd v = \int_{-\infty}^\infty \int_{\M \times \R^d} \theta_\eta(\tau) \left ( m_t - m_{t - \tau} \right ) \tilde f \dd x \dd v \dd \tau.
\ee
Next, we use estimates on $\partial_t m$: 
\be
\langle m_t - m_{t-\tau}, \tilde f\rangle = \int_{t-\tau}^t \langle \partial_t m , \tilde f \rangle_s  \dd s =  \int_{t-\tau}^t \langle m , v \cdot \nabla_x \tilde f \rangle_s +  \langle w , \nabla_v \tilde f \rangle_s  \dd s .
\ee
Then, since $m,w \in L^1$,
\be
|\langle m_t - m_{t-\tau}, \tilde f\rangle | \leq \left ( \int_{t-\tau}^t \| m_s \|_{L^1_{x,v}} + \| w_s \|_{L^1_{x,v}} \dd s  \right ) \left ( \| v \cdot \nabla_x \tilde f \|_{L^\infty} + \| \nabla_v \tilde f \|_{L^\infty} \right ).
\ee
We estimate $\tilde f$:
\begin{align}
 \| v \cdot \nabla_x \tilde f \|_{L^\infty} &=  \| v \cdot \nabla_x \psi_\delta \ast_x \chi_\e \ast_v (\zeta_R f) \|_{L^\infty} \\
 & \leq (R + C\e) \|\nabla_x \psi_\delta \|_{L^{p'}} \| \chi_\e \|_{L^{p'}} \| f \|_{L^p} \\
  & \lesssim (R + C\e) \delta^{-(1 + d/p)} \e^{-d/p} \| f \|_{L^p} 
\end{align}
and similarly
\begin{align}
 \| \nabla_v \tilde f \|_{L^\infty} &=  \|  \psi_\delta \ast_x \nabla_v \chi_\e \ast_v (\zeta_R f) \|_{L^\infty} \\
 & \leq \| \psi_\delta \|_{L^{p'}} \| \nabla_v \chi_\e \|_{L^{p'}} \| f \|_{L^p} \\
  & \lesssim  \delta^{-d/p} \e^{-(1+d/p)} \| f \|_{L^p} .
\end{align}
Thus
\be
|\langle m_t - m_{t-\tau}, \tilde f\rangle | \leq C(\delta, \e, R) \left ( \int_{t-\tau}^t \| m_s \|_{L^1_{x,v}} + \| w_s \|_{L^1_{x,v}} \dd s  \right )  \| f \|_{L^p_{x,v}}.
\ee
Finally
\begin{align}
\left | \int_{\M \times \R^d} \left ( m_t - (\theta_\eta \ast_t m)_t \right ) \tilde f \dd x \dd v \right | & \leq C(\delta, \e, R) \left ( \int_{t-\eta}^{t+\eta} \| m_s \|_{L^1_{x,v}} + \| w_s \|_{L^1_{x,v}} \dd s  \right )  \| f \|_{L^p_{x,v}} \\
& \leq \omega(\eta) \, C(\delta, \e, R) \,\| f \|_{L^p_{x,v}},
\end{align}
where $\lim_{\eta \to 0} \omega(\eta) = 0$. Thus it is possible to choose $\eta = \eta(\delta)$ depending on $\delta, \e(\delta), R(\delta)$ in such a way that
\be \label{eq:eta-constraint}
\lim_{\delta \to 0} \omega(\eta) \, C(\delta, \e, R) = 0, \quad \lim_{\delta \to 0} \eta(\delta) = 0 .
\ee 
We apply this in the case $f = f_i$ for $i = 1,2$, where
\be
(\beta_T)_l= f_1+f_2\ \ {\rm{or}}\ \ u_l(0) = f_1 + f_2, \quad f_1 \in \begin{cases} L^{s'}_{x,v} \text{ if } t = T \\ L^{q'}_{x,v} \text{ if } t = 0 \end{cases} \quad f_2 \in L^\infty_{x,v} .
\ee
Consequently, for $t = 0,T$,
\be
\int_{\M \times \R^d} \tilde m_0^{(R)} u_l(0) \dd x \dd v \to  \int_{\M \times \R^d} m_0 u_l(0) \dd x \dd v,
\ee
and
\be
\int_{\M \times \R^d} \tilde m_T^{(R)} (\b_T)_l \dd x \dd v \to  \int_{\M \times \R^d} m_T (\b_T)_l \dd x \dd v,
\ee
as $\delta \to 0$ with $\eta, R, \e$ chosen to depend on $\delta$ in the manner specified.

Finally, we take the limit $l \to - \infty$. For the term, $t=0$, convergence holds by monotonicity, and the limit is finite since
\be
- \int_{\M \times \R^d} u_0 m_0 \dd x \dd v < + \infty
\ee
by finiteness of $\cA(u,\beta,\b_T)$.
For the term $t = T$, we first note that
\be
\int_{\M \times \R^d} (\b_T)_l m_T(x,v) \dd x \dd v \leq \int_{\M \times \R^d} \cG(m_T) \dd x \dd v + \int_{\M \times \R^d} \cG^\ast((\b_T)_l) \dd x \dd v .
\ee
The second term on the right hand side converges due to the assumption on $\cG$ \eqref{hyp:G-growth}, since the integrand is dominated by
\be
\sup_{u < 0} \cG^\ast(u) \leq C_G \in L^1(\M \times \R^d) .
\ee

{\bf Term involving $\beta m$.}
Since $m \in L^1 \cap L^q$, by standard results on approximation by mollification in $L^p$ spaces we have
\be
\lim_{(\eta, \delta, \e) \to 0, R \to +\infty} \| \mt^{(R)} - m \|_{L^1 \cap L^q} = 0 ,
\ee
and thus the same limit holds with $\eta, \e, R$ chosen to depend on $\delta$ as described above.
Then, since $\beta_k \in L^\infty + L^{q'}$, we deduce that
\be
\lim_{(\delta, \e) \to 0, R \to +\infty}  \int_0^T \int_{\M \times \R^d} \beta_k \mt^{(R)}  \one_{\{u > l\}}  \dd x \dd v \dd t =  \int_0^T \int_{\M \times \R^d} \beta_k m  \one_{\{u > l\}}  \dd x \dd v \dd t .
\ee
By the definition of Fenchel conjugate,
\be
 \int_0^T \int_{\M \times \R^d} \beta_k m  \one_{\{u > l\}}  \dd x \dd v \dd t \leq  \int_0^T \int_{\M \times \R^d} \left ( \cF^\ast(x,v,\beta_k) + \cF(x,v, m) \right )  \one_{\{u > l\}}  \dd x \dd v \dd t .
\ee

We then take the limit $l \to - \infty$. Note $\cF$ and $\cF^\ast$ are both lower bounded by integrable functions (conditions \eqref{hyp:F-growth} and \eqref{eq:Fconj-lb}). Then, by monotonicity,
\begin{align*}
\lim_{l \to - \infty} \int_0^T \int_{\M \times \R^d} \left ( \cF(x,v, m) - \inf_m \cF(x,v,m) \right )  &\one_{\{u > l\}}  \dd x \dd v \dd t\\
& = \int_0^T \int_{\M \times \R^d} \left ( \cF(x,v, m) - \inf_m \cF(x,v,m) \right )   \dd x \dd v \dd t .
\end{align*}
Moreover
\be
\lim_{l \to - \infty} \int_0^T \int_{\M \times \R^d} \inf_m \cF(x,v,m) \one_{\{u > l\}}  \dd x \dd v \dd t = \int_0^T \int_{\M \times \R^d}  \inf_m \cF(x,v,m)   \dd x \dd v \dd t .
\ee
Since the lower bound is integrable and $\cF(\cdot,\cdot, m)$ has finite integral by finiteness of the energy, both of these limits are finite. Thus 
\be
\lim_{l \to -\infty} \int_0^T \int_{\M \times \R^d}  \cF( x,v,m)  \one_{\{u > l\}}  \dd x \dd v \dd t = \int_0^T \int_{\M \times \R^d}  \cF(x,v, m)  \dd x \dd v \dd t .
\ee
A similar argument shows that
\be
\lim_{l \to - - \infty} \int_0^T \int_{\M \times \R^d}   \cF^\ast(x,v,\beta_k)  \one_{\{u > l\}}  \dd x \dd v \dd t =  \int_0^T \int_{\M \times \R^d}   \cF^\ast(x,v,\beta_k)    \dd x \dd v \dd t 
\ee
where the right hand side is finite.

Finally, we consider $k \to - \infty$. Note that $\sup_{\beta \leq 0} \cF^\ast(x,v, \beta) \leq C_F(x,v) \in L^1$ by assumption (see \eqref{eq:Fconj-lb}). Since $\cF^\ast(\beta)$ is a continuous non-decreasing function of $\beta$, as $k$ decreases to negative infinity $\cF^\ast(\cdot,\cdot,\beta_k)$ is decreasing and converges almost everywhere to $\cF^\ast(\cdot,\cdot,\beta)$. Thus we deduce the convergence
\be
\lim_{k \to - \infty} \int_{\{\beta \leq 0 \}}  \cF^\ast(x,v,\beta_k)   \dd x \dd v \dd t = \int_{\{\beta \leq 0 \}}   \cF^\ast(x,v,\beta)    \dd x \dd v \dd t .
\ee
By the bounds \eqref{eq:Fconj-lb}, the right hand side is finite. Moreover, for any $k \leq 0$,
\be
\int_{\{\beta > 0 \}}  \cF^\ast(x,v,\beta_k)   \dd x \dd v \dd t = \int_{\{\beta > 0 \}}   \cF^\ast(x,v,\beta)    \dd x \dd v \dd t .
\ee
Thus we conclude that
\be
\lim_{k \to - \infty} \int_0^T \int_{\M \times \R^d} \cF^\ast(x,v,\beta_k)   \dd x \dd v \dd t = \int_0^T \int_{\M \times \R^d}  \cF^\ast(x,v,\beta)    \dd x \dd v \dd t .
\ee

{\bf Lagrangian Term.}
For the term involving the Lagrangian, we use a similar argument as was used in \cite{Cardaliaguet-Graber}. This argument is based on the joint convexity of $L(x,v,-w/m)m$ as a function of $(m,w)$. In our case we must additionally account for the convergence of the localizer $\zeta_R$.
By convexity, for all $(t,x,v)$, the integrand satisfies the inequality 
\begin{align*}
&L\left(x,v, - \frac{\wt}{\mt}\right) \mt \zeta_R \one_{\{u > l\}}\\ 
&\leq \int_0^T\int_{\M\times\R^d} \theta_\eta(t-t')\psi_\delta(x- x') \chi_\e(v - v') L\left(x,v, -\frac{w(t', x', v')}{m(t', x', v')} \right) m(t', x', v') \dd t' \dd x ' \dd v ' \, \zeta_R  \one_{\{u > l\}} \\
& = \theta_\eta\ast_t\psi_\delta \ast_{x} \chi_\e \ast_v \left[L\left(x,v, -\frac{w}{m}\right) m\right] \zeta_R \one_{\{u > l\}}.
\end{align*}
Then, note that
\begin{multline}
 \theta_\eta\ast_t\psi_\delta \ast_{x} \chi_\e \ast_v \left[L\left(x,v, -\frac{w}{m}\right) m\right](t,x,v) =  \theta_\eta\ast_t\psi_\delta \ast_{x} \chi_\e \ast_v \left[L\left( \cdot , -\frac{w}{m}\right) m\right] (t,x,v) \\
 +  \theta_\eta\ast_t\psi_\delta \ast_{x} \chi_\e \ast_v   \left [L\left(x,v, -\frac{w}{m}\right) m- L\left(\cdot, \cdot , -\frac{w}{m}\right) m \right ] (t,x,v) .
\end{multline}
Then, since $L(\cdot, \cdot , -\frac{w}{m}) m \in L^1((0,T)\times\M\times\R^d)$, $\theta_\eta\ast_t\psi_\delta \ast_{x} \chi_\e \ast_v [L(\cdot, \cdot , \frac{w}{m}) m] $ converges to $L(\cdot, \cdot , -\frac{w}{m}) m$ in $L^1((0,T)\times\M\times\R^d)$ as $\eta,\delta, \e$ tend to zero. Since $0 \leq \zeta_R \leq 1$,
\be
\lim_{\eta,\delta, \e \to 0} \sup_R \left\| \left \{ \theta_\eta\ast_t\psi_\delta \ast_{x} \chi_\e \ast_v \left[L\left(\cdot, \cdot , -\frac{w}{m}\right) m\right] - L\left(\cdot, \cdot , -\frac{w}{m}\right) m \right \} \zeta_R \one_{\{u > l\}} \right\|_{L^1} = 0 .
\ee
It follows that if we take the regime $R = R(\delta), \e = \e(\delta),\eta=\eta(\delta)$ established above, then
\be
\lim_{\delta\to 0} \left\| \left \{ \theta_\eta\ast_t\psi_\delta \ast_{x} \chi_\e \ast_v \left[L\left(\cdot, \cdot , -\frac{w}{m}\right) m\right] - L\left(\cdot, \cdot , -\frac{w}{m}\right) m \right \} \zeta_R \one_{\{u > l\}} \right\|_{L^1} = 0 .
\ee
We stay with this regime and consider the remaining term
\be
\int_0^T\int_{\M\times\R^d} \theta_\eta\ast_t\psi_\delta \ast_{x} \chi_\e \ast_v   \left \{L\left(x,v, -\frac{w}{m}\right) m- L\left(\cdot, \cdot , -\frac{w}{m}\right) m \right \} (t,x,v) \zeta_R(x,v) \one_{\{u > l\}} \dd x \dd v \dd t .
\ee
The integrand converges to zero almost everywhere: note then that
\be
\left| \theta_\eta\ast_t\psi_\delta \ast_{x} \chi_\e \ast_v   \left \{L\left(x,v, -\frac{w}{m}\right) m- L\left( \cdot,\cdot , -\frac{w}{m}\right) m \right \} \zeta_R \one_{\{u > l\}} \right| \leq \theta_\eta\ast_t\psi_\delta \ast_{x} \chi_\e \ast_v \left [ \left ( C + \frac{|w|^{r'}}{m^{r'}} \right ) m \right ] .
\ee
The right hand side converges in $L^1$ to $Cm + |w|^{r'} m^{1- r'}$ as $\delta$ tends to zero. Thus by dominated convergence we conclude that
\be
\lim_{\delta \to 0} \int_0^T \int_{\M \times \R^d} L \left (x,v, - \frac{\wt}{\mt} \right) \mt \zeta_R \one_{\{u > l\}}  \dd x \dd v \dd t \leq \int_0^T \int_{\M \times \R^d} L\left(x,v, - \frac{w}{m}\right) m \one_{\{u > l\}}  \dd x \dd v \dd t.
\ee
The reverse inequality follows from Fatou's lemma. Thus
\be
\lim_{\delta \to 0} \int_0^T \int_{\M \times \R^d} L \left (x,v, - \frac{\wt}{\mt} \right) \mt \zeta_R \one_{\{u > l\}}  \dd x \dd v \dd t = \int_0^T \int_{\M \times \R^d} L\left(x,v, - \frac{w}{m}\right) m \one_{\{u > l\}}  \dd x \dd v \dd t.
\ee

Finally, we take the limit $l \to - \infty$. Since $L\left(x,v, - \frac{w}{m}\right) m \in L^1$, in the limit we obtain
\be
\int_0^T \int_{\M \times \R^d} L\left(x,v, - \frac{w}{m}\right) m  \dd x \dd v \dd t .
\ee

{\bf Conclusion.}
From the discussion above, we have obtained
 \begin{align}
&\int_{\M \times \R^d} u_0 m_0 \dd x \dd v - \int_{\M \times \R^d} \cG(m_T) \dd x \dd v - \int_{\M \times \R^d} \cG^\ast(u_T) \dd x \dd v - \int_0^T \int_{\M \times \R^d} L \left (x,v, - \frac{w}{m} \right) m  \dd x \dd v \dd t \\
&\leq  \int_0^T \int_{\M \times \R^d} \left ( \cF^\ast(x,v,\beta) + \cF(x,v, m) \right )  \dd x \dd v \dd t  ,
\end{align}
where all terms are finite. Rearranging this inequality, we obtain the statement $\widetilde \cA(u,\beta) + \cB(m,w) \geq 0$.
\end{proof} 

\begin{corollary}\label{cor:energy_ineq}
Let $(u,\b,\b_T)\in\cK_\cA$ and $(m,w)\in\cK_\cB$ be such that $\widetilde\cA(u,\b,\b_T)<+\infty$ and $\cB(m,w)<+\infty$.

Then
$$\b_- m\in L^1((0,T)\times\M\times\R^d)\ \  {\rm{and}}\ \ (\b_T)_-m_T\in L^1(\M\times\R^d).$$
Moreover, for almost all $t \in [0,T]$,
\begin{align}\label{ineq:energy_limit-tT}
\int_{\M\times\R^d}\left[u_tm_t-\b_Tm_T\right]\dd x\dd v-\int_t^T\int_{\M\times\R^d}L\left(x,v,-\frac{w}{m}\right)m\dd x\dd v\dd t\le \int_t^T\int_{\M\times\R^d}\b m\dd x\dd v\dd t ,
\end{align}
and
\begin{align}\label{ineq:energy_limit-0t}
\int_{\M\times\R^d}\left[u_0m_0-u_tm_t\right]\dd x\dd v-\int_0^t\int_{\M\times\R^d}L\left(x,v,-\frac{w}{m}\right)m\dd x\dd v\dd t\le \int_0^t\int_{\M\times\R^d}\b m\dd x\dd v\dd t.
\end{align}

In particular,
\begin{align}\label{ineq:energy_limit}
\int_{\M\times\R^d}\left[u_0m_0-\b_Tm_T\right]\dd x\dd v-\int_0^T\int_{\M\times\R^d}L\left(x,v,-\frac{w}{m}\right)m\dd x\dd v\dd t\le \int_0^T\int_{\M\times\R^d}\b m\dd x\dd v\dd t.
\end{align}
\end{corollary}

\begin{proof}
This is a consequence of the proof of Lemma \ref{lem:inequality} and in particular the inequality \eqref{ineq:energy-2}. First, let us show the first part of the statement, i.e. that $\b_- m\in L^1((0,T)\times\M\times\R^d)\ \  {\rm{and}}\ \ (\b_T)_-m_T\in L^1(\M\times\R^d)$.

As in the mentioned proof, let us first pass to the limit with $\eta,\d,\e,R$ in the inequality \eqref{ineq:energy-2}. Then, we pass to the limit as $l\to-\infty$ and $k\to-\infty$ the remaining terms.

All the terms, except the ones involving $((\b_T)_-)_l m_T$ and $ (\beta_-)_k m \one_{\{u > l\}}$ pass to the limit, as in the proof. 
After rearranging, we also find that both $\int_{\M \times \R^d} ((\b_T)_-)_l m_T \dd x \dd v$ and $\int_0^T \int_{\M \times \R^d} (\beta_-)_k m \one_{\{u > l\}}  \dd x \dd v \dd t$ are uniformly bounded, independently of $l$ and $k$. Therefore, the monotone convergence theorem yields
$$
\lim_{l\to-\infty}\int_{\M \times \R^d} ((\b_T)_-)_l m_T \dd x \dd v =  \int_{\M \times \R^d} (\b_T)_- m_T \dd x \dd v
$$
and
$$
\lim_{l\to-\infty}\lim_{k\to-\infty}\int_0^T \int_{\M \times \R^d} (\beta_-)_k m \one_{\{u > l\}}  \dd x \dd v \dd t = \int_0^T \int_{\M \times \R^d} \beta_- m  \dd x \dd v \dd t.
$$
The summability results follow, and so does \eqref{ineq:energy_limit}.

For the case of general $t\in[0,T]$, we begin by testing the equation to find that, for example,
 \begin{align}
&\int_{\M \times \R^d} u_l(t,x,v) \mt^{(R)}(t,x,v) \dd x \dd v - \int_{\M \times \R^d} (\b_T)_l \mt^{(R)}(T,x,v) \dd x \dd v \\ 
&- \int_t^T \int_{\M \times \R^d} L(x,v, - \at) \mt^{(R)} \one_{\{u > l\}}  \dd x \dd v \dd t \\
&\leq \int_t^T \int_{\M \times \R^d} \beta_k \mt^{(R)}  \one_{\{u > l\}}  \dd x \dd v \dd t  - \int_t^T \int_{\M \times \R^d} u_l \, \cE_{\eta, \delta, \e, R} \dd x \dd v \dd t  .
\end{align}
We note that we are referring to the version of $u_l$ that is weakly right continuous in time (see Appendix~\ref{app:time_regularity}). The only term that requires attention is
\be
\int_{\M \times \R^d} u_l(t,x,v) \mt^{(R)}(t,x,v) \dd x \dd v .
\ee
For almost all $t\in[0,T]$, $u_l(t) \in (L^{q'} + L^\infty)(\M\times\R^d)$ and $m_t \in L^1 \cap L^q$. Thus the arguments for the boundary terms in Lemma~\ref{lem:inequality} show that
\be
\lim_{\delta \to 0} \int_{\M \times \R^d} u_l(t,x,v) \mt^{(R)}(t,x,v) \dd x \dd v = \int_{\M \times \R^d} u_l(t,x,v) m(t,x,v) \dd x \dd v .
\ee
The limit $l \to - \infty$ is then taken by monotone convergence, noting that $(u_t)_+ m_t \in L^1(\M\times\R^d)$. Note that the argument for the case \eqref{ineq:energy_limit-0t} shows that limit is not negative infinity, since all other terms have finite limits.

\end{proof}

\begin{corollary}\label{cor:energy_bound}
Let $(u,\b,\b_T)\in\cK_\cA$ and $(m,w)\in\cK_\cB$ be such that $\widetilde\cA(u,\b,\b_T)<+\infty$ and $\cB(m,w)<+\infty$. Then
\begin{enumerate}
\item $m|D_v u|^{r}$ is uniformly bounded in $L^1((0,T)\times\M\times\R^d)$, by a constant that depends only on the data and $\widetilde\cA(u,\b,\b_T)$ and $\cB(m,w)$. 
\item The following estimate holds:
\begin{align}
&\int_0^T \int_{\M \times \R^d} m H(x,v, D_v u) + m L \left (x,v, - \frac{w}{m} \right ) + D_v u \cdot w \, \dd x \dd v \dd t\\
& \qquad \le \int_{\M \times \R^d} \b_Tm_T \dd x \dd v - \int_{\M \times \R^d} u_0 m_0 \dd x \dd v + \int_0^T \int_{\M \times \R^d} \beta m \dd x \dd v \dd t \\
& \qquad + \int_0^T \int_{\M \times \R^d} m L \left (x,v, - \frac{w}{m} \right )  \dd x \dd v \dd t . 
\end{align} 
\end{enumerate}
\end{corollary}

\begin{proof}
This is a consequence of \eqref{ineq:test_error}. Using the same notation as in the proof of Lemma \ref{lem:inequality}, we rewrite \eqref{ineq:test_error} as
\begin{align}
&\int_0^T \int_{\M \times \R^d} \mt^{(R)} H(x,v, D_v u)  \one_{\{u > l\}}  \dd x \dd v \dd t\\
&\le \int_{\M \times \R^d} (\b_T)_l \mt^{(R)}(T,x,v) \dd x \dd v - \int_{\M \times \R^d} u_l(0,x,v) \mt^{(R)}(0,x,v) \dd x \dd v\\ 
&+ \int_0^T \int_{\M \times \R^d} \beta_k \mt^{(R)}  \one_{\{u > l\}}  \dd x \dd v \dd t -  \int_0^T \int_{\M \times \R^d} D_v u_l \cdot \wt^{(R)}\one_{\{u > l\}}  \dd x \dd v \dd t\\
&- \int_0^T \int_{\M \times \R^d} u_l \, \cE_{\eta,\delta, \e, R} \dd x \dd v \dd t \\
&\le \int_{\M \times \R^d} (\b_T)_+ \mt^{(R)}(T,x,v) \dd x \dd v + \int_{\M \times \R^d} [u_l(0,x,v)]_- \mt^{(R)}(0,x,v) \dd x \dd v\\ 
&+ \int_0^T \int_{\M \times \R^d} \beta_+ \mt^{(R)} \dd x \dd v \dd t -  \int_0^T \int_{\M \times \R^d} D_v u_l \cdot \wt^{(R)}\one_{\{u > l\}}  \dd x \dd v \dd t\\
&- \int_0^T \int_{\M \times \R^d} u_l \, \cE_{\eta,\delta, \e, R} \dd x \dd v \dd t.
\end{align}
Let us observe that for some $\theta>0$ parameter that we choose later, Young's inequality yields
\begin{align*}
&-  \int_0^T \int_{\M \times \R^d} D_v u_l \cdot \wt^{(R)}\one_{\{u > l\}}  \dd x \dd v \dd t\\ 
&= -  \int_0^T \int_{\M \times \R^d} \theta (\mt^{(R)})^{\frac{1}{r}}D_v u_l \cdot \wt^{(R)}\theta^{-1}(\mt^{(R)})^{-\frac{1}{r}}\one_{\{u > l\}}  \dd x \dd v \dd t\\
&\le\frac{1}{r}\theta^{r}\int_0^T \int_{\M \times \R^d}\mt^{(R)}|D_v u_l|^r \one_{\{u > l\}}  \dd x \dd v \dd t + \frac{1}{r'}\theta^{-r'} \int_0^T \int_{\M \times \R^d}|\wt^{(R)}|^{r'}(\mt^{(R)})^{-\frac{r'}{r}}\one_{\{u > l\}}  \dd x \dd v \dd t.
\end{align*}
We notice that $-\frac{r'}{r}=1-r'$. Thus, by using the growth condition \eqref{hyp:H} on the Hamiltonian and choosing $\theta$ appropriately, we can conclude that there exists a constant $C>0$ (independent of the parameters $l,k,R,\eta,\e,\d$), such that after passing to the limit with $R,\eta,\e,\d$, as in the proof of Lemma \ref{lem:inequality}, we obtain
\begin{align*}
\int_0^T \int_{\M \times \R^d} m |D_v u_l|^r \one_{\{u > l\}}  \dd x \dd v \dd t&\le C +\int_{\M \times \R^d} (\b_T)_+ m(T,x,v) \dd x \dd v\\
&+ \int_{\M \times \R^d} [u_l(0,x,v)]_- m_0(x,v)\dd x \dd v\\ 
&+ \int_0^T \int_{\M \times \R^d} \beta_+ m \dd x \dd v \dd t + C \int_0^T \int_{\M \times \R^d} |w|^{r'}m^{1-r'}  \dd x \dd v \dd t.
\end{align*}
Since the right hand side of this inequality is uniformly bounded, independently of $l$ (by Lemma \ref{lem:value-ub}(ii), Corollary \ref{rem:u_0}(ii) and Remark \ref{rem:min_B}), the result follows by Fatou's lemma by sending $l\to-\infty$.

Using \eqref{ineq:test_error}, we have
\begin{align}
&\int_0^T \int_{\M \times \R^d} \mt^{(R)} H(x,v, D_v u)  \one_{\{u > l\}}  \dd x \dd v \dd t + \int_0^T \int_{\M \times \R^d} \mt^{(R)} L \left (x,v, - \frac{\wt}{\mt} \right )  \one_{\{u > l\}}  \dd x \dd v \dd t \\
&+  \int_0^T \int_{\M \times \R^d} D_v u \cdot \wt^{(R)}\one_{\{u > l\}}  \dd x \dd v \dd t\\
& \qquad \le \int_{\M \times \R^d} (\b_T)_l \mt^{(R)}(T,x,v) \dd x \dd v - \int_{\M \times \R^d} u_l(0,x,v) \mt^{(R)}(0,x,v) \dd x \dd v\\ 
& \qquad + \int_0^T \int_{\M \times \R^d} \beta_k \mt^{(R)}  \one_{\{u > l\}}  \dd x \dd v \dd t + \int_0^T \int_{\M \times \R^d} \mt^{(R)} L \left (x,v, - \frac{\wt}{\mt} \right )  \one_{\{u > l\}}  \dd x \dd v \dd t\\
& \qquad - \int_0^T \int_{\M \times \R^d} u_l \, \cE_{\eta,\delta, \e, R} \dd x \dd v \dd t .
\end{align}
Passing to the limit with $R,\eta,\e,\d$, as in the proof of Lemma \ref{lem:inequality}, by Fatou's lemma we obtain
\begin{align}
&\int_0^T \int_{\M \times \R^d} m H(x,v, D_v u)  \one_{\{u > l\}}  \dd x \dd v \dd t + \int_0^T \int_{\M \times \R^d} m L \left (x,v, - \frac{w}{m} \right )  \one_{\{u > l\}}  \dd x \dd v \dd t \\
&+  \int_0^T \int_{\M \times \R^d} D_v u \cdot w \one_{\{u > l\}}  \dd x \dd v \dd t\\
& \qquad \le \int_{\M \times \R^d} (\b_T)_l m(T,x,v) \dd x \dd v - \int_{\M \times \R^d} u_l(0,x,v) m(0,x,v) \dd x \dd v\\ 
& \qquad + \int_0^T \int_{\M \times \R^d} \beta_k m \one_{\{u > l\}}  \dd x \dd v \dd t + \int_0^T \int_{\M \times \R^d} m L \left (x,v, - \frac{w}{m} \right )  \one_{\{u > l\}}  \dd x \dd v \dd t .
\end{align}
Finally, taking the limit $l, k \to -\infty$ as in the proofs of Lemma~\ref{lem:inequality} and Corollary \ref{cor:energy_ineq}, we obtain
\begin{align}
&\int_0^T \int_{\M \times \R^d} m H(x,v, D_v u) + m L \left (x,v, - \frac{w}{m} \right ) + D_v u \cdot w \, \dd x \dd v \dd t\\
& \qquad \le \int_{\M \times \R^d} \b_Tm_T \dd x \dd v - \int_{\M \times \R^d} u_0 m_0 \dd x \dd v + \int_0^T \int_{\M \times \R^d} \beta m \dd x \dd v \dd t \\
& \qquad + \int_0^T \int_{\M \times \R^d} m L \left (x,v, - \frac{w}{m} \right )  \dd x \dd v \dd t .
\end{align}

\end{proof}

\section{Existence of a Solution of the Relaxed Problem}\label{sec:5}

In this section we prove the existence of a solution for the relaxed problem. Consider a minimizing sequence $(u_n, \beta_n,(\beta_T)_n)_n$.
We will extract a convergent subsequence, and show that the limit constitutes a minimizer of the objective functional.

\subsection{Compactness of the Minimizing Sequence} \label{sec:compactness}

The following proposition enables the extraction of a convergent subsequence.

\begin{proposition} \label{prop:HJ-compactness}

Let $(u_n)_n$ be a sequence of solutions to the Hamilton-Jacobi equations
 \be \label{eq:HJ-n}
-\partial_t u_n - v\cdot D_x u_n+H(x,v,D_v u_n) = \beta_n, 
\ee
Assume that:
\begin{itemize}
\item The family $(u_n)_n$ is uniformly bounded in $L^\infty_tL^1_{x,v,\loc}(\cU_{m_0})$.
\item The family $(D_v u_n)_n$ is uniformly bounded in $L^r_{\loc}(\cU_{m_0})$.
\item The family $(\beta_n)_n$ is uniformly bounded in $L^1_{\loc}(\cU_{m_0})$.
\end{itemize}

Then there exists a subsequence $(u_{n_j})_j$ that is strongly convergent in $L^1_{\loc}(\cU_{m_0})$.

\end{proposition}

The proof of this result will be a consequence of some intermediate results that we detail below.

\begin{remark}
Let us notice that the assumptions of Proposition \ref{prop:HJ-compactness} hold true by Corollary \ref{cor:bounds_u} and Lemma \ref{lem:value-ub}.
\end{remark}

Proposition~\ref{prop:HJ-compactness} is proved by treating the Hamilton-Jacobi equation
as a kinetic transport equation with right hand side bounded in $L^1_{\loc}(U_{m_0})$:
 \be 
-\partial_t u_n - v\cdot D_x u_n = \beta_n - H(x,v,D_v u_n). 
\ee
A form of compactness for the solutions can be obtained by using an \emph{averaging lemma}.
Averaging lemmas are results in kinetic theory showing that, for $L^p$-bounded families of solutions to the kinetic transport equation, with $L^p$-bounded source terms, the velocity averages
\be
\rho_\phi [u](t,x) : = \int_{\R^d} u(t,x,v) \phi(v) \dd v \qquad \phi \in C^\infty_c(\R^d),
\ee
enjoy additional fractional Sobolev regularity and/or strong $L^p$-compactness. 
In our case we are in the setting $p=1$, and we use an $L^1$ averaging result from \cite{GSR}. It is necessary to assume a certain equi-integrability condition on the solution $u_n$. This condition is defined below.

\begin{definition}[Equi-integrability in velocity]
Let $(u_\lambda)_{\lambda \in \Lambda}$ be a bounded family in $L^1_{\rm{loc}}([0,T] \times \M \times \R^d)$. The family is locally equi-integrable in $v$ if, for all $\e > 0$ and all compact sets $K \subset [0,T] \times \M \times \R^d$, there exists $\eta > 0$ such that for all measurable families $(A_{t,x})_{(t,x) \in [0,T] \times \M}$ of measurable subsets of $\R^d$ for which $\sup_{(t,x) \in [0,T] \times \M} |A_{t,x}| < \eta$,
\be
\int_0^T \int_{\M} \int_{A_{t,x}} \one_K |u_\lambda(t,x,v)| \dd v \dd x \dd t < \e \qquad \text {for all } \lambda \in \Lambda .
\ee

\end{definition}

The required averaging lemma is quoted below. This result was proved in \cite{GSR} for the stationary case, i.e. the equation $v \cdot D_xu = \beta$.
The result can be adapted to the time dependent equation by standard techniques; see \cite{GSR_NS} or \cite{Han-Kwan} for statements in the time dependent setting.

\begin{theorem} \label{thm:GSR}
Let $(u_\lambda)_{\lambda \in \Lambda}$ be a bounded family in $L^\infty([0,T]; L^1_\loc(\M \times \R^d))$ satisfying
\be \label{eq:kinetic-transport}
\partial_t  u_\lambda + v \cdot D_x u_\lambda = \beta_\lambda,
\ee
where $(\beta_\lambda)_{\lambda \in \Lambda}$ is a bounded family in $L^1_\loc([0,T] \times \M \times \R^d)$.
Assume that $(u_\lambda)_{\lambda \in \Lambda}$ is equi-integrable in $v$. Then 
\begin{itemize}
\item The family $(u_\lambda)_{\lambda \in \Lambda}$ is locally equi-integrable in all variables in $[0,T] \times \M \times \R^d$.
\item For each $\phi \in C_c(\R^d)$, the family of averages $(\rho_\phi[u_\lambda])_{\lambda \in \Lambda}$ is relatively (strongly) compact in $L^1_\loc([0,T] \times \M)$.
\end{itemize}
\end{theorem}

In our setting we expect to have local summability estimates in $\cU_{m_0}$ rather than $[0,T] \times \M \times \R^d$. To deal with this technicality we make use of a localisation procedure: given a compact set $K$, consider a smooth bump function $\zeta$ supported in $K$. If $u_\lambda$ satisfies the kinetic transport equation \eqref{eq:kinetic-transport}, then $u_\lambda \zeta$ satisfies
\be 
\partial_t  (u_\lambda \zeta) + v \cdot D_x (u_\lambda \zeta ) = \beta_\lambda \zeta + (\partial_t \zeta + v \cdot \nabla_x \zeta ) u_\lambda.
\ee
The right hand side of the above equation is bounded in $L^1([0,T] \times \M \times \R^d)$, uniformly in $\lambda$, as long as $u_\lambda$ and $\beta_\lambda$ are bounded in $L^1_\loc(\cU_{m_0})$.

We wish to apply this to the solutions of the Hamilton-Jacobi equation. To do this, we verify the equi-integrability condition. To prove equi-integrability, we make use of the $L^r$ estimates available for the $v$-derivative $D_v u$.

\begin{lemma}
Let $(u_\lambda)_{\lambda \in \Lambda}$ be a bounded family in $L^\infty_t L^1_{x,v, \loc}(\cU_{m_0})$.
Assume that $(D_v u_\lambda)_{\lambda \in \Lambda}$ is a bounded family in $L^r_{\loc}(\cU_{m_0})$.
Then:
\begin{itemize}
\item $(u_\lambda)_{\lambda \in \Lambda}$ is bounded in $L^1_{t,x}L^{\alpha}_{v,\loc}(\cU_{m_0})$, where $1/\alpha = 1/r - 1/d$ if $r < d$, or any $\alpha < + \infty$ if $r \geq d$.
\item $(u_\lambda)_{\lambda \in \Lambda}$ is equi-integrable in $v$, locally in $\cU_{m_0}$.
\end{itemize}
\end{lemma}
\begin{proof}

We obtain higher integrability in the velocity variable by using Sobolev embedding.
We first apply a localisation procedure. Given a compact set $K \subset \cU_{m_0}$, let $\zeta_K$ denote a smooth bump function with compact support contained in $\cU_{m_0}$, such that $\zeta_K$ takes values contained in $[0,1]$ and $\zeta_K \equiv 1$ on $K$. Then
\be
D_v (u_\lambda \zeta_K) = D_v u_\lambda \, \zeta_K + u_\lambda \, D_v \zeta_K .
\ee
Let $K'$ denote the support of $\zeta_K$. Then
\be
\| D_v (u_\lambda \zeta_K) \|_{L^1} \leq C(\zeta_K) \left ( \| D_v u_\lambda \|_{L^r(K')} + \| u_\lambda \|_{L^\infty_t L^1_{x,v} (K')} \right ),
\ee
thus $D_v (u_\lambda \zeta_K)$ is bounded in $L^1$, uniformly in $\lambda$. Moreover it is compactly supported.

We then apply Sobolev embedding in the $v$ variable. Letting $1^\ast : = \frac{d}{d-1}$, we have
\be
\| u_\lambda \|_{L^1_{t,x} L^{1^\ast}_v(K)} \leq \| u_\lambda \zeta_K \|_{L^1_{t,x} L^{1^\ast}_v} \leq C_d \| D_v (u_\lambda \zeta_K) \|_{L^1} \leq C(d,\zeta_K) \left ( \| D_v u_\lambda \|_{L^r(K')} + \| u_\lambda \|_{L^\infty_t L^1_{x,v} (K')} \right ) .
\ee
Thus $(u_\lambda)_\lambda$ is uniformly bounded in $L^1_{t,x} L^{1^\ast}_{v,\loc}$.

We now apply a bootstrap argument: it follows from the above that $D_v (u_\lambda \zeta_K)$ is bounded in $L^1_{t,x} L^{1^\ast}_{v}$, and thus $(u_\lambda)_\lambda$ is bounded in $L^1_{t,x} L^{\alpha}_{v,\loc}$ for $\alpha = \left ( 1 - 2/d\right )^{-1}$.
This process can be repeated until we obtain that $(u_\lambda)_\lambda$ is bounded in $L^1_{t,x} L^{r^\ast}_{v,\loc}$, for $1/r_* = 1/r - 1/d$, if $r < d$; otherwise we may obtain $L^1_{t,x} L^{\alpha}_{v,\loc}$ for any $\alpha < + \infty$.

We now prove local equi-integrability.
Let $(A_{t,x})_{(t,x) \in [0,T] \times \M}$ be a measurable family of measurable subsets of $\R^d$, such that $|A_{t,x}| < \eta$ for all $t,x$. Then
\be
\int_0^T \int_{\M} \int_{A_{t,x}} \one_K |u_\lambda | \dd v \dd x \dd t \leq \| \mathbbm{1}_{A_{t,x}}  \mathbbm{1}_K \|_{L^{\infty}_{t,x} L^{\alpha'}_v} \| u_\lambda \one_K \|_{L^{1}_{t,x} L^{\alpha }_v},
\ee
where $'$ denotes a H\"older conjugate exponent. From the condition on the measure of $A_{t,x}$, we have
\be
\int_0^T \int_{\M} \int_{A_{t,x}} \one_K |u_\lambda(t,x,v)| \dd v \dd x \dd t \leq C(K) \sup_{t,x} |A_{t,x}|^{1/\alpha'} \| u_\lambda \|_{L^{r}_{t,x} L^{\alpha }_v (K)}  \leq C(K)  \| u_\lambda \|_{L^{1}_{t,x} L^{\alpha }_v (K)} \, \eta^{1/\alpha'} ,
\ee
which proves equi-integrability.

\end{proof}

It follows that, under the assumptions of Proposition~\ref{prop:HJ-compactness}, Theorem~\ref{thm:GSR} can be applied to $(\zeta_K u_n)_n$ for any $\zeta_K \in C^\infty_c(\cU_{m_0})$. We deduce strong $L^1_\loc$-compactness for the averages $(\rho_\phi[u_n \zeta_K])_n$.
We now use this to prove strong compactness for the full solutions $u_n$.

\begin{lemma} \label{lem:full-strong-conv}
Assume that the family $(u_n)_n$ satisfies the following:
\begin{itemize}
\item $(u_n)_n$ is uniformly bounded in $L^\infty_tL^1_{x,v}$
\item $(u_n)_{n}$ is equi-integrable in all variables.
\item $(u_n)_{n}$ share the same compact support $K$.
\item $(D_v u_n)_n$ is uniformly bounded in $L^1_{t,x}L^r_v$.
\item For each $\phi \in C^\infty_c(\R^d)$, the family of averages $(\rho_\phi[u_n])_{n}$ is relatively (strongly) compact in $L^1_\loc$.
\end{itemize}
Then the family $(u_n)_n$ is relatively (strongly) compact in $L^1$.
\end{lemma}
\begin{proof}
First, note that the first two assumptions imply the weak $L^1$ sequential compactness of $(u_n)_n$. We pass (without relabelling) to a weakly convergent subsequence $u_n$, and let $u$ denote the weak limit.
In the remainder of the proof we improve the mode of convergence of $u_n$ to $u$ to strong convergence in $L^1$.

{\bf Step 1: Approximation by smoothing in $v$.}
We approximate $u_n$ by a function that is smooth with respect to the $v$ variable.
Fix $\phi \in C^\infty_c(\R^d)$ and define, for $\e > 0$,
\be
\phi_\e (v) : = \e^{-d} \phi \left ( \frac{v}{\e} \right ).
\ee
Let
\be
u_{n, \e}(t,x,v) : = \int_{\R^d} u_n(t,x, v') \phi_\e (v - v' ) \dd v' .
\ee

{\bf Step 2: Compactness for the approximations.}
Fix $v_* \in \R^d$. For any $\psi \in L^\infty_{t,x}$ with compact support we consider testing the sequence $(u_n)_n$ against the test function
\be
\psi(t,x) \phi_\e(v_* - v) .
\ee
Since
\be
\langle u_n, \psi \phi_\e(v_* - \cdot) \rangle \to \langle u, \psi \phi_\e(v_* - \cdot) \rangle,\ \ n\to+\infty,
\ee
we deduce that $u_{n,\e}(\cdot, \cdot, v_*)$ converges weakly in $L^1_{t,x,\loc}$ to $u \ast_v \phi_\e(\cdot, \cdot, v_*)$ as $n\to+\infty$.

Note moreover for each fixed $v \in \R^d$, $u_{n,\e}(t,x,v)$ is a velocity average with respect to the test function $\phi_\e(v - \cdot)$.
Therefore, by Theorem \ref{thm:GSR} the convergence in fact holds strongly in $L^1_{t,x, \loc}$ for each $v \in \R^d$.

Furthermore, for fixed $\e > 0$, the family $(u_{n, \e})_n$ is equi-continuous in $v$ into $L^\infty_tL^1_x$: indeed
\begin{align}
|u_{n, \e}(t,x,v+h) - u_{n, \e}(t,x,v)| &= \left | \int_{\R^d} u_n(t,x,v') \left ( \phi_\e(v - v' + h) - \phi_\e(v- v') \right ) \dd v' \right | \\
& \leq |h| \| \nabla \phi_\e \|_{L^\infty} \int_{\R^d} |u_n(t,x,v')| \dd v ' .
\end{align}
Thus
\be
\| u_{n, \e}(\cdot,\cdot,v+h) - u_{n, \e}(\cdot,\cdot,v) \|_{L^\infty_t L^1_x} \leq C_\e \sup_{m} \| u_m \|_{L^\infty_tL^1_{x,v}} \, |h| .
\ee
By an Arzel\`{a}-Ascoli argument the convergence therefore holds locally uniformly in $v$, with respect to the strong topology on $L^1_{t,x,\loc}$: that is, for all compact sets $K_v \subset \R^d$ and $K_{t,x} \subset [0,T] \times \R^d$,
\be
\lim_{n \to \infty } \sup_{v \in K_v} \| u_{n, \e}(v) - u \ast \phi_\e(v) \|_{L^1_{t,x}(K_{t,x})}  = 0 .
\ee
Consequently, the convergence holds in $L^1_\loc$; in fact, since $u_{n, \e} - u \ast \phi_\e$ is supported for all $n$ in $K + B_{C\e}(0)$, the convergence holds in $L^1$.

{\bf Step 3: Removing the approximation.}

The bound on $D_v u_n$ implies that, for any $h \in \R^d$,
\be
\| u_n(t,x, \cdot + h) - u_n(t,x, \cdot) \|_{L^r_v} \leq |h| \| D_v u_n(t,x,\cdot) \|_{L^r_v} .
\ee
It follows that
\be
\| u_{n, \e} - u_n \|_{L^1_{t,x} L^r_v} \lesssim \e \sup_m \| D_v u_m \|_{L^1_{t,x} L^r_v} .
\ee

Indeed, by definition of $u_{n,\e}$,
\be
[u_{n, \e} - u_n](t,x,v) = \int_{\R^d} [u_n(t,x,v-h) - u_n(t,x,v)] \phi_\e (h) \dd h .
\ee
Thus, for any $g \in L^{r'}_v(\R^d)$,
\begin{align}
\int_{\R^d} [u_{n, \e} - u_n](t,x,v) g(v) \dd v &= \int_{\R^d} \int_{\R^d} [u_n(t,x,v-h) - u_n(t,x,v)] \phi_\e (h) g(v) \dd h \dd v \\
& \leq \|g \|_{L^{r'}_v} \int_{\R^d} \| u_n(t,x, \cdot + h) - u_n(t,x, \cdot) \|_{L^r_v} | \phi_\e (h)| \dd h \\
& \leq \|g \|_{L^{r'}_v} \| D_v u_n(t,x,\cdot) \|_{L^r_v} \int_{\R^d} |h| | \phi_1 \left ( \frac{h}{\e}\right)| \e^{-d} \dd h \\
& \leq C_\phi \e  \| D_v u_n(t,x,\cdot) \|_{L^r_v}  \|g \|_{L^{r'}_v} .
\end{align}
That is,
\be
\| u_{n, \e}(t,x,\cdot) - u_n(t,x,\cdot) \|_{L^r_v} \lesssim \e \| D_v u_n(t,x,\cdot) \|_{L^r_v},
\ee
then, one integrates in $t,x$ and takes supremum.

Finally, estimate
\begin{align}
\| u_n  - u \|_{L^1_{t,x,v}} & \leq \| u_n  - u_{n, \e} \|_{L^1_{t,x,v}} + \| u_{n, \e}  - u \ast _v \phi_{\e} \|_{L^1_{t,x,v} } + \| u \ast _v \phi_{\e}  - u \|_{L^1_{t,x,v}} \\
& \leq C_K \| u_n  - u_{n, \e} \|_{L^1_{t,x} L^r_v} + \| u_{n, \e}  - u \ast _v \phi_{\e} \|_{L^1_{t,x,v}} + C_K \| u \ast _v \phi_{\e}  - u \|_{L^1_{t,x} L^r_v} \\
& \lesssim C_K \e \sup_m \| D_v u_m \|_{L^1_{t,x} L^r_v} +  \| u_{n, \e}  - u \ast _v \phi_{\e} \|_{L^1_{t,x,v} } + \| u \ast _v \phi_{\e}  - u \|_{L^1_{t,x} L^r_v} .
\end{align}

Thus
\be
\limsup_{n \to \infty} \| u_n  - u \|_{L^1_{t,x,v}} \lesssim_K \e \sup_m \| D_v u_m \|_{L^1_{t,x} L^r_v} + \| u \ast _v \phi_{\e}  - u \|_{L^1_{t,x} L^r_v} \to 0
\ee
as $\e \to 0$, which completes the proof.
\end{proof}

\begin{proof}[Proof of Proposition \ref{prop:HJ-compactness}]
The proof of this proposition follows by applying the Lemma \ref{lem:full-strong-conv} to $(\zeta_K u_n)_n$.
\end{proof}

It remains to obtain the necessary convergence of $(\beta_n)_n$ and $(\b_{T,n})_n$. 

\begin{lemma}\label{lem:modif}
Let $(u_n, \beta_n,\beta_{T,n})$ be a minimizing sequence for Problem \ref{prob:relaxed}. There exists a modification $(u_n, \tilde \beta_n,\tilde \beta_{T,n})$ of this sequence that is also minimizing such that $(\tilde \beta_n)_n$ is weakly precompact in $L^1_\loc(\cU_{m_0})$ and $(\tilde\b_{T,n})_n$ is weakly precompact in $L^1_{\loc}(\M\times\R^d)$.
\end{lemma}

\begin{proof}
We replace $\beta_n$ with some $\tilde \beta_n \geq \beta_n$ and $(\b_{T,n})_n$ with $\tilde \b_{T,n}\ge \b_{T,n}$ such that $(\tilde \beta_n,\tilde \beta_{T,n})$ is uniformly integrable, and $(u_n, \tilde \beta_n,\tilde \beta_{T,n})$ is still a minimizing sequence. We do this in a similar manner to \cite{Cardaliaguet-Graber}: since $(\beta_n)_-$ is bounded in $L^1_\loc(\cU_{m_0})$, using a compact exhaustion of $\cU_{m_0}$ and a diagonal argument, by \cite{Kosygina-Varadhan} it is possible to pass to a subsequence such that the following holds for some $J_n \in \R$. We define $\tilde \beta_n$ by
\be(\tilde \beta_n)_- : = (\beta_n)_- \one_{\{ (\beta_n)_- \leq J_n \}} , \quad (\tilde \beta_n)_+ = (\beta_n)_+
\ee
Then it is possible to choose $J_n$ in such a way that:
\begin{itemize}
\item For each compact set $K \subset \cU_{m_0}$, the sequence $(\tilde \beta_n)_- \one_K $ is uniformly integrable.
\item The measure of the set $\{ (\beta_n)_- > J_n \} \cap K$ converges to zero as $n$ tends to infinity.
\end{itemize}
{We use the exact same construction for $\tilde\b_{T,n}$, and we can get the same properties (now taking $K\subset \M\times\R^d$).}

We notice, that by construction the constraints
\be
-\partial_t u_n - v\cdot D_x u_n+H(x,v,D_v u_n) \leq \tilde \beta_n
\ee
and 
$$u_{T,n}\le \tilde \b_{T,n}$$
are still satisfied.
Finally, 
\begin{align*}
&\left | \int_0^T\int_{\M\times\R^d} \cF^\ast(x,v,\tilde \beta_n) \dd x \dd v \dd t - \int_0^T\int_{\M\times\R^d} \cF^\ast(x,v,\beta_n) \dd x \dd v \dd t \right |\\ 
&\leq \int_0^T\int_{\M\times\R^d} |\cF^\ast(x,v,0) - \cF^\ast(x,v,\beta_n)| \one_{\{\beta_n \leq - J_n\}} \dd x \dd v \dd t .
\end{align*}
By the estimate \eqref{eq:Fconj-lb}, the integrand on the right hand side is dominated by $2 C_F \in L^1$, and thus the right hand side converges to zero as $n$ tends to infinity. 

The exact same arguments apply to $\cG^\ast$ and $\tilde \b_{T,n}$ too.
Thus $(u_n, \tilde \beta_n,\tilde\b_{T,n})$ is a minimizing sequence. Moreover, there exists $(u,\beta,\beta_T)$ such that up to passing to a subsequence, $(u_n)_n$ converges to $u$ strongly in $L^1_\loc(\cU_{m_0})$, $(\tilde \beta_n)_n$ converges weakly to $\beta$ in $L^1_\loc(\cU_{m_0})$ and $(\tilde\b_{T,n})_n$ converges to $\b_T$ weakly in $L^1_{\loc}(\M\times\R^d)$.

\end{proof}

\subsection{Existence of a minimizer of $\widetilde{\cA}$ over $\cK_\cA$}

In this subsection, we prove that there exists a minimizer $(u,\beta, \b_T)$ by passing to the limit in the functional
\begin{align*}
\widetilde{\cA}(u_n, \tilde \beta_n,\tilde\b_{T,n}) : = \int_0^T\int_{\M\times\R^d}\cF^*(x,v, \tilde \beta_n)\dd x\dd v\dd t& - \int_{\M\times\R^d}u_n(0,x,v)m_0(x,v)\dd x\dd v  \\ 
&+ \int_{\M\times\R^d}\cG^*(\tilde\b_{T,n}(x,v)) \dd x \dd v .
\end{align*}

\begin{theorem}\label{thm:existence}
Under our standing assumptions, the functional $\widetilde{\cA}$ admits a minimizer over $\cK_\cA$.
\end{theorem}

\begin{proof}
Let $(u_n,\b_n,\b_{T,n})_{n\in\N}$ be a minimizing sequence. Without loss of generality, for example by considering $u_n \in C^1$, we may assume equality in the Hamilton-Jacobi equation:
\be
-\partial_t u_n - v\cdot D_x u_n +H(x,v,D_v u_n ) = \beta_n , \quad u_{n}(T,x,v) = \b_{T,n}(x,v) .
\ee
For this minimizing sequence we have, for some constant $C>0$,
\be
\sup_n \widetilde{\cA}(u_n , \beta_n, \b_{T,n}) \leq C .
\ee
We have discussed that this implies uniform in $n$ bounds on the following quantities:
\begin{align}
& \| (\b_{T,n})_+ \|_{L^{s'}(\M \times \R^d)}, \; \int_{\M \times \R^d} (u_{0,n})_- m_0 \dd x \dd v , \\
& \| (\b_{T,n})_- \|_{ L^1_{x,v,\loc}(\M\times\R^d)} , \; \| (u_{0,n} )_+ \|_{(L^\infty + L^{q'})_{x,v}},  \\
& \| D_v u_n \|_{L^r_\loc(\cU_{m_0})}, 
\| (\beta_n)_+ \|_{L^{q'}_{t,x,v}}, \; \| (\beta_n)_- \|_{L^1_\loc(\cU_{m_0})} .
\end{align}
To get the uniform integrability on  $(\beta_n)_-$ and $(\b_{T,n})_-$, we perform the surgery argument as in Lemma \ref{lem:modif}. So, let $(u_n,\tilde\b_n,\tilde\b_{T,n})_{n\in\N}$ be the modification of the minimizing sequence (which will still have uniformly bounded energy). By Proposition \ref{prop:HJ-compactness} we know that $(u_n)_{n\in\N}$ is strongly precompact in $L^1_{\rm{loc}}(\cU_{m_0})$, while Lemma \ref{lem:modif} yields that $(\tilde\b_n)_{n\in\N}$ and $(\tilde\b_{T,n})_{n\in\N}$ are weakly precompact in $L^1_{\rm{loc}}(\cU_{m_0})$ and $L^1_{\rm{loc}}(\M\times\R^d)$, respectively. In particular, after passing to a subsequence let us denote by $u$ the strong $L^1_\loc(\cU_{m_0})$ limit of $(u_n)_n$. In what follows, to ease the notation, we drop the tilde symbol, but whenever we write $\b_n$ and $\b_{T,n}$, we mean  the corresponding modified versions.

Passing to further subsequences (that we do not relabel), there exist limit functions so that we may also assume the following weak convergences:
\begin{itemize}
\item $(\beta_n)_+ \rightharpoonup \beta_+$, weakly in $L^{q'}_{t,x,v}([0,T]\times\M\times\R^d)$, as $n\to+\infty$.
\item $\beta_n \weakly \beta$, weakly in $L^1_{\rm{loc}}(\cU_{m_0})$, as $n\to+\infty$.
\item $(\beta_{T,n})_+ \rightharpoonup (\b_T)_+$, weakly in $L^{s'}(\M \times \R^d)$, as $n\to+\infty$.
\item $\b_{n,T} \weakly \b_T$,  weakly in $L^1_{\rm{loc}}(\M \times \R^d)$, as $n\to+\infty$.
\item $D_v u_n \rightharpoonup D_v u$, weakly in $L^r_\loc(\cU_{m_0})$, as $n\to+\infty$.
\end{itemize}

\medskip

With these convergences in hand, we are ready to pass to the limit in the Hamilton-Jacobi inequality constraint and the functional. Note that the weak form of the inequality (Definition \ref{def:solution}) implies that, for all $n$ and all test functions $\phi \in C^\infty_c((0,T] \times \M \times \R^d)$ such that $\phi \geq 0$,

\begin{align}\label{ineq:test_exist}
\int_0^T\int_{\M \times \R^d}(\partial_t \phi + v \cdot D_x \phi) u_n &+ \phi H(x,v, D_v u_n) \dd x \dd v \dd t  \leq 
\int_0^T \int_{\M\times\R^d} \phi \beta_n \dd x \dd v \dd t\\
& \nonumber+ \int_{\M\times\R^d} \phi(T,x,v) \b_{T,n}  \dd x \dd v .
\end{align}

Note again that $\phi$ is compactly supported in $\cU_{m_0}$.  By the weak convergence of $D_v u_n$ in $L^{r}_{\rm{loc}}(\cU_0)$ and the convexity of $H$ it follows that
\begin{align*}
&\int_0^T\int_{\M \times \R^d} \phi H(x,v, D_v u) \dd x \dd v \dd t = \int_{\cU_{m_0}} \phi H(x,v, D_v u) \dd x \dd v \dd t\\
& \leq \liminf_n\int_{\cU_{m_0}} \phi H(x,v, D_v u_n) \dd x \dd v \dd t =\liminf_n\int_0^T\int_{\M \times \R^d} \phi H(x,v, D_v u_n) \dd x \dd v \dd t.
\end{align*}
All the other convergences stated above are sufficient to guarantee convergence against $\phi$. So, we obtain that the limit $(u,\b,\b_T)$ satisfies \eqref{eq:weak_HJ_test}.

\medskip
\medskip

Next, we consider the convergence in the functional. In addition to the previous convergences, along the previously chosen subsequence, we have
\begin{itemize}
\item $(u_{0,n})_+ \weaklys (\overline u_0)_+$, weakly-$\ast$ in $(L^\infty + L^{q'})(\M \times \R^d)$, as $n\to+\infty$.
\end{itemize}

The convergence of the sequence $((u_{0,n})_-m_0)_n$ requires special attention. The boundedness of this sequence in $L^1(\M\times\R^d)$ lets us conclude that there exists a nonnegative Radon measure $\nu$ 
such that after passing to a subsequence (that we do not relabel)
$$
(u_{0,n})_-m_0\weaklys \nu,\ \ {\rm{as}}\ n\to+\infty.
$$
This means in particular that for all $\phi\in C_c(\M\times\R^d)$, we have
$$
\int_{\M\times\R^d}\phi (u_{0,n})_-m_0\dd x\dd v \to \int_{\M\times\R^d} \phi\nu(\dd x\dd v),\ \  {\rm{as}}\ n\to+\infty.
$$
Since the the sequence $((u_{0,n})_-m_0)_{n\in\N}$ is supported in the open set $\{m_0>0\}$, we get that $\spt(\nu)\subseteq\spt(m_0)$. Now, let us take $\phi\in C_c(\{m_0>0\})$ arbitrary and define $\psi:=\phi/m_0$. Since $m_0\in C(\M\times\R^d)$, by assumption, we have
that $\psi\in C_c(\{m_0>0\})$ and so
\begin{align*}
\int_{\M\times\R^d}\psi (u_{0,n})_-m_0\dd x\dd v&=\int_{\M\times\R^d}(\phi/m_0) (u_{0,n})_-m_0\dd x\dd v \\
&=\int_{\M\times\R^d}\phi (u_{0,n})_-\dd x\dd v \to \int_{\M\times\R^d} (\phi/m_0)\nu(\dd x\dd v),\ \  {\rm{as}}\ n\to+\infty.
\end{align*}
Thus, this means that as $n\to+\infty$, $(u_{n,0})_-$ converges weakly-$\ast$ to the nonnegative Radon measure $(\overline u_0)_-:=\frac{1}{m_0}\cdot\nu$, i.e. $(u_0)_-$ has density $\frac{1}{m_0}$ with respect to $\nu$. We notice that this means that $(\overline u_0)_-$ is absolutely continuous with respect to $\nu$. In fact, we also have that $\nu$ is absolutely continuous with respect to $(\overline u_0)_-$, and so we can write $\nu = m_0\cdot (\overline u_0)_-.$ 

\medskip

Let us take now $\phi\in C_c^\infty(\cU_{m_0})$, and test the inequalities satisfied by $(u_n,\b_n,\b_{T,n})$, similarly to \eqref{ineq:test_exist}, to obtain

\begin{align*}
\int_0^T\int_{\M \times \R^d}(\partial_t \phi + v \cdot D_x \phi) u_n &+ \phi H(x,v, D_v u_n) \dd x \dd v \dd t  \leq 
\int_0^T \int_{\M\times\R^d} \phi \beta_n \dd x \dd v \dd t\\
& \nonumber+ \int_{\M\times\R^d} \phi(T,x,v) \b_{T,n}  \dd x \dd v -  \int_{\M\times\R^d} \phi(0,x,v) u_{0,n}  \dd x \dd v.
\end{align*} 
Incorporating also the previously described convergence of $(u_{0,n})_n$, we can pass to the limit along the chosen subsequence and obtain
\begin{multline}
\int_0^T\int_{\M \times \R^d}(\partial_t \phi + v \cdot D_x \phi) u + \phi H(x,v, D_v u) \dd x \dd v \dd t \\ \leq \int_{\M\times\R^d} \b_T \phi_T\dd x\dd v - \int_{\M\times\R^d} \phi_0 {\overline u_0}(\dd x\dd v)  + \int_0^T \int_{\M\times\R^d} \phi \beta \dd x \dd v \dd t,
\end{multline}
where $\overline u_0 : = (\overline u_0)_+ - (\overline u_0)_-$. We notice that $\overline u_0$ is a signed Radon measure, supported in $\spt(m_0)$.

Having in hand this last inequality, we readily check that the assumptions of Lemma \ref{lem:app_u0m0_mean} are fulfilled with the choice of $\b_0=\overline u_0$ and $\b_T$ as before. This means in particular that $u$ satisfies
$$
\left\{
\begin{array}{l}
-\partial_t u - v\cdot D_x u +H(x,v,D_v u ) \leq \beta, \, \text{in} \; \sD'((0,T) \times \M \times \R^d) \\  
u_0 \geq \overline u_0 \; \text{in} \; \sD'(\{ m_0 > 0 \}); \ \ u_T \leq \b_T \; \text{in} \; \sD'(\M\times\R^d),
\end{array}
\right.
$$
where when writing the traces $u_0$ and $u_T$, we are referring to the right continuous version of $u$ in time. Since by construction, $\langle \overline u_0,m_0\rangle=\int_{\M\times\R^d} (\overline u_0)_+m_0\dd x\dd v - \langle (\overline u_0)_-,m_0\rangle$ is finite, we have that $\langle u_0,m_0\rangle$ is meaningful and finite, with 
$$-\langle u_0,m_0\rangle \le - \langle \overline u_0,m_0\rangle .$$

\medskip

{\bf Lower semicontinuity of the term involving $-\int_{\M\times\R^d} m_0 u_0(\dd x\dd v)$.}

\medskip

\noindent {\it Claim.} $\int_{\{m_0>0\}}  m_0 (u_0)_-(\dd x\dd v)\le \int_{\M \times \R^d} m_0 (\overline u_0)_- (\dd x \dd v) \leq \liminf_{n \to \infty} \int_{\M \times \R^d} (u_{0,n})_- m_0 \dd x \dd v.$

\medskip

{\it Proof of Claim.} First, notice that since $u_0 - \overline u_0$ is a positive distribution, it can be represented by a Radon measure. We may therefore write, for some $\nu_0\in \sM_+(\M \times \R^d)$, such that $\spt(\nu_0)\subseteq\spt(m_0)$ and 
\be
u_0 = \overline u_0 + \nu_0 = [(\overline u_0)_+ + \nu_0] - (\overline u_0)_-.
\ee
It follows that the Hahn-Jordan decomposition of $u_0$ satisfies
\be
(u_0)_+ \leq (\overline u_0)_+ + \nu_0, \; (u_0)_- \leq (\overline u_0)_-.
\ee

Now consider any compactly supported function $\zeta \in C_c(\{ m_0 > 0 \})$, such that $0 \leq \zeta \leq 1$. Then
\be
\int_{\M \times \R^d} \zeta m_0 (\overline u_0)_-(\dd x \dd v) = \lim_{n \to \infty} \int_{\M \times \R^d} \zeta (u_{0,n})_- m_0 \dd x \dd v \leq \liminf_{n \to \infty} \int_{\M \times \R^d} (u_{0,n})_- m_0 \dd x \dd v .
\ee
Since $(u_0)_- \leq (\overline u_0)_-$ as measures,
\be
\int_{\M \times \R^d} \zeta m_0 (u_0)_-(\dd x \dd v) \leq \liminf_{n \to \infty} \int_{\M \times \R^d} (u_{0,n})_- m_0 \dd x \dd v .
\ee
Then take a non-decreasing sequence of functions $\zeta_k$ such that $\zeta_k$ converges pointwise to the indicator function of the set $\{ m_0 > 0\}$ as $k$ tends to infinity: consider for example functions such that
\be
\zeta_k(x,v) = \begin{cases}
1 & {\rm{if}} \; m_0(x,v) > 2^{-k}\\
0 & {\rm{if}} \; m_0(x,v) \leq 2^{-(k+1)} .
\end{cases}
\ee 
This is always possible since $m_0$ is continuous. Then, by monotone convergence, we indeed have
\be
\int_{\M \times \R^d} m_0 (u_0)_-(\dd x \dd v) = \lim_{k \to +\infty} \int_{\M \times \R^d} \zeta_k m_0 (u_0)_-(\dd x \dd v) \leq \liminf_{n \to \infty} \int_{\M \times \R^d} (u_{0,n})_- m_0 \dd x \dd v,
\ee
as desired and the claim follows.

\medskip

By the weak star convergence of $(u_{0,n})_+$ to $(\overline u_0)_+$ in $(L^\infty + L^{q'})(\M \times \R^d)$, we also have that, for $\overline u_0$, the positive part $(u_{0,n})_+ m_0$ converges to $(\overline u_0)_+ m_0$ strongly in $L^1(\M \times \R^d)$. Since $- u_0 \leq - \overline u_0$ as signed measures, we deduce that

\be \label{tracezeroliminf}
- \int_{\M \times \R^d} m_0 u_0(\dd x \dd v) \leq - \int_{\M \times \R^d} m_0 \overline u_0( \dd x \dd v) \leq \liminf_n - \int_{\M \times \R^d} u_{0,n} m_0 \dd x \dd v \,, 
\ee
as required.

\medskip

{\bf The term involving $\cG^\ast$.}

\medskip

For the term involving $\cG^\ast$, we notice that by convexity
\be \label{Gliminf}
 \int_{\M\times\R^d}\cG^*(x,v,\b_{T} ) \dd x \dd v \leq \liminf_{n \to + \infty}  \int_{\M\times\R^d}\cG^*(x,v,\b_{T,n} ) \dd x \dd v.
\ee
Indeed, by classical results (cf. \cite[Proposition I.2.3, Corollary I.2.2]{Ekeland-Temam}), this is a consequence of the convexity of the integrand in the last variable and Fatou's lemma that yields the lower semi-continuity with respect to the strong topology on $L^1_{\rm{loc}}(\M\times\R^d)$.

\medskip

{\bf The term involving $\cF^\ast$. }

\medskip

First note that, for functions $\beta$ such that $\beta_+ \in L^{q'}([0,T] \times \M \times \R^d)$ and $\beta_- \in L^1_\loc([0,T]\times\M\times\R^d)$, by \eqref{eq:Fconj-lb} the following inequality holds:
\be
\left | \cF^\ast(x,v,\beta) \right | \leq c^{-1} |\beta_+|^{q'} + C_F(x,v) \in L^1([0,T] \times \M \times \R^d) .
\ee

Thus
\be \label{F-remove-zero}
 \int_0^T \int_{\M\times\R^d}\cF^*(x,v,\beta ) \dd x \dd v \dd t =  \int_{\cU_{m_0}}\cF^*(x,v,\beta ) \dd x \dd v \dd t.
\ee
Indeed, since $\cU_{m_0}=\{0\}\times\{m_0>0\}\cup (0,T)\times\M\times\R^d$, for all $\d>0$ we have
\be
\left| \int_0^T \int_{\M\times\R^d}\cF^*(x,v,\beta ) \dd x \dd v \dd t -  \int_{\cU_{m_0}}\cF^*(x,v,\beta ) \dd x \dd v \dd t \right|\le \int_0^\d\int_{\{m_0>0\}}|\cF^*(x,v,\beta )| \dd x \dd v \dd t
\ee
The integrand is bounded by the $L^1$ function $|\cF^*(x,v,\beta )|$ and converges to zero almost everywhere as $\delta$ tends to zero. Thus, taking $\delta \to 0$ we obtain \eqref{F-remove-zero}. A similar equality holds for all $\beta_n$.

Therefore, by the convexity of $\cF^*$ (and by arguments similar to the one for $\cG^*$), we conclude that
\begin{align} \label{Fliminf}
\int_0^T \int_{\M\times\R^d}\cF^*(x,v,\beta ) \dd x \dd v \dd t &= \int_{\cU_{m_0}}\cF^*(x,v,\beta ) \dd x \dd v \dd t  \leq \liminf_{n \to + \infty}\int_{\cU_{m_0}}\cF^*(x,v,\beta_n ) \dd x \dd v \dd t\\
&= \liminf_{n \to + \infty}\int_0^T \int_{\M\times\R^d}\cF^*(x,v,\beta_n ) \dd x \dd v \dd t.
\end{align}

Thus, collecting all the previous arguments, one deduces that
\be
\widetilde{\cA}(u, \beta, \beta_T) \leq \liminf_{n \to + \infty} \widetilde{\cA}(u_n, \beta_n, w_n).
\ee
The thesis of the theorem follows.
\end{proof}

\begin{corollary}
In the setting and notation of the previous theorem, in fact $u_0=\overline u_0$ on $\{m_0>0\}$.
\end{corollary}
\begin{proof}
Since $(u, \beta,\beta_T)$ is a minimizer,
\begin{align}
\widetilde{\cA}(u, \beta,\beta_T) &\leq - \int_{\M\times\R^d}m_0 \overline u_{0}(\dd x\dd v) + \liminf_{n\to+\infty} \int_0^T\int_{\M\times\R^d}\cF^*(x,v, \beta_n)\dd x\dd v\dd t\\
& + \liminf_{n\to \infty} \int_{\M\times\R^d}\cG^*(x,v,\beta_{T,n}) \dd x \dd v . \\
&\leq \lim_{n \to +\infty} \left ( - \int_{\M\times\R^d}m_0 u_{0,n}(\dd x\dd v) + \int_0^T\int_{\M\times\R^d}\cF^*(x,v, \beta_n)\dd x\dd v\dd t  + \int_{\M\times\R^d}\cG^*(\beta_{T,n}) \dd x \dd v \right ) \\
& = \widetilde{\cA}(u, \beta,\beta_T),
\end{align}
where in the last equality we have used that $(u_n,\b_n,\b_{T,n})$ is a minimizing sequence.

All the above inequalities are therefore equalities. From the inequalities \eqref{tracezeroliminf}, \eqref{Gliminf} and \eqref{Fliminf} for each of the terms, we deduce that
\be
- \int_{\M\times\R^d}m_0 u_{0}(\dd x\dd v) = - \int_{\M\times\R^d}m_0 \overline u_{0}(\dd x\dd v) .
\ee
It follows that $u_0 = \overline u_{0}$ as signed measures on $\{ m_0 > 0 \}$. 
Indeed, first note that $u_0 \geq \overline u_0$ as signed measures, or in other words $u_0 - \overline u_0$ is a nonnegative measure. For any non-negative test function $\phi_0 \in C_c(\{m_0 > 0\})$ we have $m_0 \geq \e > 0$ on the support of $\phi_0$, for some $\e > 0$. Thus there exists a constant $C$ such that $\phi_0 \leq C m_0$. Thus
\be
0 \leq \int_{\M \times \R^d} \phi_0 (u_0 - \overline u_0)(\dd x \dd v) \leq C \int_{\M \times \R^d} m_0 (u_0 - \overline u_0)(\dd x \dd v) = 0 .
\ee
Thus $u_0 = \overline u_0$ as signed measures on $\{ m_0 > 0 \}$.
\end{proof}

\section{Existence and uniqueness of a solution to the MFG system}\label{sec:6}

In this section we prove Theorem~\ref{thm:existence_MFG}. First, we show that the minimizers of Problems~\ref{prob:density} and \ref{prob:relaxed} that we have obtained in the previous sections provide weak solutions $(u,m)$ of the MFG.

\begin{theorem}\label{thm:exist_sol}
Let $(u,\b,\b_T)$ be a minimizer of $\widetilde\cA$ over $\cK_\cA$ and let $(m,w)$ be a minimizer of $\cB$ over $\cK_\cB$. Then

(i) $\b(t,x,v,)=f(x,v,m(t,x,v))$ for a.e. $(t,x,v)\in(0,T)\times\M\times\R^d$, $\b_T(x,v)=g(x,v,m_T(x,v))$ for a.e. $(x,v)\in\M\times\R^d$; 

(ii) $w(t,x,v)=-m(t,x,v)D_{p_v}H(x,v,D_v u(t,x,v))$ for a.e. $(t,x,v)\in(0,T)\times\M\times\R^d$. 

As a consequence, $(u,m)$ is a weak solution to \eqref{eq:main} in the sense of Definition \ref{def:notion_solution}.
\end{theorem}

\begin{proof}
By Theorem~\ref{thm:duality},
\be
\widetilde \cA(u,\beta,\beta_T) + \cB(m,w) = 0.
\ee
Substituting the definitions of the functionals, we obtain
\begin{multline}
\int_0^T\int_{\M\times\R^d}\cF(x,v,  m)+ \cF^*(x,v, \beta)\dd x\dd v\dd t- \int_{\M\times\R^d}m_0u_0(\dd x\dd v)
\\+ \int_{\M\times\R^d} \cG(x,v,m_T) + \cG^*(x,v, \beta_T) \dd x \dd v + \int_0^T\int_{\M\times\R^d}L \left (x,v,- \frac{w}{m} \right)m\dd x\dd v\dd t = 0.
\end{multline}
Fenchel's inequality then implies that
\be
\int_0^T\int_{\M\times\R^d} \beta m \dd x\dd v\dd t + \int_{\M\times\R^d} \beta_T m_T - m_0 u_0 \dd x\dd v
+ \int_0^T\int_{\M\times\R^d}L \left (x,v,- \frac{w}{m} \right)m\dd x\dd v\dd t \leq  0.
\ee
By Corollary~\ref{cor:energy_ineq}, the left hand side is non-negative, and therefore equality holds:
\be \label{energy-equality-beta}
\int_0^T\int_{\M\times\R^d} \beta m \dd x\dd v\dd t + \int_{\M\times\R^d} \beta_T m_T - m_0 u_0 \dd x\dd v
+ \int_0^T\int_{\M\times\R^d}L \left (x,v,- \frac{w}{m} \right)m\dd x\dd v\dd t =  0.
\ee
Moreover, equality also holds almost everywhere in the applications of Fenchel's inequality. Thus the following hold almost everywhere in $[0,T] \times\M\times\R^d$:
\be \label{energy-equality-fg}
\beta = f(x,v,m(t,x,v)), \qquad \beta_T = g(x,v, m(T,x,v)) .
\ee
By \eqref{energy-equality-beta} and Corollary~\ref{cor:energy_bound},
\be
\int_0^T \int_{\M \times \R^d} m H(x,v, D_v u) + m L \left (x,v, - \frac{w}{m} \right ) + D_v u \cdot w \, \dd x \dd v \dd t \leq 0 .
\ee
By Fenchel's inequality, the integrand on the right hand side is non-negative; we deduce that equality holds in the above estimate and thus the integrand is equal almost everywhere to zero. It follows that
\be
\frac{w}{m} = - D_{p_v}H(x,v,D_v u)
\ee
almost everywhere on the support of $m$. Moreover,
\be \label{energy-equality-Lagrangian}
 m L \left (x,v, - \frac{w}{m} \right ) =  m \left ( D_v u \cdot D_{p_v}H(x,v,D_v u)  -  H(x,v, D_v u) \right ) .
 \ee
The energy equality then follows from substituting \eqref{energy-equality-fg} and \eqref{energy-equality-Lagrangian} into \eqref{energy-equality-beta}.

\end{proof}

We show now, conversely, that weak solutions to the MFG system are in fact minimizers in the corresponding variational problems. The proof of this result follows similar ideas as the corresponding ones from \cite{Cardaliaguet-Graber, CarGraPorTon}. 
\begin{theorem} \label{thm:consistency}
Let $(u,m)$ be a weak solution to \eqref{eq:main} in the sense of Definition~\ref{def:notion_solution}. Then by setting $\b:=f(\cdot,\cdot,m)$, $\b_T:=g(\cdot,\cdot,m_T)$ and $w:=-mD_{p_v}H(\cdot,\cdot,D_v u)$, we find that $(m,w)$ is a solution of Problem \ref{prob:density}, while $(u,\b,\b_T)$ is a solution of Problem \ref{prob:relaxed}.
\end{theorem}

\begin{proof}
First let us notice that by Fenchel's equality one has 
$$\cF^\ast(\cdot,\cdot,f(\cdot,\cdot, m)) = \cF(\cdot,\cdot,m) - mf(\cdot,\cdot,m).$$
We define the Borel set $B:=\{(t,x,v)\in[0,T]\times\M\times\R^d: f(x,v,m(t,x,v))\ge 0\}.$ Restricted to this set, we find
$$- C_F\leq\cF^\ast(\cdot,\cdot,f(\cdot,\cdot, m)) \le \cF(\cdot,\cdot,m),\ \ {\rm{a.e.\ in\ }}B,$$
where in the first inequality we have used our assumptions \eqref{eq:Fconj-lb}. Since, $C_F$, $\cF^*(\cdot,\cdot,0)$ and $\cF(\cdot,\cdot,m)$ are summable, this implies in particular that $\cF^\ast(\cdot,\cdot,f(\cdot,\cdot, m)_+)\in L^1([0,T]\times\M\times\R^d).$ Using the growth condition on $\cF^*$ we find furthermore that  $f(\cdot,\cdot, m)_+=\b_+\in L^{q'}([0,T]\times\M\times\R^d).$

Now, on $B^c$, i.e. when $f(\cdot,\cdot,m)\le 0$, we find 
\begin{align*}
0\le mf(\cdot,\cdot,m)_-=-mf(\cdot,\cdot,m)=\cF^\ast(\cdot,\cdot,f(\cdot,\cdot, m)) - \cF(\cdot,\cdot,m)\le \sup_{\beta <0} \cF^\ast(\cdot,\cdot,\beta)-\cF(\cdot,\cdot,m).
\end{align*}
Again, the summability of the right hand side, we find that $m f(\cdot,\cdot,m)_- \in L^1[0,T]\times\M\times\R^d.$ Using the exact same arguments for $\cG^*$, we find similarly that $(\b_T)_+\in L^{s'}(\M\times\R^d)$ and $m_T(\b_T)_-\in L^1(\M\times\R^d)$.

Moreover, we have that $D_v u\in L^r_{\rm{loc}}(\cU_{m_0}),$ $m\in L^1([0,T]\times\M\times\R^d)$ and $w\in L^1([0,T]\times\M\times\R^d;\R^d)$, so $(m,w)$ and $(u,\b,\b_T)$ are admissible competitors for the two optimisation problems.

Now, take $(\ov{u},\ov{\b},\ov{\b}_T)$ as an admissible competitor for the problem involving the functional $\widetilde\cA$. By the convexity and differentiability of $\cF^*$ and $\cG^*$ in their last variable we have
\begin{align*}
\widetilde\cA(\ov{u},\ov{\b},\ov{\b}_T)&=\int_0^T\int_{\M\times\R^d}\cF^*(x,v, \ov\beta)\dd x\dd v\dd t - \int_{\M\times\R^d}m_0\ov u_0(\dd x\dd v)  + \int_{\M\times\R^d}\cG^*(\ov\beta_T) \dd x \dd v\\
&\ge \int_0^T\int_{\M\times\R^d}\cF^*(x,v, \beta)\dd x\dd v\dd t+\int_0^T\int_{\M\times\R^d}\partial_\b\cF^*(x,v, \beta)(\ov \b-\b)\dd x\dd v\dd t\\ 
&- \int_{\M\times\R^d}m_0 u_0(\dd x\dd v)+\int_{\M\times\R^d}m_0(u_0-\ov u_0)(\dd x\dd v)\\
&+ \int_{\M\times\R^d}\cG^*(\beta_T) \dd x \dd v+\int_{\M\times\R^d}\partial_{\b_T}\cG^*(\beta_T)(\ov\b_T-\b_T) \dd x \dd v\\
&=\widetilde\cA(u,\b,\b_T)+\int_0^T\int_{\M\times\R^d}m(\ov \b-f(\cdot,\cdot,m))\dd x\dd v\dd t+\int_{\M\times\R^d}m_0(u_0-\ov u_0)(\dd x\dd v)\\
&+\int_{\M\times\R^d}m_T(\ov\b_T-g(\cdot,\cdot,m_T)) \dd x \dd v
\end{align*}
where we have used the fact that $mf(\cdot,\cdot,m)\in L^1([0,T]\times\M\times\R^d)$ and $m_Tg(\cdot,\cdot,m_T)\in L^1(\M\times\R^d)$ (by the arguments at the beginning of this proof). Moreover, $m\ov\b\in L^1([0,T]\times\M\times\R^d)$ and $m_T\ov\b_T\in L^1(\M\times\R^d)$ (cf. Corollary \ref{cor:energy_ineq}) and
$$\partial_\b\cF^*(x,v, \beta)=\partial_\b\cF^*(x,v, f(x,v,m))=\partial_\b\cF^*(x,v, \partial_m\cF(x,v,m))=m,$$ 
$$\partial_{\b_T}\cG^*(x,v, \beta_T)=\partial_\b\cG^*(x,v, g(x,v,m_T))=\partial_{\b_T}\cG^*(x,v, \partial_{m_T}\cG(x,v,m_T))=m_T.$$
Now, using \eqref{eq:energy_equality}, one obtains
\begin{align*}
\widetilde\cA(\ov{u},\ov{\b},\ov{\b}_T)&\ge\widetilde\cA(u,\b,\b_T)+\int_0^T\int_{\M\times\R^d}m\ov \b\dd x\dd v\dd t +\int_{\M\times\R^d}m_T\ov\b_T \dd x \dd v
-\int_{\M\times\R^d}m_0\ov u_0(\dd x\dd v)\\
&+\int_0^T\int_{\M\times\R^d}L(\cdot,\cdot,-w/m)m\dd x\dd v\dd t,
\end{align*}
where in the last line we have used $$D_{p_v}H(\cdot,\cdot,D_v u)\cdot D_v u-H(\cdot,\cdot,D_vu)=L(\cdot,\cdot,D_{p_v}H(\cdot,\cdot,D_v u)).$$
By Corollary \ref{cor:energy_ineq} we conclude that $\widetilde\cA(\ov{u},\ov{\b},\ov{\b}_T)\ge \widetilde\cA(u,\b,\b_T)$, as desired.

\medskip

Using the very same ideas and the convexity of $\cF$ and $\cG$, we can conclude similarly that $(m,w)$ must be a minimizer in Problem \ref{prob:density}.

\end{proof}

Finally, we show that solutions in the sense of Definition~\ref{def:notion_solution} are unique, again following similar ideas as the corresponding ones from \cite{CarGraPorTon}. One major difference, however, is that we develop a suitable comparison principle for the distributional solutions to the corresponding Hamilton-Jacobi inequalities.
This completes the proof of Theorem~\ref{thm:existence_MFG}.

\begin{proof}[Proof of Theorem~\ref{thm:existence_MFG}]
The existence of a weak solution $(u,m)$ follows from combining Theorem~\ref{thm:existence} (existence of a minimizer for $\tilde \cA$), Theorem~\ref{thm:duality} (duality, and the fact that the infimum for $\tilde \cB$ is attained) and Theorem~\ref{thm:exist_sol} (minimizers are weak solutions in the sense of Definition~\ref{def:notion_solution}).

For the uniqueness, we first apply Theorem~\ref{thm:consistency} to obtain that for $i=1,2$, $(u_i, f(\cdot, \cdot, m_i), g(\cdot, \cdot, m_i(T)))$ are minimizers of $\tilde \cA$ over $\cK_{\tilde \cA}$ and $(m_i, - m_i D_pH(\cdot, \cdot, D_v u_i))$ are minimizers of $\tilde \cB$ over $\cK_{\tilde \cB}$. Since the minimizer of $\tilde \cB$ is unique by strict convexity, $m_1 = m_2 = : m$ almost everywhere and $- m_1 D_pH(\cdot, \cdot, D_v u_1) = - m_2 D_pH(\cdot, \cdot, D_v u_2) = : w$ almost everywhere.

To show that $u_1 = u_2$ almost everywhere on the set $\{ m > 0\}$, we first define $u = \max\{ u_1, u_2\}$. By Lemma~\ref{lem:maximum}, $u$ also satisfies the Hamilton-Jacobi inequality, with $\beta = f(\cdot, \cdot, m)$ and $\beta_T = g(\cdot, \cdot, m_T)$. Since $u_i \leq u$ for $i=1,2$, we have
\be
- \int_{\M\times\R^d} m_0 u_0(\dd x \dd v)  \leq - \int_{\M\times\R^d} m_0 [u_i]_0 (\dd x \dd v),
\ee
and thus $\widetilde \cA(u, \beta, \beta_T) \leq \widetilde \cA(u_i, \beta, \beta_T)$. Since $u_i$ is a minimizer, equality holds. By duality, equality then holds in the energy inequalities of Corollary~\ref{cor:energy_ineq} for $u$ and $m$, with $\beta, \beta_T, w$ as defined previously.
Thus, for almost all $t \in [0,T]$,
\begin{align}\label{eq:energy_limit-tT}
\int_{\M\times\R^d}u_tm_t\dd x\dd v = \int_t^T\int_{\M\times\R^d}L\left(x,v,-\frac{w}{m}\right)m\dd x\dd v\dd t + \int_t^T\int_{\M\times\R^d}\b m\dd x\dd v\dd t + \int_{\M\times\R^d} \b_Tm_T \dd x\dd v .
\end{align}
The same is true replacing $u$ by $u_i$, and so 
\be
\int_{\M\times\R^d}u_tm_t\dd x\dd v = \int_{\M\times\R^d}(u_i)_tm_t\dd x\dd v , \qquad i=1,2.
\ee
Thus, since also $u_i \leq u$, we deduce that $u_i = u$ almost everywhere on the set $\{ m> 0 \}$.

\end{proof}

\section{Sobolev estimates on the solutions}\label{sec:regularity}

In this section, we obtain Sobolev estimates on the optimizers of the variational problems, and hence on weak solutions for the MFG system \eqref{eq:main}. The general idea is to ``compare'' the optimality of the optimizers in the variational problems with their carefully chosen translates. Then using strong convexity of the data one can deduce differential quotient estimates.

These results are inspired by \cite{GraMes,GraMesSilTon}. However, because of the kinetic nature of the model we need completely new ideas when we consider perturbations. So, the estimates that we obtain are on suitable kinetic differential operators applied to the solutions. Another crucial difference between our results and the ones in \cite{GraMes,GraMesSilTon} is that our Sobolev estimates in the $x$ and $v$ variables are local in time on $(0,T]$. The main reason behind this is that we have a weaker notion of trace for $u_0$, that we cannot ensure to be stable under perturbations. This imposed further technical complications that require us to work in the case of $r=2$.

We emphasize that these estimates are consequences of the {\it stronger} convexity and regularity assumptions on the data stated in Assumption~\ref{hyp:stronger}.

\medskip

\subsection{Local in time Sobolev estimates}\label{subsec:v-reg}

Let $\zeta:[0,T]\to\R$ be a smooth cut-off function such that {\Blue $\zeta(0)=0$ and $\zeta(t) > 0$ for all $t > 0$.} We define $\eta:[0,T]\to\R$ as $\eta(t):=\int_0^t\zeta(s)\dd s$.

For competitors $(m,w)$ in Problem \ref{prob:density}, without loss of generality one might assume the representation $w=Vm$, for a suitable vector field $V$. Let $\d\in\R^d$ with $|\d|\ll 1$ and define 
$$m^\d(t,x,v):=m(t,x+\eta(t)\d,v+\zeta(t)\d) \ \ {\rm{and}}\ \ V^\d(t,x,v):=V(t,x+\eta(t)\d,v+\zeta(t)\d)-\zeta'(t)\d.$$
We use the notation $w^\d:=V^\d m^\d$.

We notice that by construction, if $(m,w)=(m,Vm)$ is a distributional solution to \eqref{eq:continuity_main}, so is $(m^\d,w^\d)=(m^\d,V^\d m^\d)$ and $m^\d(0,\cdot,\cdot)=m_0$.

Similarly, for competitors $(u,\b,\b_T)$ in Problem \ref{prob:relaxed} we define 
\begin{align*}
&u^\d(t,x,v):=u(t,x+\eta(t)\d,v+\zeta(t)\d),\ \  \b^\d(t,x,v):=\b(t,x+\eta(t)\d,v+\zeta(t)\d),\\ 
&{\rm{and}}\  \ \b^\d_T(x,v):=\b_T(x+\eta(t)\d,v+\zeta(t)\d).
\end{align*}
Furthermore, we define 
\begin{align*}
&H^\d(x,v,\xi):=H(x+\eta(t)\d,v+\zeta(t)\d,\xi)+\zeta'(t)\d\cdot \xi,\\
&\cF^\d(x,v,\theta):=\cF(x+\eta(t)\d,v+\zeta(t)\d,\theta),\\
&\cG^\d(x,v,\theta):=\cG(x+\eta(t)\d,v+\zeta(t)\d,\theta).
\end{align*}
When computing the Legendre transforms of these functions in their last variables we obtain
\begin{align*}
&(H^\d)^*(x,v,\xi):=H^*(x+\eta(t)\d,v+\zeta(t)\d,\xi-\zeta'(t)\d),\\
&(\cF^\d)^*(x,v,\theta):=\cF^*(x+\eta(t)\d,v+\zeta(t)\d,\theta),\\
&(\cG^\d)^*(x,v,\theta):=\cG^*(x+\eta(t)\d,v+\zeta(t)\d,\theta).
\end{align*}
Let us notice that $H^\d$ satisfies in particular the hypotheses imposed in Assumptions \ref{hyp:Hamiltonian}. 
Correspondingly, we define the functionals $\widetilde\cA^\d$ and $\cB^\d$ and the constraint sets $\cK_\cA^\d$ and $\cK_\cB^\d$, using the shifted versions of the data functions. By construction, as a consequence of a change of variable formula, the proof of the following lemma is immediate.
\begin{lemma}
$(m,w)$ is an optimizer of $\cB$ over $\cK_\cB$ if and only if $(m^\d,w^\d)$ is an optimizer of $\cB^\d$ over $\cK_\cB^\d$. Similarly, $(u,\b,\b_T)$ is an optimizer of $\widetilde \cA$ over $\cK_\cA$ if and only if $(u^\d,\b^\d,\b_T^\d)$ is an optimizer of $\widetilde\cA^\d$ over $\cK_\cA^\d$.
\end{lemma}

\begin{proof}
We provide the proof of one of the statements only, the other ones follow similar steps. Suppose that $(m^\d,w^\d)$ is an optimizer of $\cB^\d$ over $\cK_\cB^\d$. This means in particular the minimality of the quantity
\begin{align*}
&\int_0^T\int_{\M\times\R^d}\cF^\d(x,v,  m^\d)\dd x\dd v\dd t + \int_0^T\int_{\M\times\R^d}(H^\d)^* \left (x,v,- \frac{w^\d}{m^\d} \right)m\dd x\dd v\dd t\\
& +\int_{\M\times\R^d}\cG^\d(x,v,m_T^\d(x,v))\dd x\dd v\\
&=\int_0^T\int_{\M\times\R^d}\cF(x+\eta(t)\d,v+\zeta(t)\d, m(t,x+\eta(t)\d,v+\zeta(t)\d))\dd x\dd v\dd t\\
& + \int_0^T\int_{\M\times\R^d}H^* \left (x+\eta(t)\d,v+\zeta(t)\d,- \frac{w(t,x+\eta(t)\d,v+\zeta(t)\d)}{m(t,x+\eta(t)\d,v+\zeta(t)\d)} \right)m\dd x\dd v\dd t\\
& +\int_{\M\times\R^d}\cG(x+\eta(t)\d,v+\zeta(t)\d,m_T(x+\eta(t)\d,v+\zeta(t)\d))\dd x\dd v\\
&=\int_0^T\int_{\M\times\R^d}\cF(x,v,  m)\dd x\dd v\dd t + \int_0^T\int_{\M\times\R^d}H^* \left (x,v,- \frac{w}{m} \right)m\dd x\dd v\dd t\\
& +\int_{\M\times\R^d}\cG(x,v,m_T(x,v))\dd x\dd v, 
\end{align*}
where in the last equality we have used the change of variables $(x,v)\mapsto(x-\eta(t)\d,v-\zeta(t)\d)$. So, this means that the minimality of $(m^\d,w^\d)$, after a change of variables, yields the minimality of $(m,w)$.
\end{proof}

Now we are ready to state the main result of this subsection.

\begin{theorem}\label{thm:reg_v}
Suppose that $(u,m)$ is a weak solution to \eqref{eq:main} in the sense of Definition \ref{def:notion_solution} and that \eqref{hyp:coupling_strong}, \eqref{hyp:H_strong_coercive}, \eqref{hyp:D_xv^2L} hold.

Then, there exists $\ov C>0$ such that
$$\|m^{\frac{q}{2} - 1}(\eta D_x+ \zeta D_v)m\|_{L^2((0,T]\times\M\times\R^d)} \leq \ov C,\ \ \ \  \|m^{1/2}(\eta D_x+ \zeta D_v)D_v u\|_{L^2((0,T]\times\M\times\R^d)} \leq \ov C$$
and 
$$
\|m_T^{\frac{s}{2} - 1}(\eta(T)D_x+ \zeta(T)D_v)m_T\|_{L^2(\M\times\R^d)} \leq \ov C.
$$
\end{theorem} 

\begin{remark}
As for Theorem~\ref{thm:reg-main}, this is an informal statement: the result we obtain is on suitable difference quotients as in estimate \eqref{ineq:fundamental} below.
\end{remark}

\begin{proof}[Proof of Theorem \ref{thm:reg_v}]

Let $(u_n,\b_n,\b_{T,n})_{n\in\N}$ be a minimizing sequence for Problem \ref{prob:relaxed} such that $u_n\in C_c^1([0,T]\times\M\times\R^d)$, 
$$
\beta_n=-\partial_t u_n - v\cdot D_x u_n + H(x,v,D_v u_n),\ \ \b_{T,n}=u(T,\cdot,\cdot).
$$
Let us recall that after passing to a subsequence, that we do not relabel, as a consequence of Proposition \ref{prop:HJ-compactness}, Lemma \ref{lem:modif} and by Claim 2 in the proof of Theorem \ref{thm:exist_sol}, we have that
\begin{itemize}
\item $(\beta_n)_+ \rightharpoonup \beta_+$, weakly in $L^{q'}([0,T]\times\M\times\R^d)$, as $n\to+\infty$.
\item $(\beta_n)_- \rightharpoonup \beta_-$, weakly in $L^1_{\rm{loc}}(\cU_{m_0})$, as $n\to+\infty$.
\item $(\beta_{T,n})_+ \rightharpoonup (\b_T)_+$, weakly in $L^{s'}(\M \times \R^d)$, as $n\to+\infty$.
\item $(\b_{T,n})_- \rightharpoonup (\b_T)_-$, weakly in $L^1_{\rm{loc}}(\M \times \R^d)$, as $n\to+\infty$.
\item $(u_{0,n})_+ \weaklys (\tilde u_0)_+$, weakly-$\ast$ in $(L^\infty + L^{q'})(\M \times \R^d)$, as $n\to+\infty$.
\item $D_v u_n \rightharpoonup D_v u$, weakly in $L^r_m([0,T]\times\M\times\R^d)$, as $n\to+\infty$.
\end{itemize}

Notice that the previous arguments imply also that the subsequence can be chosen such that for all $M<0$
\begin{align}\label{conv:beta_truncated}
\b_n\one_{\{\b_n\ge M\}}\weaklys \b\one_{\{\b\ge M\}}, \text{weakly-$*$ in }(L^\infty+L^{q'})([0,T]\times\M\times\R^d)\ \text{as } n\to+\infty
\end{align}
and
\begin{align}\label{conv:beta_T_truncated}
\b_{T,n}\one_{\{\b_{T,n}\ge M\}}\weaklys \b_T\one_{\{\b_T\ge M\}}, \text{weakly-$*$ in }(L^\infty+L^{s'})(\M\times\R^d)\ \text{as } n\to+\infty.
\end{align}

Furthermore, by Theorem \ref{thm:exist_sol}, we have that $\b=f(\cdot,\cdot,m)$ and $\b_T=g(\cdot,\cdot,m_T)$. Let {\Blue $w= - mD_{p_v}H(\cdot,\cdot,D_v u)$. }

Fix $\delta \in \R^{d}$ such that  $|\d|\ll 1$ and $\zeta:[0,T]\to\R$ as described at the beginning of this subsection.

\medskip

Now, the main idea is to use $u^\delta_n$ as a test function in the weak formulation of the equation satisfied by $(m,w)$ and and $u_n$ as test function in the weak form of the equation satisfied by $(m^\d,w^\d)$. Then we combine these inequalities with the energy equality \eqref{eq:energy_equality} written for $(m,w)$ and $(m^\d, w^\d)$, respectively, and rely on the strong convexity and regularity properties of the data to deduce a differential quotient estimate.

Following these steps, we obtain

\begin{align} \label{eq:phi_n^delta-test-m}
\int_{\M\times\R^d} [\b_{T,n}^\delta m_T &- u_{0,n}^\delta m_0]\dd x\dd v \geq \int_0^T \int_{\M\times\R^d} (H^\d(x,v,D_v u_n^\delta)-\b_n^\delta)m + D_v u_n^\delta \cdot w \dd x\dd t.
\end{align}
We combine this with the energy equality \eqref{eq:energy_equality} for $(m,w)$ to get
\begin{multline} \label{eq:space-regularity1'}
	\int_{\M\times\R^d} (\b_{T,n}^\delta - g(\cdot,\cdot,m_T)) m_T - (u^\delta_{0,n}-u_0)m_0 
	\\
	\geq
	\int_0^T \int_{\M\times\R^d} (H^\d(x,v,D_v u^\delta_n) + H^*(x,v,-w/m) + D_v u^\delta_n \cdot w/m - \b_n^\delta + f(\cdot,\cdot,m))m  \dd x\dd v\dd t.
\end{multline}
	Similarly, using $u_n$ as a test function in the weak form of the equation for $(m^\delta,w^\delta)$ and combining with \eqref{eq:energy_equality} for $(m^\d,w^\d)$,
	\begin{multline} \label{eq:space-regularity2'}
	\int_{\M\times\R^d}\left[ (\b_{T,n} - g^\d(x,v,m_T^\d)) m_T^\delta - (u_{0,n}-u_0^\delta)m^\delta_0 \right]\dd x\dd v
	\\
	\geq
	\int_0^T \int_{\M\times\R^d} (H(x,v,D_v u_n) + (H^\d)^*(x,v,-w^\delta/m^\delta) + D_v u_n \cdot w^\delta/m^\delta - \b_n + f^\delta(m^\delta))m^\delta  \dd x\dd t
	\end{multline}
Adding \eqref{eq:space-regularity1'} and \eqref{eq:space-regularity2'}, after some changes of variables (translations) and a Taylor expansion of $L$, we deduce
\begin{align} \label{eq:space-regularity1}
&\int_0^T \int_{\M\times\R^d} (H(x+\eta(t)\d,v+\zeta(t)\d,D_v u^\delta_n) + H^*(x+\eta(t)\d,v+\zeta(t)\d,-w/m) + D_v u^\delta_n \cdot w/m)m  \dd x\dd v\dd t\\
&+ \int_0^T \int_{\M\times\R^d} (H(x-\eta(t)\d,v-\zeta(t)\d,D_v u^{-\delta}_n) + H^*(x-\eta(t)\d,v-\zeta(t)\d,-w/m) + D_v u_n^{-\delta} \cdot w/m)m  \dd x\dd v\dd t\\
&\leq \int_{\M\times\R^d} \left[\b_{T,n}(m_T^\delta + m_T^{-\delta}) - 2g(x,v,m_T)m_T\right]\dd x\dd v-  \int_{\M\times\R^d} 2(u_{0,n} - u_0)m_0\dd x\dd v\\
&+ \int_0^T \int_{\M\times\R^d} \left[\b_n^\delta + \b_n^{-\delta} - 2f(m)\right]m \dd x\dd v \dd t,\\
&+ \int_0^T \int_{\M\times\R^d}\left[-H^*(x,v,-w/m)+H^*(x+\eta(t)\d,v+\zeta(t)\d,-w/m)-\zeta'(t)\d\cdot D_v u^\delta_n\right]m\dd x\dd v\dd t\\
&+\int_0^T \int_{\M\times\R^d}\left[-H^*(x,v,-w/m)+H^*(x-\eta(t)\d,v-\zeta(t)\d,-w/m)+\zeta'(t)\d\cdot D_vu_n^{-\d}\right]m \dd x\dd v\dd t\\
&=\int_{\M\times\R^d} \left[\b_{T,n}(m_T^\delta + m_T^{-\delta}) - 2g(x,v,m_T)m_T\right]\dd x\dd v-  \int_{\M\times\R^d} 2(u_{0,n} - u_0)m_0\dd x\dd v\\
&+ \int_0^T \int_{\M\times\R^d} \left[\b_n^\delta + \b_n^{-\delta} - 2f(m)\right]m \dd x\dd v \dd t,\\
&+\int_0^T \int_{\M\times\R^d} \left[\zeta'(t)\d\cdot D_vu_n^{-\d}-\zeta'(t)\d\cdot D_v u^\delta_n \right]m \dd x\dd v \dd t\\
&+ \int_0^T \int_{\M\times\R^d} \int_0^1\int_s^{-s}\langle \left[\eta^2(t)D^2_{xx}H^*+\zeta^2(t)D^2_{vv}H^*\right](x+\tau\eta(t)\d,v+\t\zeta(t)\d,-w/m)\d,\d\rangle m\dd\tau\dd s\dd x\dd v\dd t\\
&+ \int_0^T \int_{\M\times\R^d} \int_0^1\int_s^{-s}\langle2\eta(t)\zeta(t)D^2_{xv}H^*(x+\tau\eta(t)\d,v+\t\zeta(t)\d,-w/m)\d,\d\rangle m\dd\tau\dd s\dd x\dd v\dd t.
\end{align} 
where we have also used the facts that by the choice of $\eta$ and $\zeta$, we have $u_{0,n}^\d=u_{0,n}$, $u_0^\d=u_0$ and $m_0^\d=m_0$.

Our aim now is to pass to the limit $n\to+\infty$ in \eqref{eq:space-regularity1} and derive a differential quotient estimate. For this, we consider each of the terms separately.

\medskip

{\bf Step 1.} First, we notice that by \eqref{hyp:D_xv^2L} and by the fact that $\frac{|w|^{r'}}{m^{r'-1}}\in L^1([0,T]\times\M\times\R^d)$, there exists $C>0$ such that
\begin{align*}
& \int_0^T \int_{\M\times\R^d} \int_0^1\int_s^{-s}\langle \left[\eta^2(t)D^2_{xx}H^*+\zeta^2(t)D^2_{vv}H^*\right](x+\tau\eta(t)\d,v+\t\zeta(t)\d,-w/m)\d,\d\rangle m\dd\tau\dd s\dd x\dd v\dd t\\
&+ \int_0^T \int_{\M\times\R^d} \int_0^1\int_s^{-s}\langle2\eta(t)\zeta(t)D^2_{xv}H^*(x+\tau\eta(t)\d,v+\t\zeta(t)\d,-w/m)\d,\d\rangle m\dd\tau\dd s\dd x\dd v\dd t\\
&\le C_0|\d|^2\int_0^T \int_{\M\times\R^d}\left(\frac{|w|^{r'}}{m^{r'-1}}+m\right)\dd x\dd v\dd t\le C|\d|^2.
\end{align*}

\medskip

{\bf Step 2.} Second, let us notice that by the fact that $m\in (L^1\cap L^q)([0,T]\times\M\times\R^d)$ and by \eqref{conv:beta_truncated}, for any $M<0$ we have
\begin{align*}
\lim_{n\to+\infty}\int_0^T \int_{\M\times\R^d} \b_n^{\pm\delta}\one_{\{\b_n^{\pm\delta}\ge M\}}m \dd x\dd v \dd t=\int_0^T \int_{\M\times\R^d} f^{\pm\delta}(m^{\pm\d})\one_{\{f^{\pm\delta}(m^{\pm\d})\ge M\}}m \dd x\dd v \dd t.
\end{align*}
Therefore,
\begin{align*}
\limsup_{n\to+\infty} \int_0^T \int_{\M\times\R^d}& \left[\b_n^\delta + \b_n^{-\delta} - 2f(m)\right]m \dd x\dd v \dd t\\
&\le\lim_{n\to+\infty}\int_0^T \int_{\M\times\R^d}\left[ \b_n^\delta\one_{\{\b_n^\delta\ge M\}}+\b_n^{-\delta}\one_{\{\b_n^{-\delta}\ge M\}} - 2f(m)\right] m\dd x\dd v \dd t\\
&=\int_0^T \int_{\M\times\R^d}\left[ f^\delta(m^\d)\one_{\{f^\delta(m^\d)\ge M\}}+f^{-\delta}(m^{-\d})\one_{\{f^{-\delta}(m^{-\d})\ge M\}}-2f(m)\right]m \dd x\dd v \dd t.
\end{align*}
Now, sending $M\to-\infty$, we conclude that
\begin{align}\label{ineq:f_shifted}
\limsup_{n\to+\infty} \int_0^T \int_{\M\times\R^d}& \left[\b_n^\delta + \b_n^{-\delta} - 2f(m)\right]m \dd x\dd v \dd t\\
&\le\int_0^T \int_{\M\times\R^d}\left[ f^\delta(m^\d)+f^{-\delta}(m^{-\d})-2f(m)\right]m \dd x\dd v \dd t,
\end{align}
where we have used the fact that $f(m)m, (f^\delta(m^\delta))_+, (f^{-\delta}(m^{-\delta}))_+ \in L^1$ so that the integrand is upper bounded by an $L^1$ function to allow us to apply the monotone convergence theorem. Since the left hand side of inequality \eqref{eq:space-regularity1} is bounded from below by zero, it follows the right hand side of \eqref{ineq:f_shifted} is not negative infinity.

By the very same arguments
one can conclude that

\begin{align}\label{ineq:g_shifted}
\limsup_{n\to+\infty}\int_{\M\times\R^d} &\left[\b_{T,n}(m_T^\delta + m_T^{-\delta}) - 2g(x,v,m_T)m_T\right]\dd x\dd v\\
&\le\int_{\M\times\R^d}\left[ g^\delta(m_T^\d)+g^{-\delta}(m_T^{-\d})-2g(m_T)\right]m_T \dd x\dd v.
\end{align}

\medskip

{\bf Step 3.} By Young's inequality, we have 
\begin{align*}
\int_0^T \int_{\M\times\R^d} \left[\zeta'(t)\d\cdot D_vu_n^{-\d}-\zeta'(t)\d\cdot D_v u^\delta_n \right]m \dd x\dd v \dd t &\le C|\d|^2\\
&+c\int_0^T \int_{\M\times\R^d}\left|D_vu_n^{-\d}- D_v u^\delta_n \right|^2m \dd x\dd v \dd t,
\end{align*}
where $c>0$ is an arbitrary constant, and $C=C(c,T,\zeta')>0$.

\medskip

{\bf Step 4.} By the previous steps we can conclude that
\begin{align*}
&\int_0^T \int_{\M\times\R^d} (H(x\pm\eta(t)\d,v\pm\zeta(t)\d,D_v u^{\pm\delta}_n) + H^*(x\pm\eta(t)\d,v\pm\zeta(t)\d,-w/m) + D_v u^{\pm\delta}_n \cdot w/m)m  \dd x\dd v\dd t\\
&-c\int_0^T \int_{\M\times\R^d}\left|D_vu_n^{-\d}- D_v u^\delta_n \right|^2m \dd x\dd v \dd t
\end{align*}
is uniformly bounded above, independently of $n\in\N$. Let us recall that $\frac{|w|^{2}}{m}\in L^1((0,T)\times\M\times\R^d)$, and so is $H^*(x\pm\eta(t)\d,v\pm\zeta(t)\d,-w/m)m\in L^1((0,T)\times\M\times\R^d)$.
Using the growth condition on $H$, by choosing $c>0$ small enough in our application of Young's inequality we deduce that $D_v u^{\pm\delta}_n$ is uniformly bounded in $L^2_m((0,T)\times\M\times\R^d;\R^d)$. By a change of variable, one can similarly deduce that $D_v u_n$ is uniformly bounded in $L^2_{m^{\pm\d}}((0,T)\times\M\times\R^d;\R^d)$. 
\medskip

{\it Claim}. After passing to a subsequence that we do not relabel, we have $D_v u^{\pm\delta}_n\weakly D_v u^{\pm\delta}$ weakly in $L^2_m((0,T)\times\M\times\R^d;\R^d)$, as $n\to+\infty$.

{\it Proof of Claim}. By the uniform boundedness of the sequence, we know that there exists a subsequence of it (that we do not relabel) and $\xi\in L^2_m((0,T)\times\M\times\R^d;\R^d)$, as weak limit, i.e. 
$$
\int_0^T\int_{\M\times\R^d}D_v u^{\pm\delta}_n\cdot\phi m\dd x\dd v\dd t\to \int_0^T\int_{\M\times\R^d}\xi\cdot\phi m\dd x\dd v\dd t,\ \ \forall \phi\in L^{2}_m((0,T)\times\M\times\R^d;\R^d),\ \ {\rm{as}}\ n\to+\infty.
$$
Thus, we aim to show that $\xi=D_v u^{\pm\d}$. As $D_v u^{\pm\delta}_n\weakly D_v u^{\pm\delta}$, weakly in $L^2_{\rm{loc}}(\cU_{m_0})$, as $n\to+\infty$, we can argue similarly as in the proof of Claim 2, in the proof of Theorem \ref{thm:exist_sol} to deduce the claim.

\medskip

\medskip

{\bf Step 5.} By summarizing, \eqref{eq:space-regularity1} implies that
\begin{align*}  
&\int_0^T \int_{\M\times\R^d} (H(x+\eta(t)\d,v+\zeta(t)\d,D_v u^\delta_n) + H^*(x+\eta(t)\d,v+\zeta(t)\d,-w/m) + D_v u^\delta_n \cdot w/m)m  \dd x\dd v\dd t\\
&+ \int_0^T \int_{\M\times\R^d} (H(x-\eta(t)\d,v-\zeta(t)\d,D_v u^{-\delta}_n) + H^*(x-\eta(t)\d,v-\zeta(t)\d,-w/m) + D_v u_n^{-\delta} \cdot w/m)m  \dd x\dd v\dd t\\
&-c\int_0^T \int_{\M\times\R^d}\left|D_vu_n^{-\d}- D_v u^\delta_n \right|^2m \dd x\dd v \dd t\\
&\leq \int_{\M\times\R^d} \left[\b_{T,n}(m_T^\delta + m_T^{-\delta}) - 2g(x,v,m_T)m_T\right]\dd x\dd v-  \int_{\M\times\R^d} 2(u_{0,n} - u_0)m_0\dd x\dd v\\
&+ \int_0^T \int_{\M\times\R^d} \left[\b_n^\delta + \b_n^{-\delta} - 2f(m)\right]m \dd x\dd v \dd t\\
&+ C|\d|^2.
\end{align*}
Using the additional assumption \eqref{eq:Hcoercivity} and the inequality $|a+b|^2 \leq 2\left( a^2 +b^2\right)$, for $c>0$ sufficiently small, one can conclude that there exists $c_0>0$ depending only on the data, such that
\begin{align}\label{ineq:reg_v_last}
&c_0\int_0^T \int_{\M\times\R^d}\left|D_vu_n^{-\d}- D_v u^\delta_n \right|^2m \dd x\dd v \dd t\\
\nonumber&\leq \int_{\M\times\R^d} \left[\b_{T,n}(m_T^\delta + m_T^{-\delta}) - 2g(x,v,m_T)m_T\right]\dd x\dd v-  \int_{\M\times\R^d} 2(u_{0,n} - u_0)m_0\dd x\dd v\\
\nonumber&+ \int_0^T \int_{\M\times\R^d} \left[\b_n^\delta + \b_n^{-\delta} - 2f(m)\right]m \dd x\dd v \dd t\\
\nonumber&+ C|\d|^2.
\end{align}
Now, our aim is to pass to the limit with $n\to+\infty$ first in \eqref{ineq:reg_v_last}. For this we take $\liminf_{n\to+\infty}$ of the left hand side and $\limsup_{n\to+\infty}$ of the right hand side. We notice that the term $- \int_{\M\times\R^d} 2u_{0,n}m_0\dd x\dd v$ needs special attention, since we do not have upper semicontinuity of it. Because of this, we add to both sides of \eqref{ineq:reg_v_last} the quantity
$$2\int_0^T\int_{\M\times\R^d}\cF^*(x,v, \beta_n)\dd x\dd v\dd t  + 2\int_{\M\times\R^d}\cG^*(\beta_{T,n}) \dd x \dd v $$
before passing to the limit. Thus we obtain
\begin{align*}
&\liminf_{n\to+\infty}c_0\int_0^T \int_{\M\times\R^d}\left|D_vu_n^{-\d}- D_v u^\delta_n \right|^2m \dd x\dd v \dd t\\
&+\liminf_{n\to+\infty}2\int_0^T\int_{\M\times\R^d}\cF^*(x,v, \beta_n)\dd x\dd v\dd t  + \liminf_{n\to+\infty}2\int_{\M\times\R^d}\cG^*(\beta_{T,n}) \dd x \dd v\\
&\leq \limsup_{n\to+\infty}\int_{\M\times\R^d} \left[\b_{T,n}(m_T^\delta + m_T^{-\delta}) - 2g(x,v,m_T)m_T\right]\dd x\dd v+2  \int_{\M\times\R^d}u_0m_0\dd x\dd v\\
&+ \limsup_{n\to+\infty}\int_0^T \int_{\M\times\R^d} \left[\b_n^\delta + \b_n^{-\delta} - 2f(m)\right]m \dd x\dd v \dd t\\
&+\limsup_{n\to+\infty}\left(-2\int_{\M\times\R^d} u_{0,n}m_0\dd x\dd v+2\int_0^T\int_{\M\times\R^d}\cF^*(x,v, \beta_n)\dd x\dd v\dd t  + 2\int_{\M\times\R^d}\cG^*(\beta_{T,n}) \dd x \dd v\right)\\
&+ C|\d|^2.
\end{align*}

All the arguments in the previous steps allow us to pass to the limit. By the fact that $(u_n,\b_n,\b_{T,n})$ is a minimizing sequence, we get that 
\begin{align*}
&\limsup_{n\to+\infty}\left(-2\int_{\M\times\R^d} u_{0,n}m_0\dd x\dd v+2\int_0^T\int_{\M\times\R^d}\cF^*(x,v, \beta_n)\dd x\dd v\dd t  + 2\int_{\M\times\R^d}\cG^*(\beta_{T,n}) \dd x \dd v\right)\\
&\lim_{n\to+\infty}\left(-2\int_{\M\times\R^d} u_{0,n}m_0\dd x\dd v+2\int_0^T\int_{\M\times\R^d}\cF^*(x,v, \beta_n)\dd x\dd v\dd t  + 2\int_{\M\times\R^d}\cG^*(\beta_{T,n}) \dd x \dd v\right)\\
&=2\widetilde\cA(u,\b,\b_t)\\
&=2\left(-\int_{\M\times\R^d} u_{0}m_0\dd x\dd v+\int_0^T\int_{\M\times\R^d}\cF^*(x,v, \beta)\dd x\dd v\dd t  + 2\int_{\M\times\R^d}\cG^*(\beta_{T}) \dd x \dd v\right).
\end{align*}

So, after simplification, one obtains

\begin{align} \label{eq:space-regularity122}
c_0\int_0^T \int_{\M\times\R^d}&\left|D_vu^{-\d}- D_v u^\delta \right|^2m \dd x\dd v \dd t\\
&\leq \int_{\M\times\R^d} \left[g^\d(m_T^\d)+g^{-\d}(m_T^{-\d}) - 2g(x,v,m_T\right]m_T\dd x\dd v\\
&+ \int_0^T \int_{\M\times\R^d} \left[f^\d(m^\d) + f^{-\d}(m^{-\d}) - 2f(m)\right]m \dd x\dd v \dd t\\
&+ C|\d|^2.
\end{align}

\medskip

\medskip

Now, using \eqref{hyp:f_Lipschitz_x} and \eqref{hyp:f_strongly_monotone} the very same arguments as in \cite[computation (4.25)]{GraMes} yield
	\begin{multline} \label{eq:f_monotonicity_estimates}
	\int_{\M\times\R^d} \left(f^\delta(m^\delta) + f^{-\delta}(m^{-\delta}) - 2f(m)\right){m} \dd x\dd v
	\\
	\leq
	C|\delta|^2 \left(1+\int_{\M\times\R^d} \min\{ m^\delta,m\}^{q} \dd x\dd v\right)
	- \frac{c_f}{2}\int_{\M\times\R^d} \min\{ (m^\delta)^{q-2}, m^{q-2}\}| m^\delta - m|^2 \dd x\dd v.
	\end{multline}
Similarly, \eqref{hyp:g_Lipschitz_x} and \eqref{hyp:g_strongly_monotone} yield	
\begin{multline} \label{eq:g_monotonicity_estimates}
	\int_{\M\times\R^d} \left(g^\delta(m_T^\delta) + g^{-\delta}(m_T^{-\delta}) - 2g(m_T)\right){m_T} \dd x\dd v
	\\
	\leq
	C|\delta|^2 \left(1+\int_{\M\times\R^d} \min\{ m_T^\delta,m_T\}^{s} \dd x\dd v\right)
	- \frac{c_g}{2}\int_{\M\times\R^d} \min\{ (m_T^\delta)^{s-2}, m_T^{s-2}\}| m_T^\delta - m_T|^2 \dd x\dd v.
\end{multline}

	 Combining these estimates with \eqref{eq:space-regularity122}, we get
\begin{multline}\label{ineq:fundamental} 
	c_0\int_0^T \int_{\M\times\R^d}\left|D_vu^{-\d}- D_v u^\delta \right|^2m \dd x\dd v \dd t
	\\
	+ \frac{c_f}{2}\int_0^T \int_{\M\times\R^d} \min\{(m^\delta)^{q-2},m^{q-2}\}|m^\delta - m|^2 \dd x\dd v \dd t
	\\
	+ \frac{c_g}{2}\int_{\M\times\R^d} \min\{(m_T^\delta)^{s-2},m_T^{s-2}\}|m_T^\delta - m_T|^2 \dd x\dd v
	\leq    C |\d|^2.
\end{multline}	 
Dividing by $|\d|^2$ and letting $\d\to 0$, we easily obtain the result.
\end{proof}

\subsubsection{Recovering estimates on the operator $(tD_x+D_v)$ applied to solutions} By choosing a specific structure for the cut-off function $\zeta$, we can recover estimates on more particular differential operators. Suppose that $\zeta(t)=0$ for $t\in[0,t_0/2]$, and $\zeta(t)=1,$ for $ t\in(t_0,T]$ for some $t_0\in (0,T)$ (to be chosen to be arbitrary), in such a way that also $\eta(t_0) = t_0$. Then in Theorem \ref{thm:reg_v}, the operator $(\eta D_x+ \zeta D_v)$, for $t > t_0$ becomes $(tD_x+D_v)$. So, one can state the following local in time corollary.

\begin{corollary}\label{cor:cor_reg_1}
Suppose that the assumptions of Theorem \ref{thm:reg_v} take place. Then, there exists $\ov C>0$ such that
$$\|m^{\frac{q}{2} - 1}(tD_x+D_v)m\|_{L^2_{\rm{loc}}((0,T]\times\M\times\R^d)} \leq \ov C,\ \ \ \  \|m^{1/2}(tD_x+D_v)D_v u\|_{L^2_{\rm{loc}}((0,T]\times\M\times\R^d)} \leq \ov C$$
and 
$$
\|m_T^{\frac{s}{2} - 1}(TD_x+D_v)m_T\|_{L^2(\M\times\R^d)} \leq \ov C.
$$
\end{corollary}

\subsubsection{Recovering estimates on the operator $D_x$ applied to solutions}

Now suppose that $\eta(t)=0$ for $t\in[0,t_0/2]$ and $\eta(t)=1$ for $t\in(t_0,T]$ (where $t_0\in(0,T)$ can be chosen arbitrarily). We still require that $\zeta:=\eta'$. With this choice of cut-off functions $\eta,\zeta$, we can formulate the following result as a corollary of  Theorem \ref{thm:reg_v}.

\begin{corollary}\label{cor:cor_reg_2}
Suppose that the assumptions of Theorem \ref{thm:reg_v} take place. Then, there exists $\ov C>0$ such that
$$\|m^{\frac{q}{2} - 1}D_xm\|_{L^2_{\rm{loc}}((0,T]\times\M\times\R^d)} \leq \ov C,\ \ \ \  \|m^{1/2}D_xD_v u\|_{L^2_{\rm{loc}}((0,T]\times\M\times\R^d)} \leq \ov C$$
and 
$$
\|m_T^{\frac{s}{2} - 1}D_xm_T\|_{L^2(\M\times\R^d)} \leq \ov C.
$$

\end{corollary}

\subsubsection{Proof of Theorem~\ref{thm:reg-main}}

Finally, the proof of Theorem~\ref{thm:reg-main} follows from the previous two corollaries and the inequality
\be
|D_v h| \leq   |(t D_x + D_v)h| + T |D_x h|   \quad \text{for all} \; t \in [0, T],
\ee
for any $h$ Sobolev function.

\appendix

\section{Time Regularity}\label{app:time_regularity}

In this appendix, we collect some facts about the regularity with respect to time of solutions $u$ of
\be\label{eq:app_HJ}
-\partial_t u - v\cdot D_x u+H(x,v,D_v u) \leq \beta, \quad \text{in } \sD'((0,T)\times\M\times\R^d).
\ee
By this we mean that, for any non-negative test function $0 \leq \phi \in C^\infty_c \left ( (0,T) \times \M \times \R^d \right )$,
\be\label{eq:app_HJ_weak}
\int_0^T \int_{\M \times \R^d} \left (\partial_t \phi + v\cdot D_x \phi \right ) u + \phi \, H(x,v, D_v u) \dd x \dd v \dd t \leq \int_0^T \int_{\M \times \R^d} \beta \phi  \dd x \dd v \dd t .
\ee
What we discuss is close to the standard theory of distributional solutions. However, in our case technical difficulties arise since, firstly, \eqref{eq:app_HJ} is an inequality and, secondly, we wish to work on the atypical domain $\cU_{m_0}$. We therefore found it useful to clarify several points. Our main goal is to give a precise sense to the specification of boundary data for this problem at time $t=T$, and to give a meaning to $u_0$ (the `value of $u$ at time $t=0$'), which appears in the functional $\widetilde\cA$ defining the variational problem.

Throughout this appendix we impose the following summability conditions on the pair $(u, \beta) \in L^1_{\rm{loc}}(\cU_{m_0}) \times L^1_{\rm{loc}}(\cU_{m_0})$ and that $H$ satisfies \eqref{hyp:H}. 
\begin{assumption} \label{hyp:app-summ}
The pair $(u, \beta) \in L^1_{\rm{loc}}(\cU_{m_0}) \times L^1_{\rm{loc}}(\cU_{m_0})$ satisfies the following assumptions:
\begin{itemize}
\item The positive part of $\beta$ satisfies $\beta_+ \in L^{q'}([0,T] \times \M \times \R^d)$;
\item $D_v u \in L^r_{\loc}(\cU_{m_0})$;
\end{itemize}
\end{assumption}

Under Assumption~\ref{hyp:app-summ}, by a density argument the weak form \eqref{eq:app_HJ_weak} extends additionally to test functions in $C^1_c((0,T)\times\M\times\R^d)$.

\begin{lemma}\label{lem:app_trace_1}
Let $(u, \beta) \in L^1_{\rm{loc}}(\cU_{m_0}) \times L^1_{\rm{loc}}(\cU_{m_0})$ be a distributional solution to \eqref{eq:app_HJ} satisfying Assumption~\ref{hyp:app-summ}.
Then 
\begin{enumerate}
\item for any $\phi\in C^1_c((0,T)\times\M\times\R^d)$ the function
\begin{equation}\label{eq:app_function}
(0,T)\ni t \mapsto \langle \phi(t), u(t) \rangle : = \int_{\M \times \R^d} \phi(t,x,v) u (t,x,v) \dd x \dd v 
\end{equation}
is of locally bounded variation and therefore has a right continuous representative with a countable number of jump discontinuities.
\item There exists a path $(0,T) \ni t \mapsto \tilde u_t \in C^1_c(\M \times \R^d)'$ which is right continuous with respect to the weak-star topology on $C^1_c(\M \times \R^d)'$ and such that $\tilde u_t = u_t$ as elements of $(C^1_c)'$ for almost every $t \in (0,T)$.
\end{enumerate}
\end{lemma}

\begin{proof}

Since $\beta + \partial_t u + v\cdot D_x u - H(x,v,D_v u)$ is a positive distribution, it is given by some Radon measure $\nu$ on $(0,T) \times \M \times \R^d$. 
We have 
\be\label{eq:HJ-deficit-measure}
\partial_t u  = - v\cdot D_x u + H(x,v,D_v u) - \beta + \nu =: \mu,
\ee
which we will use to deduce weak time regularity for $u$.

Consider a test function $\phi \in C^1_c \left ((0,T) \times \M \times \R^d \right )$. 
The function
\begin{align} \label{app:weak-trajectory}
f_\phi : (0,T) &\rightarrow \R \\
t &\mapsto \langle \phi, u(t) \rangle : = \int_{\M \times \R^d} \phi(t,x,v) u (t,x,v) \dd x \dd v 
\end{align}
has distributional derivative
\be \label{app:weak-distributional-derivative}
f_\phi ' = \frac{\dd}{\dd t} \langle \phi, u\rangle = \langle \partial_t \phi + v \cdot D_x \phi, u\rangle + \langle \phi, H(x,v, D_vu) +  \nu - \beta \rangle .
\ee
By Assumption~\ref{hyp:app-summ}, $u$, $\beta$ and $H(x,v, D_vu)$ are all locally integrable functions on $\cU_{m_0}$ and so in particular on $(0,T) \times \M \times \R^d$. Thus the distributional derivative $f_\phi '$ defined in \eqref{app:weak-distributional-derivative} is a Radon measure on $(0,T)$, and so the path \eqref{app:weak-trajectory} is of locally bounded variation.

It follows that $f_\phi$ has a unique right-continuous version. That is, there exists a set $E_\phi \subset [0,T]$ of full measure and a right continuous function $ \tilde{f}_\phi$ such that $ \tilde{f}_\phi(t) = f_\phi(t)$ for all $t \in E_\phi$. The function $ \tilde{f}_\phi$ satisfies
\be
\tilde{f}_\phi(t) = \tilde{f}_\phi(s) + f_\phi ' ((s,t])
\ee
for all $0 < s < t < T$.

Now consider $\psi \in C^1_c(\M \times \R^d)$ (independent of time). The path $t \mapsto \langle u(t), \psi \rangle$ has time derivative
\be
 \frac{\dd}{\dd t} \langle \psi, u(t) \rangle = \langle  v \cdot \nabla_x \psi, u\rangle + \langle \psi, H(x,v, D_vu) +  \nu - \beta \rangle .
\ee

For each compact subset $K \subset \M \times \R^d$, we define the following Radon measure on $(0,T)$: for $A \subset (0,T)$ Borel,
\be\label{def:muK}
\mu_K(A) : = \sup_{v: \exists (x,v) \in K} |v| \|u\|_{L^1(A \times K)} + \| H(\cdot,\cdot,D_v u)-\beta\|_{L^1(A\times K)} + \nu(A \times K) .
\ee

For the right-continuous versions we have, for all $0<s<t<T$ and all $\psi \in C^1_c(\M\times\R^d)$ with support contained in $K$,
\be \label{est:right-continuity}
|\tilde{f}_\psi(t) - \tilde{f}_\psi(s)| \leq \| \psi \|_{C^1} \mu_K((s,t]) .
\ee

Using these estimates, it is possible to construct a right continuous version of $u$: that is, a path $t \mapsto \tilde u_t \in C^1_c(\M \times \R^d)'$ that is right continuous with respect to the weak-star topology.

Such a construction is classical, but because of the lack of a precise reference in our context, we sketch the main ideas here. Take a countable dense set $Z \subset C^1_c(\M\times\R^d)$; there is a full measure set $E \subset (0,T)$ such that $\langle \psi, u(t) \rangle = \widetilde f_\psi(t)$ for all $t \in E$ and all $\psi \in Z$, and moreover $u(t) \in L^1_\loc(\M \times \R^d)$ (the latter is true for almost all $t$ since $u$ is $L^1_\loc$). Then
\be
\langle \tilde u(t), \psi \rangle : = \tilde f_\psi(t)
\ee
defines a bounded linear functional on $Z$, for all $t \in E$. The estimate \eqref{est:right-continuity} can be used to show that this is in fact true for all $t \in (0,T)$. The resulting functional $\tilde u(t)$ extends by density to a continuous linear functional on $C^1_c(\M\times\R^d)$. Then the estimate \eqref{est:right-continuity} can be used to prove that $\langle \tilde u(t), \psi \rangle$ is right continuous for all $\psi \in C^1_c(\M\times\R^d)$, not just on $Z$.

\end{proof}

Next, we construct the extension of $\tilde u$ to the boundaries $t = 0,T$.

\begin{definition}[Transport shift]\label{def:shift} Let $t\in\R$.
The operator $\cT_t : C_c(\M \times \R^d) \rightarrow C_c(\M \times \R^d)$ is defined by
\be
\cT_t \phi (x,v) = \phi(x-tv, v) .
\ee
\end{definition}

\begin{remark}[Group property]
For any $s,t \in \R$, $\cT_s \cT_t = \cT_{s+t}$.
\end{remark}

\begin{lemma}\label{lem:finite_or_infinity}
Let $u$ be a solution to \eqref{eq:app_HJ} and let $\tilde u$ be its right continuous representative, obtained in Lemma \ref{lem:app_trace_1}. Let $\psi \in C^1_c(\M \times \R^d)$ be non-negative.
Consider the function $(0,T)\ni t \mapsto \langle \cT_t \psi, \tilde u(t)\rangle$. Then
\begin{enumerate}
\item As $t$ tends to $T-$, $\langle \cT_t \psi, \tilde u(t)\rangle$ either tends to a finite limit or to positive infinity.
\item As $t$ tends to $0+$, $\langle \cT_t \psi, \tilde u(t)\rangle$ either tends to a finite limit or to negative infinity.
\end{enumerate}

\end{lemma}
\begin{proof}
Observe that
\be
(\partial_t + v \cdot \nabla_x) \cT_t \psi = 0.
\ee

It follows that
\be
\frac{\dd}{\dd t} \langle \cT_t \psi, \tilde u(t) \rangle = \langle \cT_t \psi, H(x,v, D_vu) +  \nu - \beta \rangle .
\ee
Then the negative part of the time derivative satisfies
\be
\left [ \frac{\dd}{\dd t} \langle \cT_t \psi, \tilde u(t) \rangle \right ]_- \leq \langle \cT_t \psi, C_H + \beta_+ \rangle \in L^1(0,T).
\ee
Thus $\langle \cT_t \psi, \tilde u(t) \rangle$ can be written as the difference of monotone functions, where the decreasing part is absolutely continuous on $(0,T)$ and can be extended to finite limits at the endpoints. By monotonicity, the increasing part either has a finite limit at $t=T$, or tends to positive infinity; similarly, at $t=0$ it either has a finite limit or tends to negative infinity.

\end{proof}

\begin{definition}[Weak traces] \label{def:trace} For any $\psi_T, \psi_0 \in C^1_c(\M\times \R^d)$, let
\be \label{def:traces}
\langle \psi_T, u_T \rangle := \lim_{t \to T-}  \langle \cT_{t-T} \psi_T, \tilde u(t) \rangle, \qquad \langle \psi_0, u_0 \rangle := \lim_{t \to 0 +}  \langle \cT_t \psi_0, \tilde u(t) \rangle.
\ee
These define linear maps from $C^1_c(\M \times \R^d)$ to $\R \cup \{-\infty\}$ in the case of $u_0$, and $\R \cup \{+\infty\}$ in the case of $u_T$. 
\end{definition}

We now suppose that, in addition to the weak Hamilton-Jacobi inequality in the interior \eqref{eq:app_HJ_weak}, $u$ satisfies the following: for all non-negative test functions $\phi \in C^1_c((0, T] \times \M \times \R^d)$
\be \label{app:eq:HJ_weak}
\int_0^T\int_{\M\times\R^d}u[\partial_t\phi+\diver_x(v\phi)] + \phi H(x,v,D_vu)\dd x\dd v\dd t\le \int_0^T\int_{\M\times\R^d}\beta\phi\dd x\dd v\dd t+\int_{\M\times\R^d}\beta_T\phi_T \dd x\dd v .
\ee
In our setting, we will have that $\beta_T \in L^1_\loc(\M \times\R^d)$ is a given function whose positive part satisfies $(\beta_T)_+\in L^{s'}(\M\times\R^d)$.
In this case, we show below that the time trace $u_T$, enjoys some more properties.

\begin{lemma}\label{lem:app_boundary_condition}
If $u$ satisfies \eqref{app:eq:HJ_weak} with $\b_T\in\sM(\M\times\R^d)$, then $u_T$ as defined in \eqref{def:traces} is a bounded linear functional on $C^1_c(\M \times \R^d)$ and $u_T \leq \beta_T$ in the sense of distributions: that is, for all $\psi_T \in C^1_c(\M \times \R^d)$ non-negative,
\be
\langle \psi_T, u_T \rangle \leq \langle \psi_T, \beta_T \rangle.
\ee
In particular, we have that $\langle \psi, u_T\rangle = +\infty$ does not occur for any $\psi \in C^1_c(\M\times\R^d)$.
\end{lemma}
\begin{remark}
Since then $\beta_T - u_T$ is a positive distribution, if $\beta_T \in L^1_\loc(\M\times\R^d)$ then we in fact have that $u_T$ is represented by a signed Radon measure with absolutely continuous positive part.
\end{remark}

\begin{proof}[Proof of Lemma \ref{lem:app_boundary_condition}]
In what follows, we will use the right continuous representative of $u$ constructed in Lemma \ref{lem:app_trace_1}. By the abuse of notation, we write simply $u$ for $\tilde u$.
Fix $\psi_T \in C^1_c(\M \times \R^d)$ non-negative. For each $\e >0$ small consider a smooth, non-negative test function $\eta_\e:[0,T]\to\R$, chosen such that $\eta(t) = 0$ for all $0\le t \leq T-\e$ and the derivative $\eta_\e ' $ satisfies
\be
0 \leq \eta_\e ' \leq \e^{-1}, \qquad \eta_\e ' (t) =
\begin{cases}
0 & t \in [0, T-\e] \\
\e^{-1} & t \in [T- \e +\e^2, T-\e^2],
\end{cases}
\ee
Note that as a consequence of the fundamental theorem of calculus, one has $\lim_{\e \to 0} \eta_\e(T) = 1$.

We define the following non-negative test function $\phi_\e \in C^1_c((0, T] \times \M \times \R^d)$:
\be
\phi_\e(t,x,v) = \eta_\e(t) \psi_T(x + (T-t)v,v).
\ee

Substituting this choice of $\phi$ into \eqref{app:eq:HJ_weak}, we obtain
\be \label{app:trace-inequality-1}
\int_{T-\e}^T \eta_\e '(t) \int_{\M\times\R^d} u(t)\psi_T(x + (T-t)v,v) \dd x\dd v\dd t \le \eta_\e(T) \int_{\M\times\R^d}\beta_T\psi_T \dd x\dd v + \alpha_1(\e) .
\ee
where
\be
\alpha_1(\e) = \int_{T-\e}^T\int_{\M\times\R^d}\beta\phi_\e \dd x\dd v\dd t  - \int_{T-\e}^T\int_{\M\times\R^d} \phi_\e H(x,v,D_vu)\dd x\dd v\dd t .
\ee
Note that $\lim_{\e \to0} \alpha_1(\e) = 0$, since $\beta$ and $H(x,v, D_v u)$ are locally integrable and $\phi_\e$ are bounded in $L^\infty$, uniformly in $\e$.

We exclude the possibility that $\langle \psi_T, u_T \rangle = + \infty$. Indeed, if this occurs, then for any $M > 0$, there exists $\e'$ such that for any $t \in [T- \e', T]$,
\be
M <  \int_{\M\times\R^d} u(t)\psi_T(x + (T-t)v,v) \dd x\dd v .
\ee
Then by bounding the left hand side of inequality \eqref{app:trace-inequality-1} from below we obtain that for any $\e < \e'$,
\be 
 M (1 - 2 \e) \le \eta_\e(T) \int_{\M\times\R^d}\beta_T\psi_T \dd x\dd v + \alpha_1(\e) .
\ee
Taking the limit $\e \to 0$ gives
\be 
 M \leq \int_{\M\times\R^d}\beta_T\psi_T \dd x\dd v .
\ee
Since this holds for any $M>0$, we derive a contradiction. Thus---using also Lemma \ref{lem:finite_or_infinity}---$u_T$ is in fact a linear map from $C^1_c(\M \times \R^d)$ to $\R$. We note also that the map
\be
t \mapsto \int_{\M\times\R^d} u(t) \psi_T(x + (T-t)v,v) \dd x\dd v 
\ee
extends to a function that is bounded and continuous (from the left) at $t = T$.

Next, we show that $u_T \leq \beta_T$ as functionals on $C^1_c(\M \times \R^d)$.
We have
\be  \label{app:trace-inequality-2}
\frac{1}{\e} \int_{T-\e}^T\int_{\M\times\R^d} u\psi_T(x + (T-t)v,v) \dd x\dd v \dd t \le \eta_\e(T) \int_{\M\times\R^d}\beta_T\psi_T \dd x\dd v + \alpha_2(\e) .
\ee
where
\be
\alpha_2(\e) := \alpha_1(\e) + \int_{T-\e}^T (\e^{-1} - \eta_\e ') \int_{\M\times\R^d} u\psi_T(x + (T-t)v,v) \dd x\dd v \dd t  .
\ee

For the second term here we have
\begin{align}
\left \lvert \int_{T-\e}^T (\e^{-1} - \eta_\e ')  \int_{\M\times\R^d} u\psi_T(x + (T-t)v,v) \dd x\dd v \dd t  \right \rvert & \leq \left \lvert \int_{T-\e}^{T-\e + \e^2} |\e^{-1} - \eta_\e '| |\langle u(t), \psi_T(x+(T-t)v, v) \rangle| \dd t \right \rvert \\
& \qquad + \left \lvert \int_{T-\e^2}^{T} |\e^{-1} - \eta_\e '| |\langle u(t), \psi_T(x+(T-t)v, v) \rangle| \dd t \right \rvert \\
& \leq 2 \e \| \langle u(t), \psi_T(x+(T-t)v, v) \rangle \|_{L^\infty[T-\e,T]},
\end{align}
which converges to zero as $\e \to 0$ since the trajectory $\langle u(t), \psi_T(x+(T-t)v, v)  \rangle$ is bounded near $t=T$. Thus $\lim_{\e \to 0} \alpha_2(\e) = 0$.

Taking the limit $\e \to 0$ in inequality \eqref{app:trace-inequality-2}, we conclude that
\be
\langle \psi_T, u_T \rangle \leq \langle \psi_T, \beta_T \rangle.
\ee
Since $\beta_T - u_T$ is a positive linear functional on $C^1_c(\M \times \R^d)$, it is bounded, and therefore $u_T$ is also a bounded linear functional on $C^1_c(\M \times \R^d)$.
\end{proof}

\begin{corollary} \label{cor:nufinite}
If $u$ satisfies \eqref{app:eq:HJ_weak} with $\b_T\in\sM(\M\times\R^d)$ then, in the notation of {\Blue equation~\eqref{eq:HJ-deficit-measure}}, the measure $\nu$ extends to a finite Radon measure on $(0,T] \times \M\times\R^d$ given by $\nu(A) = \nu(A\one_{(0,T)})$.
\end{corollary}
\begin{proof}
We show that, for any non-negative test function $\phi \in C^1_c((0,T]\times\M\times\R^d)$,
\be \label{nu-finite}
\langle\nu ,\phi \one_{(0,T)}\rangle < + \infty .
\ee

It suffices to prove \eqref{nu-finite} for test functions of the form $\phi(t,x,v) = \theta(t) \cT_t\psi(x,v)$, where $\theta \in C^1_c(0,T]$ and $\psi \in C^1_c(\M\times\R^d)$.

Then
\be
\frac{\dd}{\dd t} \langle \phi, u(t) \rangle = \theta '(t) \langle \cT_t\psi, u \rangle + \langle \phi, H(x,v, D_vu) +  \nu - \beta \rangle \geq \theta '(t) \langle \cT_t\psi, u \rangle + \langle \phi, \nu - C_H - \beta_+ \rangle .
\ee
It follows that (once again using the right continuous version of $u$), for $t < T$,
\be
\langle\nu,  \phi \one_{(0,t]} \rangle \leq \langle \phi, u(t) \rangle - \int_0^t \int_{\M\times\R^d} \theta'(s) \langle \cT_s\psi, u \rangle + \phi (C_H + \beta_+)  \dd x \dd v \dd s .
\ee
Taking the limit $t \to T$, by definition of $u_T$,
\be
\langle\nu, \phi \one_{(0,T)} \rangle \leq \langle \phi, u_T \rangle - \int_0^T \int_{\M\times\R^d} \theta'(s) \langle \cT_s\psi, u \rangle + \phi (C_H + \beta_+)  \dd x \dd v \dd s < + \infty .
\ee

\end{proof}

We now discuss the trace of $u$ at $t=0$: $u_0$ as defined in Definition~\ref{def:traces}. $\langle \psi, u_0\rangle $ is defined for all $\psi \in C^1_c(\M \times \R^d)$. Our aim is to give a meaning to the quantity $\langle m_0, u_0 \rangle$, which appears in the definition of the functional $\widetilde{\cA}$. In the case where $m_0 \in C^1_c(\M \times \R^d)$ this is straightforward, noting that we allow the possible value $-\infty$. We now consider the more general case where $m_0 \in C(\M \times \R^d)$.

\begin{lemma} \label{lem:trace0}
Assume that, for all $\phi \in C^1_c(\{ m_0 > 0 \})$, $\langle \phi, u_0 \rangle \neq - \infty$. Then $u_0$ is represented by a Radon measure on $\{ m_0 > 0\}$. Furthermore, the positive part $(u_0)_+$ has the property that
\be
\int_{\{m_0 > 0\}} m_0 \dd (u_0)_+(x,v) < + \infty .
\ee
\end{lemma}
\begin{proof}
Let $\phi \in C^1_c(\M\times\R^d)$ be non-negative. Since
\be
 \frac{\dd}{\dd t} \langle \cT_t \phi, \tilde u(t) \rangle \geq - \langle \cT_t \phi, C_H + \beta_+ \rangle,
\ee
we have
\be
\langle \phi, u_0 \rangle \leq \int_0^T \langle \cT_t \phi, C_H + \beta_+ \rangle \dd t + \langle \cT_T \phi, \beta_T \rangle = : S\phi .
\ee

The right hand side is linear in $\phi$ and satisfies
\be
\left \lvert \int_0^T \langle \cT_t \phi, C_H + \beta_+ \rangle \dd t + \langle \cT_T \phi, \beta_T \rangle\right\rvert \leq \| \phi \|_{L^\infty} \left ( \| C_H + \beta_+ \|_{L^1(K_{[0,T]})} + \| \beta_T \|_{L^1(K_T)} \right );
\ee
here $K_T$ denotes the set
\be
K_T : = \{ (x+vT, v) : (x,v) \in K\},
\ee
where $K$ is the support of $\phi$, and $K_{[0,T]}$ is the set 
\be
K_{[0,T]} : = \{ (x+vt, v) : (x,v) \in K, t \in [0,T]\} .
\ee
Thus $S$ defines a bounded linear functional on $C_c(\M \times \R^d)$. In particular it is a distribution; moreover, it is represented by a signed Radon measure.

Observe next that $S - u_0$ is a positive linear functional on $C^1_c(\{ m_0 > 0 \})$, and thus bounded and a distribution. By positivity it is given by a Radon measure $\nu_0$ on $\{ m_0 > 0 \}$. We deduce that
\be
u_0 = S - \nu_0 .
\ee
That is, $u_0$ is a signed Radon measure. 

Moreover, from the definition of the Hahn-Jordan decomposition we have the following estimate for the positive part:
\be
\langle \phi, (u_0)_+ \rangle \leq \int_0^T \langle \cT_t \phi, C_H + \beta_+ \rangle \dd t + \langle \cT_T \phi, (\beta_T)_+ \rangle .
\ee

Let $\phi_n \in C_c(\M \times \R^d)$ be an increasing sequence of functions such that $\phi_n$ converges to $m_0$ as $n \to \infty$. Since
\begin{multline}
\sup_n \langle \phi_n, (u_0)_+ \rangle \leq \int_0^T \langle \cT_t m_0, C_H + \beta_+ \rangle \dd t + \langle \cT_T m_0, (\beta_T)_+ \rangle \leq C_H T \| m_0\|_{L^1} \\ + T^{1/q} \|\beta_+\|_{L^{q'}} \| m_0 \|_{L^q} + \|(\beta_T)_+\|_{L^{s'}} \|m_0\|_{L^s},
\end{multline}
we conclude that $\langle m_0, (u_0)_+\rangle$ is finite.

\end{proof}

Based on the previous lemma, we make the following definition.
\begin{definition} \label{def:u0m0}
We define $- \langle m_0, u_0 \rangle$ as follows:
\begin{enumerate}
\item If there exists $\phi \in C^1_c(\{m_0 >0\})$ such that $\langle \phi , u_0 \rangle = -\infty$, then we define 
\be
- \langle m_0, u_0 \rangle = +\infty .
\ee
\item Otherwise, let $\phi_n \in C_c(\{m_0>0\})$ be an increasing sequence of functions such that $\phi_n$ converges to $m_0$ as $n \to \infty$, and define
\be
- \langle m_0, u_0 \rangle = -\lim_{n \to \infty} \langle \phi_n, u_0 \rangle .
\ee

This is well-defined (allowing for the possible value $+\infty$) by Lemma~\ref{lem:trace0}.
\end{enumerate}
\end{definition}

\begin{lemma}\label{lem:app_u0m0_mean}
Suppose that the assumptions of Lemma \ref{lem:app_boundary_condition} hold and suppose in addition that 
\begin{align}\label{eq:HJ_weak_0T}
\int_0^T\int_{\M\times\R^d}u[\partial_t\phi+\diver_x(v\phi)] &+ \phi H(x,v,D_vu)\dd x\dd v\dd t\le \int_0^T\int_{\M\times\R^d}\beta\phi\dd x\dd v\dd t\\
\nonumber&+\int_{\M\times\R^d}\beta_T\phi_T \dd x\dd v - \int_{\M\times\R^d}\b_0\phi_0 \dd x\dd v,
\end{align}
holds for all $\phi\in C_c^1(\cU_{m_0})$, where $\b_0\in\sM(\{m_0>0\})$ is also given. Then for the trace $u_0$ of the right continuous version of $u$ we have
$$\b_0\le u_0,\ \ {\rm{in}}\ \sD'(\{m_0>0\}),$$
and in particular $\langle u_0,\psi\rangle\neq -\infty$ for any $\psi\in C_c^1(\{m_0>0\}).$

If in addition we suppose that $\b_0$ is such that  $\langle\b_0,m_0\rangle$ is meaningful and finite, then $\langle u_0,m_0\rangle$ is finite and 
$$\langle\b_0,m_0\rangle\le \langle u_0,m_0\rangle.$$
\end{lemma}

\begin{proof}
The proof of this result follows the same lines as the proof of Lemma \ref{lem:app_boundary_condition}, so we point out only the main differences. 
Let $u$ stand for the right continuous representative constructed in Lemma \ref{lem:app_trace_1}. Fix $\psi_0 \in C^1_c(\{m_0>0\})$ non-negative. For each $\e >0$ small consider a smooth, non-negative test function $\eta_\e:[0,T]\to\R$, chosen such that $\eta(t) = 0$ for all $\e\le t \leq T$ and the derivative $\eta_\e ' $ satisfies
\be
-\e^{-1} \leq \eta_\e ' \leq 0, \qquad \eta_\e ' (t) =
\begin{cases}
-\e^{-1} & t \in [\e^2, \e-\e^2],\\
0 & t \in [\e, T].
\end{cases}
\ee
Note that as a consequence of the fundamental theorem of calculus, one has $\lim_{\e \to 0} \eta_\e(0) = 1$.

We define the following non-negative test function $\phi_\e \in C^1_c(\cU_{m_0})$:
\be
\phi_\e(t,x,v) = \eta_\e(t) \psi_T(x -tv,v).
\ee

Substituting this choice of $\phi$ into \eqref{eq:HJ_weak_0T}, we obtain
\be \label{app:trace-inequality-1-2}
\int_{0}^\e \eta_\e '(t) \int_{\M\times\R^d} u(t)\psi_T(x -tv,v) \dd x\dd v\dd t \le -\eta_\e(0) \langle\beta_0,\psi_0\rangle + \alpha_1(\e) .
\ee
where
\be
\alpha_1(\e) = \int_{0}^\e\int_{\M\times\R^d}\beta\phi_\e \dd x\dd v\dd t  - \int_{0}^\e\int_{\M\times\R^d} \phi_\e H(x,v,D_vu)\dd x\dd v\dd t .
\ee
As before, we note that $\lim_{\e \to0} \alpha_1(\e) = 0.$
We exclude the possibility that $\langle u_0,\psi_0 \rangle = - \infty$. For this, we rewrite the previous inequality as 
$$
\eta_\e(0) \langle\beta_0,\psi_0\rangle - \alpha_1(\e)\le-\int_{0}^\e \eta_\e '(t) \int_{\M\times\R^d} u(t)\psi_T(x -tv,v) \dd x\dd v\dd t,
$$
and use the same arguments as when proving $\langle u_T,\psi_T\rangle \neq +\infty$ in the proof of Lemma  \ref{lem:app_boundary_condition}.
 
Therefore, $u_0$ defines a linear map on $C^1_c(\{m_0>0\})$. Having this, we can show the inequality $\langle\b_0,\psi_0\rangle\le \langle u_0,\psi_0\rangle$ in the same way as corresponding inequality in Lemma  \ref{lem:app_boundary_condition}.

Now, using the Definition \ref{def:u0m0}, $\langle u_0, m_0\rangle$ is meaningful, having also the possibility that it is $-\infty$. However, if the additional assumption that  $\langle\b_0,m_0\rangle$  is finite takes place, taking a an increasing sequence of test functions,  we find that 
$ \langle\b_0,m_0\rangle\le \langle u_0,m_0\rangle,$ so clearly, the latter term cannot be $-\infty$.
\end{proof}

{\Blue 
Through similar arguments it is possible to justify the existence of weak time traces for competitors $m$ in Problem~\ref{prob:density}, thereby giving meaning to the initial value problem
\be\label{eq:continuity_main_app}
\begin{cases} \partial_t m + v \cdot D_x m + \diver_v w = 0,\ \  {\rm{in}}\ \   \sD'((0,T)\times\M\times\R^d)  \\
m \vert_{t=0} = m_0 .
\end{cases}
\ee
Recall that in Remark~\ref{rmk:m_trace} we established that, in the cases of interest to us, there exists a function $V\in L^{r'}(m\dd x\dd v\dd t)$ such that $w=V m$, and so we may assume that $m$ is a distributional solution of the following equation:
$$
\partial_t m + \diver_{x}(vm)+ \diver_v (Vm) = 0,
$$
with $|V| m \in L^1([0,T] \times \M \times \R^d)$. 

This setting is much more standard since here the time derivatives $\frac{\dd}{\dd t} \langle \phi, m_t \rangle$ will be in $L^1[0,T]$ for any $\phi \in C^1_c(\M \times \R^d)$ rather than measures, that is, we expect absolutely continuous rather than right continuous trajectories. Moreover we can work on the whole space $[0,T] \times \M \times \R^d$ rather than only the reachable set $\cU_{m_0}$.

Deducing that $m$ has a narrowly continuous representative is essentially an application of \cite[Lemma 8.1.2]{AmbGigSav}. However, since we do not necessarily have
\be
\int_0^T \int_{\M \times \R^d} |v| m \dd x \dd v \dd t < + \infty,
\ee
due to the unbounded drift $v$, we cannot immediately apply this lemma. Below we briefly sketch the adaptation to our case.

\begin{lemma}[See~{\cite[Lemma 8.1.2]{AmbGigSav}}] \label{app:lem:m_trace}
Let $0 \leq m \in (L^1 \cap L^q)([0,T] \times \M \times \R^d)$ be a {\color{black}non-negative} function satisfying
\be \label{eq:m_appendix}
\partial_t m + \diver_{x}(vm)+ \diver_v (Vm) = 0
\ee
in the sense of distributions on $(0,T) \times \M \times \R^d$, {\color{black} where $V$ is given, such that $|V| m \in L^1([0,T] \times \M \times \R^d)$}.

Then there exists a continuous curve $\tilde m_{\bullet} : [0,T] \to (C^1_c(\M \times \R^d))'$ such that $\tilde m_t = m_t$ for almost all $t \in [0,T]$. Thus {\color{black}$\tilde m_0$} is well-defined as an element of $(C^1_c(\M \times \R^d))'$ (or in fact, by positivity, a Radon measure).

Furthermore, if {\color{black}$\tilde m_0$ is a probability measure}, then $\tilde m_\bullet$ extends uniquely to a narrowly continuous curve in the space of probability measures, i.e. $\tilde m_\bullet \in C([0,T] ; \cP(\M \times \R^d))$.

\end{lemma}
\begin{proof}
Since $m, |V| m \in L^1([0,T] \times \M \times \R^d)$, for any compact set $K \subset \M \times \R^d$ we have
\be \label{eq:app_loc_finite}
\int_0^T \int_{K} (|v| + |V|) m \dd x \dd v \dd t < + \infty .
\ee
It follows that, as in the proof of \cite[Lemma 8.1.2]{AmbGigSav}, we may select a dense subset $\{ \phi_n \}_{n \in \N}$ of $C^1_c(\M \times \R^d)$ and take a version $\tilde m_t$ of $m_t$ such that $t \mapsto \langle \phi_n, \tilde m_t \rangle$ is continuous with respect to $t$ for all $n$ and $\tilde m_t = m_t$ for almost all $t \in [0,T]$, and a define a unique weak-$\ast$ continuous extension of {\color{black}$\tilde m_\bullet$} to $(C^1_c(\M \times \R^d))'$. Thus $\tilde m_0$ is well-defined as the unique element of $(C^1_c(\M \times \R^d))'$ satisfying
\be
\langle \phi, \tilde m_0 \rangle : = \lim_{t \to 0} \langle \phi, \tilde m_t \rangle \quad \text{for all} \; \phi \in C^1_c(\M \times \R^d).
\ee
Moreover, since $\tilde m$ is {\color{black}non-negative}, in fact $\tilde m_t$ is a Radon measure on $\M \times \R^d$ for all $t \in [0,T]$.

Furthermore, it follows from the continuity of $\tilde m_\bullet$ in the weak-$\ast$ sense of $(C^1_c(\M \times \R^d))'$, and the fact that $\tilde m$ is locally finite, that, for any (N.B. now time-dependent) $\phi \in C^1_c([0,T] \times \M \times \R^d)$, the path $t \mapsto \langle \phi(t, \cdot), \tilde m_t \rangle$ is continuous. Thus we may also use the final argument from \cite[Lemma 8.1.2]{AmbGigSav}
(similar to our argument for the time traces at the boundary in Lemma~\ref{lem:app_boundary_condition}) to prove the following equality (c.f. \cite[{Equation (8.1.4)}]{AmbGigSav}): for any $\phi \in C^1_c([0,T] \times \M \times \R^d)$ and any $0 \leq t_1 \leq t_2 \leq T$,
\be \label{eq:app_AGS814}
\langle \phi(t_2, \cdot), \tilde m_{t_2} \rangle - \langle \phi(t_1, \cdot), \tilde m_{t_1} \rangle = \int_{t_1}^{t_2} \langle \partial_t \phi + v \cdot D_x \phi + V \cdot D_v \phi, \tilde m_t \rangle \dd t .
\ee

Next we wish to show that, if $\tilde m_0$ is a probability measure, then $\tilde m_t$ is a probability measure for all $t \in [0,T]$. If this is the case, then we may apply \cite[Remark 5.1.6]{AmbGigSav}---if $\tilde m_{t_n} \to \tilde m_t$ in the sense of distributions as $n$ tends to infinity, then this convergence also holds in the narrow sense---to deduce that $\tilde m_t$ is a narrowly continuous path in the space of probability measures, as desired.

To do this we use the argument of Lemma~\ref{lem:finite_or_infinity} to avoid the need for $v m$ to be integrable. First, fix a sequence $\zeta_R \in C^\infty_c(\M \times \R^d)$ of smooth, compactly supported functions, approximating the constant function $1$ in a monotone limit as $R$ tends to infinity---i.e., let $\zeta_R$ satisfy the assumptions given in equation \eqref{hyp:zeta}. We note in particular that $\| D \zeta_R \|_{L^\infty} \leq C/R$ for some constant $C>0$ independent of $R$.
Then, for each $t_\ast$, consider the test function $\cT_{t - t_\ast} \zeta_R$ {(recall $\cT$ from Definition \ref{def:shift})}, which satisfies
\be
(\partial_t + v \cdot D_x) \cT_{t - t_\ast} \zeta_R = 0, \quad \cT_{t - t_\ast} \zeta_R \vert_{t = t_\ast} = \zeta_R .
\ee
Using \eqref{eq:app_AGS814} with $t_1 = 0$, $t_2 = t_\ast$ and $\phi = \cT_{t - t_\ast} \zeta_R$, we find that
\be
\langle  \zeta_R, \tilde m_{t_\ast} \rangle = \langle \cT_{- t_\ast} \zeta_R, \tilde m_0 \rangle + \int_0^{t_\ast} \langle D_v ( \cT_{t - t_\ast} \zeta_R ), {V} \tilde m_t  \rangle  \dd t.
\ee
We observe that, for all $t \in [0, t_\ast]$,
\be
|D_v ( \cT_{t - t_\ast} \zeta_R )| = | [(t_\ast - t) D_x + D_v] \zeta_R| \leq (1 + t_\ast) \| D \zeta_R \|_{L^\infty} \leq \frac{C (1 + t_\ast)}{R} .
\ee
Thus (since $\tilde m_t = m_t$ for almost all $t$)
\be
\left | \int_0^{t_\ast} \langle D_v ( \cT_{t - t_\ast} \zeta_R ), {V} \tilde m_t  \rangle  \dd t \right | \leq \frac{C (1 + t_\ast)}{R} \int_0^T \int_{\M \times \R^d} |V| m_t \dd x \dd v \dd t .
\ee
The right hand side tends to zero as $R$ tends to infinity, since $|V| m \in L^1([0,T] \times \M \times \R^d)$. Moreover, since $\tilde m_0$ is a probability measure and $\cT_{- t_\ast} \zeta_R$ increases monotonically to $1$ pointwise as $R$ tends to infinity, it follows that
\be
\lim_{R \to 0} \langle \cT_{- t_\ast} \zeta_R, \tilde m_0 \rangle = 1 .
\ee

Finally, since $\tilde m_{t_\ast}$ is a Radon measure and $\lim_{R \to \infty} \langle \zeta_R, \tilde m_{t_\ast} \rangle$ is a monotone limit,
\be
\tilde m_{t_\ast} (\M \times \R^d) = \lim_{R \to \infty} \langle \zeta_R, \tilde m_{t_\ast} \rangle = 1.
\ee
That is, $\tilde m_{t_\ast}$ is a probability measure for all $t_\ast \in [0,T]$. This completes the proof.

\end{proof}

}

\section{Truncations and Maxima}\label{app:2}

Given a distributional solution to the Hamilton-Jacobi inequality
\begin{align}\label{eq:HJ-truncsec}
-\partial_t u - v\cdot D_x u+H(x,v,D_v u) \leq \beta, \quad &\text{in } \sD'((0,T)\times\M\times\R^d) \\
u_T \leq \beta_T, \quad &\text{in } \sD'(\M\times\R^d) .
\end{align}
in the sense of Definition~\ref{def:solution}, we show that the truncations of $u$ from below, i.e. the functions $\max\{ u , l\}$ for some $l < 0$, satisfy a similar inequality. In a similar vein, we also show that given $u^1$ and $u^2$ both satisfying \eqref{eq:HJ-truncsec}, their maximum satisfies the same inequality \eqref{eq:HJ-truncsec}.

\begin{lemma} \label{lem:truncation}
Let $u \in L^1_\loc((0,T) \times \M \times \R^d)$ satisfy \eqref{eq:HJ-truncsec} in the sense of distributions. 
Assume that $\beta \in L^1_\loc((0,T) \times \M \times \R^d)$ and $D_v u \in L^r_\loc((0,T) \times \M \times \R^d)$.
Then $u_l : = (u-l)_+$ satisfies
\begin{align} \label{eq:HJ-trunc-l}
-\partial_t u_l - v\cdot D_x u_l+H(x,v,D_v u_l) \one_{\{ u > l\}} \leq \beta \one_{\{ u > l\}}, \quad &{\rm{in}\ } \sD'((0,T)\times\M\times\R^d), \\
[(u-l)_+]_T \leq (\beta_T-l)_+ \quad &{\rm{in\ }} \sD'(\M\times\R^d).
\end{align}
\end{lemma}
A similar result holds for the truncation $u \vee l = (u-l)_+ + l$. Moreover it suffices to consider the case $l=0$.

\begin{lemma} \label{lem:maximum}
Let $u_1, u_2$ satisfy Assumption~\ref{hyp:app-summ} and \eqref{eq:HJ-truncsec}. Then $u = \max\{ u_1,u_2\}$ also satisfies \eqref{eq:HJ-truncsec}.
\end{lemma}

The result here is in the spirit of renormalisation \cite{DiPerna-Lions}. Bouchut \cite[Theorem 1.1]{Bouchut} proved a chain rule for the kinetic transport operator, i.e. the identity
\be \label{eq:chain-rule}
(\partial_t + v \cdot D_x) h(u) = h'(u) (\partial_t + v \cdot D_x) u 
\ee
that applies when $h$ is a Lipschitz function and $\partial_t u + v \cdot D_x u \in L^1_\loc$.
However, since in our case $\partial_t u + v \cdot D_x u $ may only be a measure, we are not able to use this result directly, or indeed prove a chain rule with equality as in equation \eqref{eq:chain-rule}. 
Nevertheless, the ideas of the proofs in \cite{Bouchut, GraMesSilTon} can be used to obtain the inequality that is sufficient for our case.

The argument proceeds in several steps.

\subsection{Extension}

We define the following time shift and extension of $u$ on the time interval $(-2\eta, T+2\eta)$ for $\eta > 0$:
\be
\tilde u(t,x,v) = \begin{cases}
u(t,x,v), & t \in (0,T) \\
0, & t \in [T, T+2\eta].
\end{cases}
\ee
Then, defining
\be
\tilde \beta(t,x,v) = \begin{cases}
\beta(t,x,v), & t \in (0,T) \\
H(x,v,0), & t \in [T, T+2\eta],
\end{cases}
\ee
$\tilde u$ satisfies the following inequality in the sense of distributions on $(0, T+2\eta) \times \M \times \R^d$ (see the similar construction in \cite[Section 6.3]{CarGraPorTon}):
\begin{align}\label{eq:HJ-truncsec-extended}
-\partial_t \tilde u - v\cdot D_x \tilde u+H(x,v,D_v \tilde u) \leq \tilde \beta + \beta_T \, \partial_t \one_{[T, T+2\eta]} , \quad &\text{in } \sD'((0,T+2\eta)\times\M\times\R^d) .
\end{align}

\subsection{Fenchel's Inequality}

Since $L$ is the Fenchel conjugate of $H$, for any continuously differentiable vector field $a \in C^1_b$ we have
\begin{align}\label{eq:HJ-truncsec-Fenchel}
-\partial_t \tilde u - v\cdot D_x \tilde u - a \cdot D_v \tilde u \leq  L(x,v, -a) + \tilde\beta + \beta_T \, \partial_t \one_{[T, T+2\eta]} , \quad &\text{in } \sD'((0,T+2\eta)\times\M\times\R^d) .
\end{align}

\subsection{Regularisation}

Fix non-negative, symmetric, unit mass mollifiers $\chi,\psi \in C^\infty_c(\R^d)$ and $\theta \in C^\infty_c(\R)$. 
Assume that $\theta$ is supported on the set $[-1,1]$.
Then define, for $\e, \delta, \eta > 0$,
\be
\chi_\e \left ( v \right ) : = \e^{-d} \chi \left ( \frac{v}{\e} \right ), \quad \psi_\delta(x) : = \delta^{-d} \psi \left ( \frac{x}{\delta} \right ) , \quad \theta_\eta(t) : = \eta^{-1} \theta \left ( \frac{t}{\eta} \right ) .
\ee
Then define the full mollifier $\varphi$ by
\be
\varphi(t,x,v) = \theta_\eta(t) \psi_\delta(x) \chi_\e(v) .
\ee
Then the regularisation $u_{\eta,\e,\delta} : = \tilde u \ast \varphi$ satisfies the following inequality, in a pointwise sense on $(\eta, T+\eta] \times \M \times \R^d$:
\be
-(\partial_t +v\cdot D_x + a \cdot D_v) u_{\eta,\e,\delta} \leq L_{\eta,\e,\delta} + \beta_{\eta,\e,\delta} + \varphi(t-T, \cdot) \ast_{x,v} \beta_T +  \cE_{\eta,\e,\delta} ,
\ee
where $\beta_{\eta,\e,\delta} : = \tilde \beta \ast \vphi$ and $L_{\eta,\e,\delta} : = L(x,v, -a)  \ast \vphi$, while $ \cE_{\eta,\e,\delta}$ denotes the commutator
\be
 \cE_{\eta,\e,\delta} : = \theta_\eta \ast_t \left [ \chi_\e \ast_v \psi_\delta \ast_{x} (v \cdot D_x + a \cdot D_v)  u  - (v \cdot D_x + a \cdot D_v ) \chi_\e \ast_v \psi_\delta \ast_{x} u \right ].
\ee

With the choice $\e = \delta^2$ this error converges to zero in $L^1_\loc$ by Lemma~\ref{lem:commutator} and \cite{DiPerna-Lions}.

\subsection{The Maximum Function}

We fix a smooth approximation of the functions $x_+ = \max\{x,0\}$ and $\max\{x_1, x_2\}$.
First, for each $\alpha > 0$ we fix a smooth function $\g_\alpha(x)$ approximating $x_+$ in such a way that $0 \leq \g_\alpha(x) \leq x_+$ and $0 \leq \g_\alpha'(x) \leq 1$ for all $\alpha > 0$ and $\lim_{\alpha\to0} \g_\alpha(x) = x_+$ for all $x$ and $\lim_{\alpha\to0} \g_\alpha '(x) = 1$ for $x > 0$. Note in particular that then $\g_\alpha ' (x) = 0$ for all $\alpha > 0$ and $x \leq 0$, so that $\g_\alpha ' $ converges pointwise to the function $\one_{\{x > 0\}}$ as $\alpha$ tends to zero.

We similarly define an approximation $h_\alpha$ of the maximum function by
\be
h_\alpha(x_1, x_2) : = x_2 + \g_\alpha(x_1 - x_2) .
\ee
Observe that $h_\alpha$ satisfies $0 \leq \partial_{x_i} h_\alpha \leq 1$ for $i=1,2$, and 
\be
\partial_{x_1} h_\alpha(x_1, x_2) + \partial_{x_2}  h_\alpha(x_1, x_2) = 1.
\ee

\subsection{Inequality for the Truncations and Maximum}

Since $\g_\alpha '$ is non-negative,
\be \label{HJ-ineq-truncation-smoothed}
-(\partial_t + v\cdot D_x + a \cdot D_v) u_{\eta,\e,\delta} \, \g_\alpha '(u_{\eta,\e,\delta}) \leq \left [\beta_{\eta,\e,\delta} + L_{\eta,\e,\delta} + \varphi(t-T, \cdot) \ast_{x,v} \beta_T  + \cE_{\eta,\e,\delta} \right ] \g_\alpha '(u_{\eta,\e,\delta}).
\ee
Thus, since $\g_\alpha \in C^1$ and $u_{\eta,\e,\delta}$ is smooth in all variables, applying the usual chain rule
\be
 -(\partial_t + v\cdot D_x + a \cdot D_v) \g_\alpha (u_{\eta,\e,\delta}) \leq \left [\beta_{\eta,\e,\delta} + L_{\eta,\e,\delta} + \varphi(t-T, \cdot) \ast_{x,v} \beta_T + \cE_{\eta,\e,\delta} \right ] \g_\alpha '(u_{\eta,\e,\delta}).
\ee

Similarly, given two subsolutions $u^1$ and $u^2$,
\begin{multline}
-(\partial_t + v\cdot D_x + a \cdot D_v) h_\alpha (u^1_{\eta,\e,\delta}, u^2_{\eta,\e,\delta})\\
 \leq \left [\beta_{\eta,\e,\delta} + L_{\eta,\e,\delta} + \varphi(t-T, \cdot) \ast_{x,v} \beta_T + \cE^1_{\eta,\e,\delta} \right ] \partial_1 h_\alpha (u^1_{\eta,\e,\delta}, u^2_{\eta,\e,\delta}) \\ 
 + \left [\beta_{\eta,\e,\delta} + L_{\eta,\e,\delta}+ \varphi(t-T, \cdot) \ast_{x,v} \beta_T  + \cE^2_{\eta,\e,\delta} \right ] \partial_2 h_\alpha (u^1_{\eta,\e,\delta}, u^2_{\eta,\e,\delta}).
\end{multline}
Thus
\begin{multline}
-(\partial_t + v\cdot D_x + a \cdot D_v) h_\alpha (u^1_{\eta,\e,\delta}, u^2_{\eta,\e,\delta}) \leq \beta_{\eta,\e,\delta} + L_{\eta,\e,\delta} + \varphi(t-T, \cdot) \ast_{x,v} \beta_T \\ + \cE^1_{\eta,\e,\delta} \partial_1 h_\alpha (u^1_{\eta,\e,\delta}, u^2_{\eta,\e,\delta}) +  \cE^2_{\eta,\e,\delta} \partial_2 h_\alpha (u^1_{\eta,\e,\delta}, u^2_{\eta,\e,\delta}).
\end{multline}

\subsection{Limits}

We now take the limit as the smoothing parameters tend to zero in the previous inequalities. This procedure yields the proofs of the Lemmas \ref{lem:truncation} and \ref{lem:maximum}. We continue to choose $\e = \delta^2$ to ensure convergence of the commutator. We detail the procedure in the case of the truncation \eqref{HJ-ineq-truncation-smoothed}; the case of the maximum function is similar.

\begin{proof}[Proof of Lemma \ref{lem:truncation}]

We first test the inequality \eqref{HJ-ineq-truncation-smoothed} with an arbitrary non-negative smooth function $\zeta \in C^\infty_c((0,T] \times \M \times \R^d)$. We fix an extension of $\zeta$ to a function $\zeta \in C^\infty_c((0,T+1] \times \M \times \R^d)$ and consider integrating over $[0,T+\eta]\times \M \times \R^d$. For all $\eta > 0$ small enough that the support of $\zeta$ is contained in $(\eta, T+1] \times \M \times \R^d$,
\begin{multline}
\int_0^{T+\eta} \int_{\M \times \R^d} \g_\alpha (u_{\eta,\e,\delta}) (\partial_t \zeta + v\cdot D_x \zeta + \diver_v(a \zeta)) \dd x \dd v \dd t\\
 \leq \int_0^{T+\eta} \int_{\M \times \R^d} \left [\beta_{\eta,\e,\delta} + L_{\eta,\e,\delta} + \cE_{\eta,\e,\delta} \right ] \g_\alpha '(u_{\eta,\e,\delta}) \zeta \dd x \dd v \dd t\\
+ \int_{T-\eta}^{T+\eta} \int_{\M \times \R^d} \g_\alpha' (u_{\eta,\e,\delta}(t,x,v)) \theta_\eta(T-t) \chi_\e \ast_v \psi_\delta \ast_x \beta_T   \zeta(t,x,v) \dd x \dd v  \dd t.
\end{multline}
We have used that $u_{\eta,\e,\d}(T+\eta,x,v) = 0$.
Since $0 \leq \g_\alpha ' \leq 1$, we may estimate the boundary term from above to obtain
\begin{multline}
\int_0^{T+\eta} \int_{\M \times \R^d} \g_\alpha (u_{\eta,\e,\delta}) (\partial_t \zeta + v\cdot D_x \zeta + \diver_v(a \zeta)) \dd x \dd v \dd t\\
 \leq \int_0^{T+\eta} \int_{\M \times \R^d} \left [\beta_{\eta,\e,\delta} + L_{\eta,\e,\delta} + \cE_{\eta,\e,\delta} \right ] \g_\alpha '(u_{\eta,\e,\delta}) \zeta \dd x \dd v \dd t\\
+ \int_{\M \times \R^d} [\beta_T(x,v)]_+ \varphi \ast \zeta (T,x,v)\dd x \dd v .
\end{multline}

Since $\tilde u, \tilde \beta, L(x,v, -a) \in L^1_\loc$, $u_{\eta,\e,\d}, \beta_{\eta,\e,\d}, L_{\eta,\e,\d}$ converge respectively to these strongly in $L^1_\loc$ by standard results on convolutions. We have already noted that $ \cE_{\eta,\e,\delta}$ converges to zero in $L^1_\loc$.
We therefore also obtain pointwise convergence along a subsequence.
Similarly, $\varphi \ast \zeta$ converges to $\zeta$ pointwise since $\zeta$ is smooth.
From this we obtain convergence of all terms, by continuity of $\g_\alpha$ and $\g_\alpha '$ and applying dominated convergence. Hence we obtain the following inequality:
\begin{multline}
\int_0^{T} \int_{\M \times \R^d} \g_\alpha (u) (\partial_t \zeta + v\cdot D_x \zeta + \diver_v(a \zeta)) \dd x \dd v \dd t \leq \int_0^{T} \int_{\M \times \R^d} \left [\beta + L(x,v,-a) \right ] \g_\alpha '(u) \zeta \dd x \dd v \dd t\\
+ \int_{\M \times \R^d} [\beta_T(x,v)]_+ \zeta (T,x,v)\dd x \dd v .
\end{multline}

The convergences as $\alpha \to 0$ all follow by dominated convergence: for example, since $u \in L^1_\loc((0,T]\times\M\times\R^d)$ and $\zeta$ has support contained in $(0,T]\times\M\times\R^d$, we have
\be
\g_\alpha (u) (\partial_t \zeta+ v\cdot D_x \zeta + \diver_v(a \zeta)) \leq |u| | (\partial_t + v\cdot D_x) \zeta | \in L^1 .
\ee
A similar argument is used for the term involving $\beta_T$.

For the remaining term, use that $|\g_\alpha'| \leq 1$ (the bound being uniform in $\alpha > 0$), and both $\beta$ and $H(x,v,D_v u)$ are in $L^1_\loc((0,T]\times\M\times\R^d)$ by assumption. Then
\begin{multline}
\int_0^{T} \int_{\M \times \R^d} u_+ (\partial_t \zeta + v\cdot D_x \zeta + \diver_v(a \zeta)) \dd x \dd v \dd t \leq \int_0^{T} \int_{\M \times \R^d} \left [\beta + L(x,v,-a) \right ] \one_{\{u > 0\}} \zeta \dd x \dd v \dd t\\
+ \int_{\M \times \R^d} [\beta_T]_+ \zeta_T \dd x \dd v .
\end{multline}
Finally, taking a sequence of vector fields $a$ converging in $L^{r'}((0,T) \times \M \times \R^d)$ to $-D_p H(x,v, D_v u_+)$, we conclude that
\begin{multline}
\int_0^{T} \int_{\M \times \R^d} u_+ (\partial_t \zeta + v\cdot D_x \zeta) \dd x \dd v \dd t \leq \int_0^{T} \int_{\M \times \R^d} \left [\beta - H(x,v, D_v u_+) \right ] \one_{\{u > 0\}} \zeta \dd x \dd v \dd t\\
+ \int_{\M \times \R^d} [\beta_T]_+ \zeta_T \dd x \dd v ,
\end{multline}
that is, the following holds in the sense of distributions:
\begin{align} \label{eq:HJ-positive}
-\partial_t u_+ - v\cdot D_x u_+ +H(x,v,D_v u_+) \one_{\{ u > 0\}} \leq \beta \one_{\{ u > 0\}}, \quad &{\rm{in}\ } \sD'((0,T)\times\M\times\R^d). \\
[u_+]_T \leq (\beta_T)_+ \quad &{\rm{in\ }} \sD'(\M\times\R^d).
\end{align}
\end{proof}

\medskip

\noindent {\it Acknowledgements:} We thank Mikaela Iacobelli for useful discussions regarding Lemma~\ref{lem:finite_or_infinity}. We thank the two anonymous referees for carefully reading our manuscript and for their valuable comments. 

\medskip

\bibliographystyle{abbrv}
\bibliography{MFGbib}

\begin{thebibliography}{10}

\bibitem{AchCarDelPorSan}
Y.~Achdou, P.~Cardaliaguet, F.~Delarue, A.~Porretta, and F.~Santambrogio.
\newblock {\em Mean field games}.
\newblock Lecture Notes in Mathematics, C.I.M.E. Foundation Subseries, Vol.
  2281, Springer, 2020.

\bibitem{AchKob:21}
Y.~Achdou and Z.~Kobeissi.
\newblock Mean field games of controls: finite difference approximations.
\newblock {\em Math. Eng.}, 3(3):Paper No. 024, 35, (2021).

\bibitem{AchManMarTch:20}
Y.~Achdou, P.~Mannucci, C.~Marchi, and N.~Tchou.
\newblock Deterministic mean field games with control on the acceleration.
\newblock {\em NoDEA Nonlinear Differential Equations Appl.}, 27(3):Paper No.
  33, 32, (2020).

\bibitem{AchManMarTch:21}
Y.~Achdou, P.~Mannucci, C.~Marchi, and N.~Tchou.
\newblock Deterministic mean field games with control on the acceleration and
  state constraints.
\newblock {\em arXiv:2104.07292}, (2021).

\bibitem{Amb:18}
D.~M. Ambrose.
\newblock Strong solutions for time-dependent mean field games with
  non-separable {H}amiltonians.
\newblock {\em J. Math. Pures Appl. (9)}, 113:141--154, (2018).

\bibitem{Amb:21}
D.~M. Ambrose.
\newblock Existence theory for non-separable mean field games in {Sobolev}
  spaces.
\newblock {\em Indiana U. Math. J.}, to appear.

\bibitem{AmbGigSav}
L.~Ambrosio, N.~Gigli, and G.~Savar\'e.
\newblock {\em Gradient flows in metric spaces and in the space of probability
  measures.}
\newblock Second edition. {L}ecture notes in {M}athematics {E}{T}{H}
  Z{\"u}rich. {B}irkh{\"a}user Verlag, Basel, 2008.

\bibitem{BarCar}
M.~Bardi and P.~Cardaliaguet.
\newblock Convergence of some mean field games systems to aggregation and
  flocking models.
\newblock {\em Nonlinear Anal.}, 204:Paper No. 112199, 24, (2021).

\bibitem{BouButSep}
G.~Bouchitt\'e, G.~Buttazzo, and P.~Seppecher.
\newblock Energies with respect to a measure and applications to
  low-dimensional structures.
\newblock {\em Calc. Var. Partial Differential Equations}, 5(1):37--54, (1997).

\bibitem{Bouchut}
F.~Bouchut.
\newblock Renormalized solutions to the {V}lasov equation with coefficients of
  bounded variation.
\newblock {\em Arch. Ration. Mech. Anal.}, 157(1):75--90, (2001).

\bibitem{Bouchut2002}
F.~Bouchut.
\newblock Hypoelliptic regularity in kinetic equations.
\newblock {\em J. Math. Pures Appl. (9)}, 81(11):1135--1159, (2002).

\bibitem{CanMen}
P.~Cannarsa and C.~Mendico.
\newblock Mild and weak solutions of mean field game problems for linear
  control systems.
\newblock {\em Minimax Theory Appl.}, 5(2):221--250, (2020).

\bibitem{Cardaliaguet2013}
P.~Cardaliaguet.
\newblock Weak solutions for first order mean field games with local coupling.
\newblock In {\em Analysis and geometry in control theory and its
  applications}, volume~11 of {\em Springer INdAM Ser.}, pages 111--158.
  Springer, Cham, (2015).

\bibitem{Cardaliaguet-Graber}
P.~Cardaliaguet and P.~Graber.
\newblock Mean field games systems of first order.
\newblock {\em ESAIM Control Optim. Calc. Var.}, 21(3):690--722, (2015).

\bibitem{CarGraPorTon}
P.~Cardaliaguet, P.~Graber, A.~Porretta, and D.~Tonon.
\newblock Second order mean field games with degenerate diffusion and local
  coupling.
\newblock {\em NoDEA Nonlinear Differential Equations Appl.}, 22(5):1287--1317,
  (2015).

\bibitem{CarMen}
P.~Cardaliaguet and C.~Mendico.
\newblock Ergodic behavior of control and mean field games problems depending
  on acceleration.
\newblock {\em Nonlinear Anal.}, 203:Paper No. 112185, 40, (2021).

\bibitem{CarMesSan}
P.~Cardaliaguet, A.~M\'esz\'aros, and F.~Santambrogio.
\newblock First order {M}ean {F}ield {G}ames with density constraints: pressure
  equals price.
\newblock {\em SIAM J. Control Optim.}, 54(5):2672--2709, (2016).

\bibitem{CarDel:13}
R.~Carmona and F.~Delarue.
\newblock Probabilistic analysis of mean-field games.
\newblock {\em SIAM J. Control Optim.}, 51(4):2705--2734, (2013).

\bibitem{CarDel:vol1}
R.~Carmona and F.~Delarue.
\newblock {\em Probabilistic theory of mean field games with applications.
  {I}}, volume~83 of {\em Probability Theory and Stochastic Modelling}.
\newblock Springer, Cham, (2018).
\newblock Mean field FBSDEs, control, and games.

\bibitem{CarDel:vol2}
R.~Carmona and F.~Delarue.
\newblock {\em Probabilistic theory of mean field games with applications.
  {II}}, volume~84 of {\em Probability Theory and Stochastic Modelling}.
\newblock Springer, Cham, 2018.
\newblock Mean field games with common noise and master equations.

\bibitem{CarDelLac}
R.~Carmona, F.~Delarue, and D.~Lacker.
\newblock Mean field games with common noise.
\newblock {\em Ann. Probab.}, 44(6):3740--3803, (2016).

\bibitem{CirGof:20}
M.~Cirant and A.~Goffi.
\newblock Lipschitz regularity for viscous {H}amilton-{J}acobi equations with
  {$L^p$} terms.
\newblock {\em Ann. Inst. H. Poincar\'{e} Anal. Non Lin\'{e}aire},
  37(4):757--784, (2020).

\bibitem{CirGof:21}
M.~Cirant and A.~Goffi.
\newblock Maximal {$L^q$}-regularity for parabolic {H}amilton-{J}acobi
  equations and applications to mean field games.
\newblock {\em Ann. PDE}, 7(2):Paper No. 19, 40, (2021).

\bibitem{DiPerna-Lions}
R.~J. DiPerna and P.-L. Lions.
\newblock Ordinary differential equations, transport theory and {S}obolev
  spaces.
\newblock {\em Invent. Math.}, 98(3):511--547, (1989).

\bibitem{DraFel}
F.~Dragoni and E.~Feleqi.
\newblock Ergodic mean field games with {H}\"{o}rmander diffusions.
\newblock {\em Calc. Var. Partial Differential Equations}, 57(5):Paper No. 116,
  22, (2018).

\bibitem{Dunford-Schwartz}
N.~Dunford and J.~T. Schwartz.
\newblock {\em Linear Operators, Part 1: General Theory}.
\newblock Wiley Classics Library. John Wiley, 1988.

\bibitem{Ekeland-Temam}
I.~Ekeland and R.~T{\'{e}}mam.
\newblock {\em Convex analysis and variational problems}, volume~28 of {\em
  Classics in Applied Mathematics}.
\newblock Society for Industrial and Applied Mathematics (SIAM), Philadelphia,
  PA, 1999.

\bibitem{FelGomTad}
E.~Feleqi, D.~Gomes, and T.~Tada.
\newblock Hypoelliptic mean field games---a case study.
\newblock {\em Minimax Theory Appl.}, 5(2):305--326, (2020).

\bibitem{GLPS}
F.~Golse, P.-L. Lions, B.~Perthame, and R.~Sentis.
\newblock Regularity of the moments of the solution of a transport equation.
\newblock {\em J. Funct. Anal.}, 76(1):110--125, (1988).

\bibitem{GPS}
F.~Golse, B.~Perthame, and R.~Sentis.
\newblock Un r\'{e}sultat de compacit\'{e} pour les \'{e}quations de transport
  et application au calcul de la limite de la valeur propre principale d'un
  op\'{e}rateur de transport.
\newblock {\em C. R. Acad. Sci. Paris S\'{e}r. I Math.}, 301(7):341--344,
  (1985).

\bibitem{GSR}
F.~Golse and L.~Saint-Raymond.
\newblock Velocity averaging in {L}1 for the transport equation.
\newblock {\em C. R. Acad. Sci. Paris, Ser. 1}, 334:557--562, (2002).

\bibitem{GSR_NS}
F.~Golse and L.~Saint-Raymond.
\newblock The {N}avier--{S}tokes limit of the {B}oltzmann equation for bounded
  collision kernels.
\newblock {\em Invent. Math.}, 155(1):81--161, (2004).

\bibitem{GomPimSan16}
D.~A. Gomes, E.~Pimentel, and H.~S{\'a}nchez-Morgado.
\newblock Time-dependent mean-field games in the superquadratic case.
\newblock {\em ESAIM Control Optim. Calc. Var.}, 22(2):562--580, (2016).

\bibitem{GomPimVos}
D.~A. Gomes, E.~Pimentel, and V.~Voskanyan.
\newblock {\em Regularity theory for mean-field game systems}.
\newblock SpringerBriefs in Mathematics. Springer, [Cham], (2016).

\bibitem{GomPimSan15}
D.~A. Gomes, E.~A. Pimentel, and H.~S{\'a}nchez-Morgado.
\newblock Time-dependent mean-field games in the subquadratic case.
\newblock {\em Comm. Partial Differential Equations}, 40(1):40--76, (2015).

\bibitem{GomVos}
D.~A. Gomes and V.~K. Voskanyan.
\newblock Extended deterministic mean-field games.
\newblock {\em SIAM J. Control Optim.}, 54(2):1030--1055, (2016).

\bibitem{Gra:14}
P.~J. Graber.
\newblock Optimal control of first-order {H}amilton-{J}acobi equations with
  linearly bounded {H}amiltonian.
\newblock {\em Appl. Math. Optim.}, 70(2):185--224, (2014).

\bibitem{GraMes}
P.~J. Graber and A.~R. M\'esz\'aros.
\newblock Sobolev regularity for first order mean field games.
\newblock {\em Ann. Inst. H. Poincar\'e Anal. Non Lin\'eaire}, 35:1557--1576,
  (2018).

\bibitem{GraMesSilTon}
P.~J. Graber, A.~R. M\'{e}sz\'{a}ros, F.~J. Silva, and D.~Tonon.
\newblock The planning problem in mean field games as regularized mass
  transport.
\newblock {\em Calc. Var. Partial Differential Equations}, 58(3):Paper No. 115,
  28, (2019).

\bibitem{GraMulPfe}
P.~J. Graber, A.~Mullenix, and L.~Pfeiffer.
\newblock Weak solutions for potential mean field games of controls.
\newblock {\em NoDEA Nonlinear Differential Equations Appl.}, 28(5):Paper No.
  50, 34, (2021).

\bibitem{Han-Kwan}
D.~Han-Kwan.
\newblock ${L}^1$ averaging lemma for transport equations with {L}ipschitz
  force fields.
\newblock {\em Kinet. Relat. Models}, 3(4):669--683, (2010).

\bibitem{HuaMalCai}
M.~Huang, R.~P. Malham{\'e}, and P.~E. Caines.
\newblock Large population stochastic dynamic games: closed-loop
  {M}c{K}ean-{V}lasov systems and the {N}ash certainty equivalence principle.
\newblock {\em Commun. Inf. Syst.}, 6(3):221--251, (2006).

\bibitem{Jabin}
P.-E. Jabin.
\newblock Averaging lemmas and dispersion estimates for kinetic equations.
\newblock {\em Riv. Mat. Univ. Parma (8)}, 1:71--138, (2009).

\bibitem{Kob}
Z.~Kobeissi.
\newblock On classical solutions to the mean field game system of controls.
\newblock {\em Comm. Partial Differential Equations}, to appear, (2021).

\bibitem{Kosygina-Varadhan}
E.~Kosygina and S.~R.~S. Varadhan.
\newblock Homogenization of {H}amilton-{J}acobi-{B}ellman equations with
  respect to time-space shifts in a stationary ergodic medium.
\newblock {\em Comm. Pure Appl. Math.}, 61(6):816--847, (2008).

\bibitem{LasLio06i}
J.-M. Lasry and P.-L. Lions.
\newblock Jeux \`a champ moyen {I}. {L}e cas stationnaire.
\newblock {\em C. R. Math. Acad. Sci. Paris}, 343:619--625, (2006).

\bibitem{LasLio06ii}
J.-M. Lasry and P.-L. Lions.
\newblock Jeux \`a champ moyen {II}. {H}orizon fini et contr\^ole optimal.
\newblock {\em C. R. Math. Acad. Sci. Paris}, 343:679--684, (2006).

\bibitem{LasLio07}
J.-M. Lasry and P.-L. Lions.
\newblock Mean field games.
\newblock {\em Jpn. J. Math.}, 2:229--260, (2007).

\bibitem{Lions-course}
P.-L. Lions.
\newblock \textit{Cours au {C}oll\`ege de {F}rance}.
\newblock {\em {\rm www.college-de-france.fr}}, (2007-2011).

\bibitem{Men}
C.~Mendico.
\newblock Singular perturbation problem of mean field game of acceleration.
\newblock {\em arXiv:2107.08479}, (2021).

\bibitem{MesMou}
A.~M\'esz\'aros and C.~Mou.
\newblock Mean field games systems under displacement monotonicity.
\newblock {\em arXiv:2109.06687}, (2021).

\bibitem{Mimikos-Stamatopoulos}
N.~Mimikos-Stamatopoulos.
\newblock Weak and renormalized solutions to a hypoelliptic mean field games
  system.
\newblock {\em arXiv:2105.05777}, (2021).

\bibitem{Mun:1}
S.~Munoz.
\newblock Classical and weak solutions to local first order mean field games
  through elliptic regularity.
\newblock {\em arXiv:2006.07367}, (2020).

\bibitem{Mun:2}
S.~Munoz.
\newblock Classical solutions to local first order extended mean field games.
\newblock {\em arXiv:2102.13093}, (2021).

\bibitem{NouCaiMal}
M.~Nourian, P.~E. Caines, and R.~P. Malham\'e.
\newblock Mean field analysis of controlled {Cucker-Smale} type flocking:
  Linear analysis and perturbation equations.
\newblock {\em Proceedings of the 18th IFAC WC, Milan, Italy}, pages
  4471--4476, (2011).

\bibitem{OrrPorSav}
C.~Orrieri, A.~Porretta, and G.~Savar\'{e}.
\newblock A variational approach to the mean field planning problem.
\newblock {\em J. Funct. Anal.}, 277(6):1868--1957, (2019).

\bibitem{Por15}
A.~Porretta.
\newblock Weak solutions to {F}okker-{P}lanck equations and mean field games.
\newblock {\em Arch. Ration. Mech. Anal.}, 216(1):1--62, (2015).

\bibitem{ProSan}
A.~Prosinski and F.~Santambrogio.
\newblock Global-in-time regularity via duality for congestion-penalized mean
  field games.
\newblock {\em Stochastics}, 89(6-7):923--942, (2017).

\bibitem{San}
F.~Santambrogio.
\newblock Regularity via duality in calculus of variations and degenerate
  elliptic {PDE}s.
\newblock {\em J. Math. Anal. Appl.}, 457(2):1649--1674, (2018).

\bibitem{SanShi}
F.~Santambrogio and W.~Shim.
\newblock A {C}ucker-{S}male inspired deterministic mean field game with
  velocity interactions.
\newblock {\em SIAM J. Control Optim.}, 59(6):4155--4187, (2021).

\bibitem{Sta}
G.~Stampacchia.
\newblock Le probl\`eme de {D}irichlet pour les \'{e}quations elliptiques du
  second ordre \`a coefficients discontinus.
\newblock {\em Ann. Inst. Fourier (Grenoble)}, 15(fasc. 1):189--258, (1965).

\end{thebibliography}

\end{document}